\documentclass[a4paper,11pt,reqno]{amsart}
\usepackage{amsmath,amssymb,amsthm,graphicx,verbatim}

\theoremstyle{plain}
\newtheorem{theo}{Theorem}[section]
\newtheorem{lem}[theo]{Lemma}

\newtheorem{cor}[theo]{Corollary}
\newtheorem{prop}[theo]{Proposition}

\theoremstyle{definition}
\newtheorem{defi}[theo]{Definition}
\newtheorem{rem}[theo]{Remark}

\numberwithin{equation}{section}

\newcommand{\BottomEven}{\mathcal{B}^N_e}
\newcommand{\LeftOdd}{\mathcal{L}^N_o}
\newcommand{\RightOdd}{\mathcal{R}^N_o}

\newcommand{\Oddr}{\mathcal{O}_R}
\newcommand{\Oddl}{\mathcal{O}_L}
\newcommand{\Oddd}{\mathcal{O}_D}
\newcommand{\Oddu}{\mathcal{O}_U}
\newcommand{\oddr}{o_R}
\newcommand{\oddl}{o_L}
\newcommand{\oddd}{o_D}
\newcommand{\oddu}{o_U}

\newcommand{\Evenr}{\mathcal{E}_R}
\newcommand{\Evenl}{\mathcal{E}_L}
\newcommand{\Evend}{\mathcal{E}_D}
\newcommand{\Evenu}{\mathcal{E}_U}
\newcommand{\evenr}{e_R}
\newcommand{\evenl}{e_L}
\newcommand{\evend}{e_D}
\newcommand{\evenu}{e_U}

\newcommand{\oddleft}{\scalebox{0.5}{\includegraphics{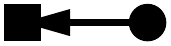}}}
\newcommand{\odddown}{\raisebox{-0.25cm}{\scalebox{0.5}{\includegraphics{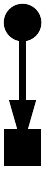}}}}

\newcommand{\evenleft}{\scalebox{0.5}{\includegraphics{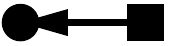}}}

\newcommand{\evenup}{\raisebox{-0.25cm}{\scalebox{0.5}{\includegraphics{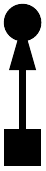}}}}
\newcommand{\oddleftdown}{\raisebox{-0.25cm}{\scalebox{0.5}{\includegraphics{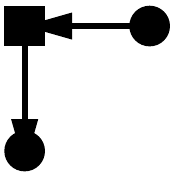}}}}
\newcommand{\odddownleft}{\raisebox{-0.25cm}{\scalebox{0.5}{\includegraphics{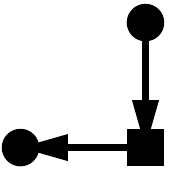}}}}
\newcommand{\oddupleft}{\raisebox{-0.25cm}{\scalebox{0.5}{\includegraphics{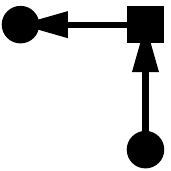}}}}
\newcommand{\oddleftup}{\raisebox{-0.25cm}{\scalebox{0.5}{\includegraphics{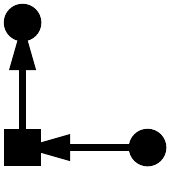}}}}
\newcommand{\evenleftup}{\raisebox{-0.25cm}{\scalebox{0.5}{\includegraphics{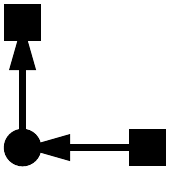}}}}
\newcommand{\evenupleft}{\raisebox{-0.25cm}{\scalebox{0.5}{\includegraphics{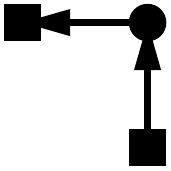}}}}
\newcommand{\evendownleft}{\raisebox{-0.25cm}{\scalebox{0.5}{\includegraphics{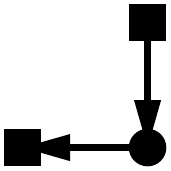}}}}
\newcommand{\evenleftdown}{\raisebox{-0.25cm}{\scalebox{0.5}{\includegraphics{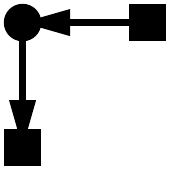}}}}
\newcommand{\oddleftleft}{\scalebox{0.5}{\includegraphics{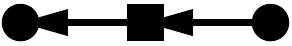}}}
\newcommand{\evenleftleft}{\scalebox{0.5}{\includegraphics{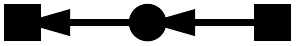}}}

\newcommand{\dlo}{\raisebox{-0.15cm}{\scalebox{0.7}{\includegraphics{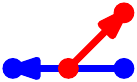}}}}
\newcommand{\dle}{\raisebox{-0.15cm}{\scalebox{0.7}{\includegraphics{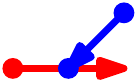}}}}
\newcommand{\luo}{\raisebox{-0.15cm}{\scalebox{0.7}{\includegraphics{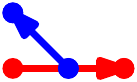}}}}
\newcommand{\lue}{\raisebox{-0.15cm}{\scalebox{0.7}{\includegraphics{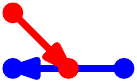}}}}
\newcommand{\ldo}{\raisebox{-0.15cm}{\scalebox{0.7}{\includegraphics{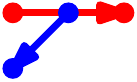}}}}
\newcommand{\lde}{\raisebox{-0.15cm}{\scalebox{0.7}{\includegraphics{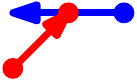}}}}
\newcommand{\ulo}{\raisebox{-0.15cm}{\scalebox{0.7}{\includegraphics{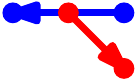}}}}
\newcommand{\ule}{\raisebox{-0.15cm}{\scalebox{0.7}{\includegraphics{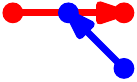}}}}
\newcommand{\db}{\raisebox{-0.15cm}{\scalebox{0.7}{\includegraphics{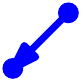}}}}
\newcommand{\dr}{\raisebox{-0.15cm}{\scalebox{0.7}{\includegraphics{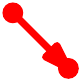}}}}

\newcommand{\evenll}{\scalebox{0.7}{\includegraphics{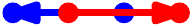}}}
\newcommand{\oddll}{\scalebox{0.7}{\includegraphics{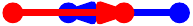}}}
\newcommand{\blueh}{\scalebox{0.7}{\includegraphics{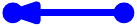}}}
\newcommand{\redh}{\scalebox{0.7}{\includegraphics{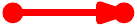}}}

\renewcommand{\circlearrowleft}{{\scalebox{0.05}{\includegraphics{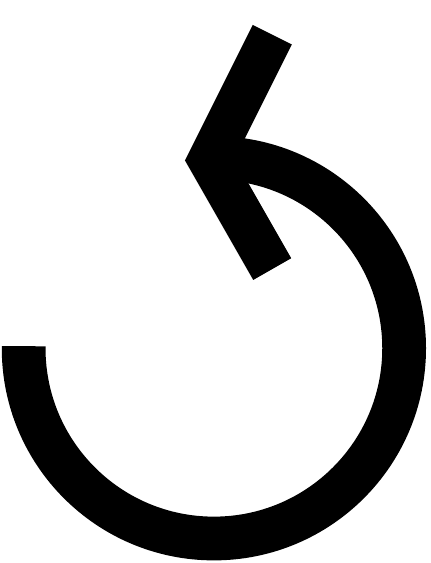}}}}
\renewcommand{\circlearrowright}{{\scalebox{0.05}{\includegraphics{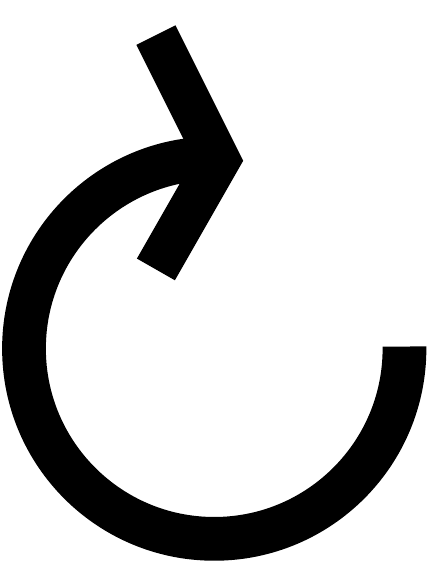}}}}

\renewcommand{\d}{\operatorname{d}}

\newcommand{\mr}{\operatorname{MoveR}}
\newcommand{\ml}{\operatorname{MoveL}}
\newcommand{\p}{\operatorname{Path}}
\newcommand{\lh}{\operatorname{HeightL}}
\newcommand{\rh}{\operatorname{HeightR}}
\newcommand{\lp}{\operatorname{PuzzleL}}
\newcommand{\rp}{\operatorname{PuzzleR}}
\newcommand{\exc}{\operatorname{exc}}

\newcommand{\TG}{\mathcal{T}} % The triangular grid for puzzles

\newcommand{\Pathconfig}{BlueRed} % The set of blue-red path configurations

\newcommand{\ssyt}{\operatorname{SSYT}}
\newcommand{\Dn}{\mathcal{D}_n}

\begin{document}

\title[TFPLs and matchings]{Fully Packed Loops in a triangle: matchings, paths and puzzles}

\author{Ilse Fischer}
\address{Ilse Fischer, Universit\"at Wien, Fakult\"at f\"ur Mathematik, Nordbergstrasse 15, 1090 Wien, Austria}
\email{ilse.fischer@univie.ac.at}
\author{Philippe Nadeau}
\address{Philippe Nadeau, CNRS, Institut Camille Jordan, Universit\'e Lyon 1, 43 boulevard du 11 novembre 1918, 69622 Villeurbanne cedex, France }
\email{nadeau@math.univ-lyon1.fr}
\thanks{Supported by the Austrian Science Foundation FWF, START grant Y463 and NFN grant S9607-N13}

\begin{abstract}
 Fully Packed Loop configurations in a triangle (TFPLs) first appeared in the study of ordinary Fully Packed Loop configurations (FPLs) on the square grid where they were used to show that the number of FPLs with  a given link pattern that has $m$ nested arches is a polynomial function in $m$. It soon turned out that TFPLs possess a number of other nice properties. For instance, they can be seen as a generalized model of Littlewood-Richardson coefficients. We start our article by introducing oriented versions 
 of TFPLs; their main advantage in comparison with ordinary TFPLs is that they involve only local constraints. Three main contributions are provided. Firstly, we show that the number of ordinary TFPLs can be extracted from a weighted enumeration of oriented TFPLs and thus it suffices to consider the latter. Secondly, we decompose oriented TFPLs into two matchings and use a classical bijection to obtain two families of nonintersecting  lattice paths (path tangles). This point of view turns out to be extremely useful for giving easy proofs of previously known conditions on the boundary of TFPLs  necessary for them to exist. One example is the inequality $\d(u) + \d(v) \le \d(w)$ where $u,v,w$ are $01$-words that encode the boundary conditions of ordinary TFPLs and $\d(u)$ is the number of cells in the Ferrers diagram associated with $u$. In the third part we consider TFPLs with $\d(w) - \d(u) - \d(v) = 0,1$; in the first case their numbers are given by Littlewood-Richardson coefficients, but also in the second case we provide formulas that are in terms of Littlewood-Richardson coefficients.  The proofs of these formulas are of a purely combinatorial nature.
\end{abstract}

\maketitle

%\tableofcontents

%%%%%%%%%%%%%%%%%%%%%%%%%%%%%%%%%%%%%%%%%%%%%%%%%%%%%%%%%%%%%
%%%%%%%%%%%%%%%%%%%%%%%%%%%%%%%%%%%%%%%%%%%%%%%%%%%%%%%%%%%%%

\section*{Introduction}

Fully Packed Loop configurations (FPLs) are subgraphs of a finite square grid such that each 
internal vertex has degree $2$, and the boundary conditions are alternating, see Figure~\ref{fig:intro}. These objects made their first appearance in statistical mechanics, and were later 
realized to be in bijection with the famous Alternating Sign Matrices, as well as numerous other structures; cf.~\cite{Propp_ASM}. In particular, the total number $A_n$ of FPLs on a grid with $n^2$ 
vertices is given by the famous formula first proven by Zeilberger; the story of this problem is told in the book by Bressoud~\cite{Bressoud}.

What distinguishes FPL configurations from other structures in bijection are the {\em paths} that join two of the external edges. Therefore each FPL configuration is associated with a {\em link pattern}, which encodes the pairs of endpoints that are joined by paths. The quantity of interest then becomes: given a link pattern $\pi$, how many FPL configurations have associated pattern $\pi$? These integers $A_\pi$ attracted some interest from mathematicians thanks to the Razumov-Stroganov (ex-)conjecture~\cite{RS-conj}, which says that there exists a simple Markov chain on the set of link patterns of size $n$ such that its stationary distribution $(\psi_\pi)_{\pi}$ is given by $\left(A_\pi/A_n\right)_{\pi}$. This was proven by Cantini and Sportiello in 2010~\cite{ProofRS}. The proof proceeds by pretty combinatorial arguments, making heavy use of \emph{Wieland gyration}~\cite{Wieland} which is a particular operation on FPLs that was originally invented to prove a certain  
rotational invariance of the numbers $A_{\pi}$. 

%The proof does not compute the values of the numbers $A_\pi$.
 Another line of research was developed in the works of Di Francesco and Zinn-Justin, see~\cite{ZJ-hdr} and the articles cited in there. They relate the numbers $\psi_\pi$ to certain quantities 
$\phi_{\alpha}$ through a change of basis, and give formulas for the latter in terms of 
multiple integrals.
%They obtained expressions for the numbers $\psi_\pi$, consisting of certain multiple integrals defining certain  quantities $\phi_\alpha$ which are related to the $\psi_\pi$ through a change of basis. 
Thanks to the result of~\cite{ProofRS}, these become formulas for the numbers $A_\pi$ themselves. The drawback of these formulas is that they lack combinatorial interpretations. That is, one would like to understand FPLs well enough to be able to write formulas directly from a combinatorial analysis, which reflects 
 their structure. 

\begin{figure}[!ht]
 \begin{center}
\includegraphics[height=4cm]{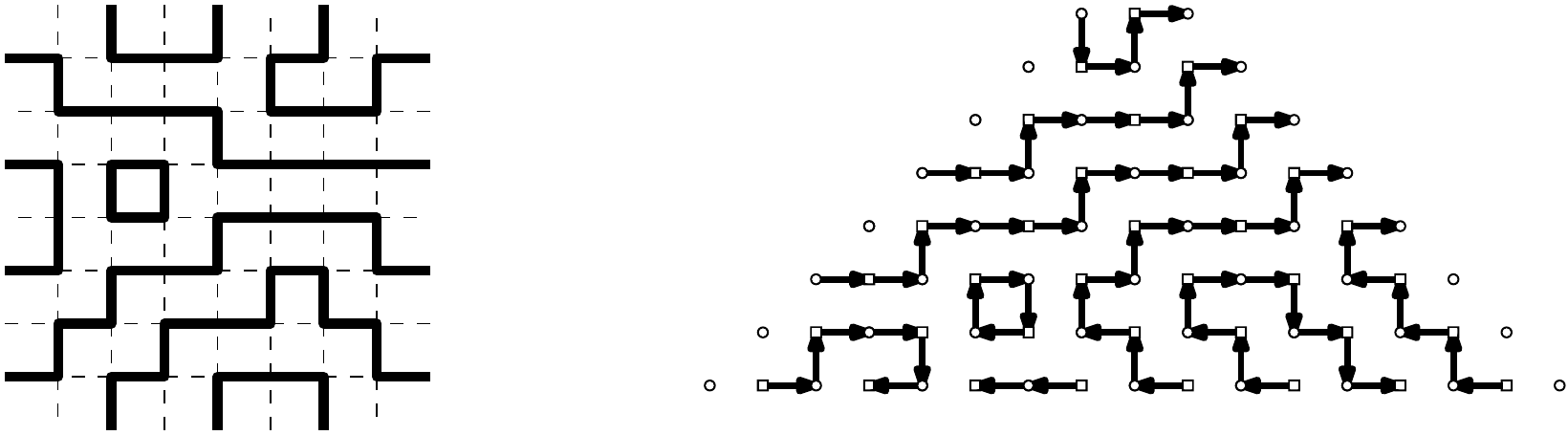}
\caption{Examples of a FPL configuration (left) and an oriented TFPL configuration \label{fig:intro}.}  
 \end{center}
\end{figure}

From the combinatorial point of view, the study of FPLs was advertised in the paper~\cite{Zuber-conj}, which gathers several conjectures around the enumeration of FPL configurations. For instance, it deals with FPLs refined according to link patterns with $m$ ``nested arches'', written $\pi\cup m$: the conjecture is that $m\mapsto A_{\pi\cup m}$ is a polynomial function with an explicit dominant term. This was proven in~\cite{CKLN}. Experimentally, these polynomials appear to have a number of surprising properties, notably regarding their roots: several conjectures were made in the work of the second author with Fonseca~\cite{Fonseca_Nadeau}.
% This decomposition leads to a polynomial in $m$ which turns out to give the correct value for $A_{\pi\cup m}$ even for small values of $m$ in which the combinatorial decomposition is not valid any more~\cite{CKLN}

For $m$ big enough, FPLs with link pattern $\pi\cup m$ admit a combinatorial decomposition in which \emph{Fully Packed Loops in a triangle} (TFPLs) naturally arise (it has in particular the enumerative consequence summarized later in the expression~\eqref{eq:tfpl_to_fpl}). These TFPLs as well as \emph{oriented} versions of them (see Definition~\ref{defi:otfpl}) are the focus of the present article, their precise definition being given in Section~\ref{sub:tfpls}. An oriented example is displayed in Figure~\ref{fig:intro}, right. Here we shall just say that the internal vertices have degree $2$; only on the three boundaries of the triangle there are vertices of degree $0$ or $1$. TFPLs are enumerated according to a refinement encoded by three binary words $u,v,w$ of the same length $N$, so that we speak of TFPLs with boundary $(u,v;w)$. 
%The most important statistic for words here is the number of inversions in a word $u$, which we will denote by $\d(u)$.
\medskip
 
Aside from their appearance in the expression~\eqref{eq:tfpl_to_fpl} for FPLs, there are at least two other motivations to study TFPLs: The first one is that they allow us to compute nice coefficients $c_{\pi'\pi}$ in the following polynomial identity:
\[
 A_{\pi\cup (m+1)}=\sum_{\pi'} c_{\pi'\pi} A_{\pi'\cup m}\qquad\text{for all }m\geq 0.
 \]
Here $\pi,\pi'$ are link patterns of size $n$. These relations were conjectured in~\cite{Thapper} and proven in~\cite{TFPL1}; notice that they are different from the relations in the Razumov-Stroganov (ex-)conjecture, because they involve FPLs on two different grid sizes.

The second motivation comes {\em a posteriori}, in the sense that, when one delves into the structure of TFPLs, some intriguing and beautiful combinatorics come up naturally. This was already the case in the two articles just mentioned, where in particular a strong link with semistandard tableaux is established. It was then shown in~\cite{TFPL2} that, under the constraint $\d(w)=\d(u)+\d(v)$, TFPLs with boundary $(u,v;w)$ are enumerated by \emph{Littlewood--Richardson coefficients}; here $\d(\cdot)$ is the number of inversions of a binary word.
The proof in~\cite{TFPL2} proceeds by a convoluted argument. In Section~\ref{sec:excess_0} of the present paper, this result will be proven again, in a straightforward way, and in a more general setting. 
%We shall use \emph{oriented} TFPLs  which will be in fact our central object of study in this work: they are oriented versions of ordinary TFPLs (see Definition~\ref{defi:otfpl}), and we shall explain their relation to TFPLs. Also we will show how to encode them via a certain configurations of paths, which we chose to call \emph{path tangles} given the numerous intersections that occur. 
\medskip

We distinguish three main contributions in this article: 
\begin{enumerate}
\item A study of the connections between TFPLs and oriented TFPLs, which results in Corollary~\ref{cor:otfpls_to_tfpls} relating the two enumerations;
\item  A combinatorial interpretation of the quantity $\d(w)-\d(u)-\d(v)$ in oriented TFPLs (Theorem~\ref{theo:excessformula}). This makes use of a new path model for oriented TFPLs which we call \emph{path tangles};
\item
 An enumeration of TFPLs (oriented or not) for boundary conditions $(u,v;w)$ verifying  $\d(w)-\d(u)-\d(v)=0$ or $1$ (Theorems~\ref{theo:LR} and~\ref{theo:excess_1}). These involve the study of certain~\emph{puzzles} which extend the puzzles of Knutson and Tao~\cite{KT}.
\end{enumerate}

The article can be divided roughly into three parts: Sections~\ref{sec:definitions} and~\ref{sec:TFPLs} deal with the relations between ordinary and oriented TFPLs, Sections~\ref{sec:matchings_and_paths} and~\ref{sec:blue_red} with the structure of oriented TFPLs and path tangles, and Sections~\ref{sec:excess_0} and ~\ref{sec:excess_1} with the enumeration of TFPLs, ordinary and oriented, in some special cases. We now detail the content of each of these parts.

\subsection*{TFPLs and oriented TFPLs}

The interplay between ordinary TFPLs and oriented TFPLs is the focus of Sections~\ref{sec:definitions} and~\ref{sec:TFPLs}. We first define both objects in slightly greater generality than 
in previous works: both kinds of TFPLs were used indeed in~\cite{TFPL1,TFPL2,Thapper,ZJ-triangle}, but with a specific restriction on the boundary words $u,v,w$ and a specific orientation of the paths. The obvious reason for this is that this restriction on the boundary holds in Formula~\eqref{eq:tfpl_to_fpl} expressing FPLs in terms of TFPLs, and that, with the specific orientation, ordinary TFPLs and oriented TFPLs are in fact equivalent. So why bother generalizing? First, such generalized TFPLs occurred already (although they were note explicitly defined) in~\cite[Section 5]{TFPL1}. Second, when studying TFPLs directly as we will do in this paper, it soon appears that keeping the extra constraints is not useful. Although it makes the definition of ordinary TFPLs a bit more involved since we need to extend the definition of patterns, the definition of oriented TFPLs becomes more natural.

Now given any three words $u,v,w$ of the same length, the set of TFPLs with boundary $(u,v;w)$ can be naturally embedded into the set of oriented TFPLs with the same boundary conditions. In particular, if there is no oriented TFPL with boundary $(u,v;w)$, then there can be no such ordinary TFPL either. Constraints on the boundaries of TFPLs are summarized in Theorem~\ref{theo:tfpl_necessary_conditions} and its corollary. The upshot is that oriented TFPLs with boundary $u,v;w$ are easier to manipulate than TFPLs, essentially because their definition involves only local constraints.  But even more so, we can actually derive the enumeration of TFPLs from a certain weighted enumeration of oriented TFPLs: this is the content of Section~\ref{sec:TFPLs}, the final result being the second equation in Corollary~\ref{cor:otfpls_to_tfpls}.

As a consequence oriented TFPLs are the primary object of study in the rest of the paper.

\subsection*{Oriented TFPLs, Matchings and  Path Tangles}

The second contribution of this paper is the introduction of certain \emph{path tangles} which encode oriented TFPLs nicely, and help prove a combinatorial interpretation of the quantity $\d(w)-\d(u)-\d(v)$, see Theorem~\ref{theo:excessformula}. This is the focus of Sections~\ref{sec:matchings_and_paths} and~\ref{sec:blue_red}.The key idea is to split an oriented TFPL into two perfect matchings (Theorem~\ref{theo:tfpl_to_matchings}); we can then proceed to study individually each matching, which allows already to prove the first two statements of Theorem~\ref{theo:tfpl_necessary_conditions}. Each matching can be itself encoded as a configuration of nonintersecting lattice paths. 

When the path configurations from the two matchings are reunited, one obtains what we chose to call a path tangle: indeed here the paths intersect in general. The possible ways that they may cross is constrained by the fact that the two perfect matchings are disjoint. This gives a bijection between oriented TFPLs with a given boundary and path tangles with prescribed departure and arrival points for each path: this is Theorem~\ref{theo:path_model_characterization}. The study of the ways the paths intersect in a path tangle leads to the explicit formula of Theorem~\ref{theo:excessformula}, which gives a combinatorial interpretation of the quantity $\d(w)-\d(u)-\d(v)$: it is the number of occurrences of certain local patterns in any TFPL with boundary $(u,v;w)$. In particular, it shows that such TFPLs cannot exist if $\d(w)-\d(u)-\d(v)<0$, which completes the proof of Theorem~\ref{theo:tfpl_necessary_conditions}.

\subsection*{Enumeration of TFPLs when $\d(w)-\d(u)-\d(v)=0$ or $1$}

The starting point here is Theorem~\ref{theo:excessformula}, which says that the quantity $\d(w)-\d(u)-\d(v)$ (which we call the \emph{excess}) is equal to the number of occurrences of various local patterns in any oriented TFPL with boundary $(u,v;w)$. It is therefore natural to look first at the special cases of small excess.

The case where the excess is $0$ was dealt with first in~\cite{TFPL2}: it was proven there that, in the case of ordinary TFPLs and when $w$ is restricted to be essentially a Dyck word, the number of TFPLs with boundary $(u,v;w)$ 
is given by the \emph{Littlewood--Richardson coefficient} $c_{\lambda(u),\lambda(v)}^{\lambda(w)}$, where $\lambda(\cdot)$ is a natural way to associate an integer partition with a word. Here we prove this result again in Section~\ref{sec:excess_0}, removing the restriction on $w$ and extending it to oriented TFPLs as well. The bijection at the core of the proof is the same as in~\cite{TFPL2}, i.e. a map to Knutson--Tao puzzles~\cite{KT}; here, though, the proof that this is indeed a one-to-one correspondence is direct and avoids the roundabout approach of the aforementioned article.

We then go one step further and achieve the enumeration of configurations with excess $1$ in Section~\ref{sec:excess_1}: this is the most complex part of this work. However, the line of argument is of combinatorial nature. We show first that TFPL configurations of excess $1$ look like a configuration of excess $0$ but with one ``defect''. We encode such configurations by puzzles which extend those of Knutson and Tao by one extra piece. To enumerate the puzzles, we determine some rules which move the extra piece to the boundary of the puzzle. This shows how the enumeration in the case of excess $1$ can be reduced to the case of excess $0$; the resulting expression is Theorem~\ref{theo:excess_1}(3). To finish, we show how to deduce the enumeration of ordinary TFPLs with excess $1$ in Theorem~\ref{theo:excess_1o}.

%%%%%%%%%%%%%%%%%%%%%%%%%%%%%%%%%%%%%%%%%%%%%%%%%%%%%%%%%%%%%
%%%%%%%%%%%%%%%%%%%%%%%%%%%%%%%%%%%%%%%%%%%%%%%%%%%%%%%%%%%%%

%%%%%%%%%%%%%%%%%%%%%%%%%%%%%%%%%%%%%%%%%%%%%%%%%%%
\section{Definitions}
\label{sec:definitions}
%%%%%%%%%%%%%%%%%%%%%%%%%%%%%%%%%%%%%%%%%%%%%%%%%%%

%In this section we will first introduce elementary results about $01$-words in Section~\ref{sub:words}. We will then define Fully Packed Loop configurations (FPLs) and their relation to Triangular FPLs  in Section~\ref{sub:fpls_tfpls}. These TFPLs and their oriented versions will be defined precisely in Section~\ref{sub:tfpls}.

In this section we will define FPLs on a triangle (TFPLs), as well as oriented TFPLs, in a more general setting than in previous works. 

\begin{defi}[The graph $G^N$]
 Let $N$ be a positive integer. We define $G^N$ as the induced subgraph of the square lattice $\mathbb{Z}^2$ made up of $N$ consecutive centered rows with $3,5,\ldots,2N+1$ vertices from top to bottom.  
\end{defi}

The graph $G^{6}$ is represented in Figure~\ref{fig:G6}. Note that $G^N$ is a bipartite graph, where the bipartition consists of \emph{odd} and \emph{even} vertices; by convention the vertices on the left side are odd. In the pictures, we will represent odd vertices by circles while even vertices will be represented by squares.  Some vertices play a special role: we let $\BottomEven=\{B_1,\ldots,B_N\}$ be the set of \emph{even} vertices on the bottom row of $G^N$, and $\LeftOdd=\{L_1,\ldots,L_N\}$ (\emph{resp.} $\RightOdd=\{R_1,\ldots,R_N\}$) be the set of odd vertices which are leftmost (\emph{resp.} rightmost) in each row of $G^N$. All vertices $B_i,L_i,R_i$  are numbered from left to right, cf. Figure~\ref{fig:G6} again.

\begin{figure}[!ht]
\centering
\includegraphics{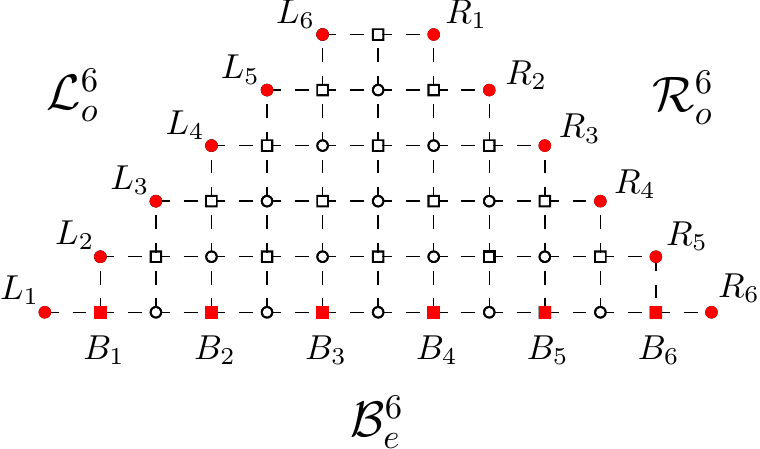}
\caption{\label{fig:G6} The graph $G^{6}$.}
\end{figure}

In the whole article we call \emph{words} of length $N$ the finite sequences $u=u_1u_2\cdots u_N$ where $u_i\in\{0,1\}$ for all $i$. We denote by $|u|_0$ (\emph{resp.} $|u|_1$) the number of occurrences of $0$ (\emph{resp.} $1$) in the word.
Define a partial order on words of length $N$ by $u\leq v$ if and only if $\left|u_1\cdots u_i\right|_1\leq \left|v_1\cdots v_i\right|_1$ for $i=1,\ldots,N$. This is especially nice to see on the 
{\em Ferrers diagram} $\lambda(u)$ associated to the word $u$, cf. Figure~\ref{fig:word_to_partition}.: if $u,v$ each have $N_0$ occurrences of $0$ and $N_1$ occurrences of $1$, then $u\leq v$ if and only if $\lambda(u)\subseteq \lambda(v)$.

\begin{figure}[!ht]
\centering
\includegraphics[height=3cm]{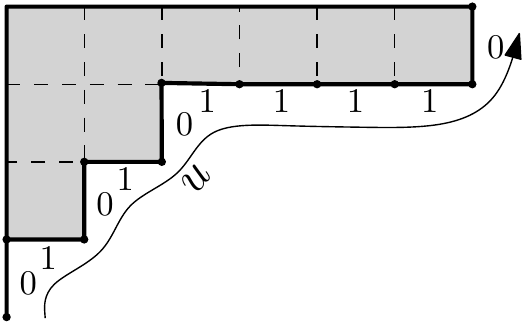}
\caption{\label{fig:word_to_partition} From the word $u=0101011110$ to the Ferrers diagram $\lambda(u)$.}
\end{figure}

%%%%%%%%%%%%%%%%%%%%%%%%%%%%%%%%
\subsection{TFPLs}
\label{sub:tfpls}
%%%%%%%%%%%%%%%%%%%%%%%%%%%%%%%%%

We can now define Fully Packed Loop configurations in a triangle, or TFPLs in short.

\begin{defi}
\label{defi:tfpl}
  A TFPL of size $N$ is a subgraph $f$ of $G^N$, such that: 
\begin{enumerate} 
  \item The $2N$ vertices of $\LeftOdd \cup \RightOdd$ have degree $0$ or $1$.
  \item The $N$ vertices of $\BottomEven$ have degree $1$.
  \item All other vertices of $G^N$ have degree $2$.
  \item A path in $f$ cannot join two vertices of $\LeftOdd$, nor two vertices of $\RightOdd$.\label{it_global}
 \end{enumerate}
 \end{defi}

An example of a TFPL for $N=8$ is given on Figure~\ref{fig:Example_TFPL}.

 In Section~\ref{sub:tfpls_paths}, we will need to consider local configurations around each vertex of a TFPL, and for this reason it is necessary that all vertices have degree $2$ in a TFPL. Therefore we introduce external edges on the left and right boundary, as well as below all even vertices on the bottom boundary, to ensure that all vertices of $G^N$ have degree $2$. These external edges are represented in Figure~\ref{fig:Example_TFPL} by dotted lines.

\begin{figure}[!ht]
\centering
\includegraphics[height=5cm]{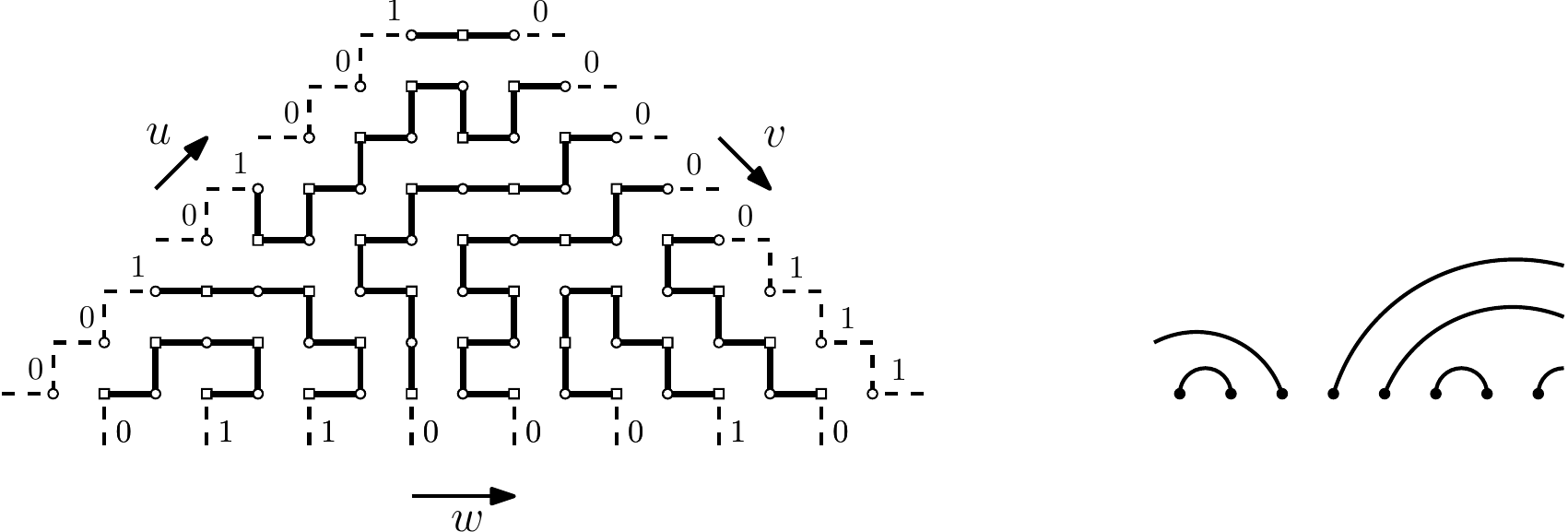}
\caption{\label{fig:Example_TFPL} TFPL of size $8$ and its extended link pattern.}
\end{figure}

The first three conditions of Definition~\ref{defi:tfpl} show that a TFPL configuration is composed of a number of paths, in which the non-closed paths have their extremities in $\LeftOdd\cup\RightOdd\cup\BottomEven$. We will be interested in the structure and enumeration of TFPLs according to certain boundary conditions that depend on the extremities of non-closed paths:

\begin{defi}
\label{defi:tfpl_boundary_conditions}
   To each TFPL $f$ are associated three words $u,v,w$ of length $N$ as follows:
\begin{enumerate} 
  \item If the vertex $L_i\in\LeftOdd$ has degree $1$ then $u_i:=1$, otherwise $u_i:=0$. \label{it_tfpl1}
  \item If the vertex $R_i\in\RightOdd$ has degree $1$ then $v_i:=0$, otherwise $v_i:=1$. \label{it_tfpl2}
  \item Consider the path starting from the vertex $B_i$, and let $X$ be the other endpoint of this path. If $X\in\LeftOdd\cup\{B_1,\ldots,B_{i-1}\}$ then $w_i:=1$, while if it belongs to $\RightOdd\cup\{B_{i+1},\ldots,B_{N}\}$ then $w_i:=0$.\label{it_tfpl3}                               
 \end{enumerate}
 We say that the TFPL $f$ has boundary $(u,v;w)$. We denote the set of configurations with boundary $(u,v;w)$ by $T_{u,v}^{w}$, and its cardinality by $t_{u,v}^{w}$.
\end{defi}

The words $u,v,w$ attached to the TFPL of Figure~\ref{fig:Example_TFPL} are represented on the same figure.
We first note an evident symmetry of TFPLs. Given a word $u=u_1u_2\ldots u_N$, define $u^*$ as the word $\overline{u_N}\ldots \overline{u_2} \, \overline{u_1}$ where $\overline{0}=1,\overline{1}=0$.  Also note that $\lambda(u^*)$ is the conjugate of $\lambda(u)$.

\begin{prop}
\label{prop:vertical_symmetry}
 Vertical symmetry exchanges $T_{u,v}^{w}$ and $T_{v^*,u^*}^{w^*}$; in particular, $t_{u,v}^{w}=t_{v^*,u^*}^{w^*}$.
\end{prop}
 
  Define $\LeftOdd(u)=\{L_i\in\LeftOdd~:~u_i=1\}$ and $\RightOdd(v)=\{R_i\in\RightOdd~:~v_i=0\}$. Given a configuration in $T_{u,v}^{w}$, the set of all endpoints of its paths is then $\LeftOdd(u)\cup\RightOdd(v)\cup\BottomEven$ (when ignoring the external edges). To encode the pairs of endpoints linked by a path in $f$, we need the notion of extended link patterns.

%%%%%%%%%%%%%%%%%%%%%%%%%%%%%%%%%%%%
\subsection{Extended link patterns}
\label{sub:elp}
%%%%%%%%%%%%%%%%%%%%%%%%%%%%%%%%%%%%

 Define a {\em link pattern} $\pi$ of size $n$ as a partition of $\{1,\ldots,2n\}$ in $n$ {\em pairwise noncrossing} pairs $\{i,j\}$, which means that there are no integers $i<j<k<\ell$ such that $\{i,k\}$ and $\{j,\ell\}$ are both in $\pi$. We will represent link patterns as noncrossing arches between $2n$ aligned points, see Figure~\ref{fig:link_pattern}. We denote by $\Dn$ the set of words $u$ of length $N=2n$, such that $|u|_0=|u|_1=n$ and each prefix $u'$ of $u$ verifies $|u'|_0\geq|u'|_1$; these are called \emph{Dyck words}.

\begin{figure}[!ht]
\begin{center}
\includegraphics[width=0.3\textwidth]{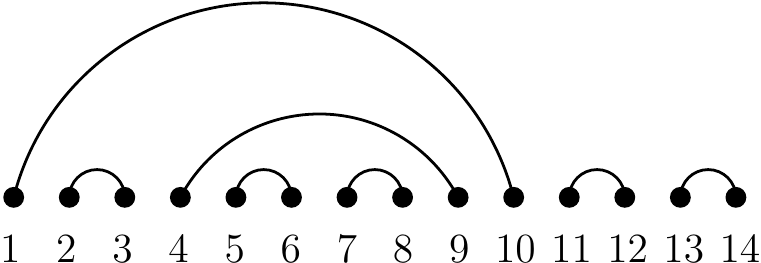}
\caption{A link pattern.
\label{fig:link_pattern}}
\end{center}
\end{figure}

There is then a bijection between link patterns and Dyck words, defined simply by associating to $\pi$ the word $u$ such that $u_i=0$ if and only if $i$ is the smaller element in the pair $\{i,j\}$ of $\pi$. As an example, $00100101110101$ is  associated to the pattern of Figure~\ref{fig:link_pattern}. This correspondence can be extended to all words (and not only Dyck words) by introducing the notion of extended link patterns:

\begin{defi}
 An \emph{extended link pattern} $\pi$ on $\{1,\ldots,N\}$ is the data of integers 
$1\leq \ell_1<\ell_2<\ldots<\ell_i$ (left points) and $r_1<\ldots<r_j\leq N$ (right points), with $\ell_i<r_1$, together with a link pattern on each maximal interval of integers not containing any of the points $\ell_k$ or $r_k$. 
\end{defi}

In figures we will represent left and right points by attaching the extremity of an arch to the points $\ell_k$ and $r_k$, with the arch going left (\emph{resp.} right) for a left point $\ell_k$ (\emph{resp.} a right point $r_k$); see Figure~\ref{fig:extended_link_pattern}, where the left points are $3,4$ and the right point is $11$. Extended link patterns can be in fact equivalently defined as usual link patterns on 
$\{-(i-1),\ldots,N+j\}$ for certain $i,j\geq 0$, such that no two elements in  $\{-(i-1),\ldots,0\}\cup\{N+1,\ldots,N+j\}$ belong to the same pair.

\begin{figure}[!ht]
\begin{center}
\includegraphics[width=0.4\textwidth]{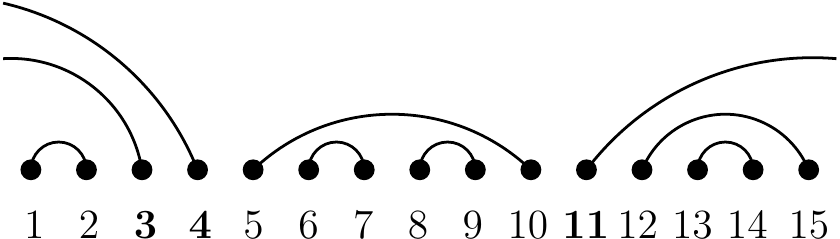}
\caption{An extended link pattern.
\label{fig:extended_link_pattern}}
\end{center}
\end{figure}

Given an extended link pattern $\pi$ with integers $\ell_k$ and $r_k$ as above, define a word $w=\mathbf{w}(\pi)$ of length $N$ as follows: first set $w_{\ell_k}:=1$ and $w_{r_k}:=0$ for all left and right points, and associate with each link pattern appearing in $\pi$ its corresponding Dyck word. As an example, the word associated with the pattern of Figure~\ref{fig:extended_link_pattern} is $011100101100011$.

\begin{prop}
\label{prop:bij_ext_link}
The function $\mathbf{w}$ is a bijection from extended link patterns on $\{1,\ldots,N\}$ to
 words of length $N$.
\end{prop}

\begin{proof}
 We show how to construct the inverse of $\mathbf{w}$. Let $w$ be a word of length $N$: it can be uniquely decomposed as the concatenation
\begin{equation}
\label{eq:decompo_w}
 w=(x_11)(x_21)\ldots (x_i 1)y(0z_1) (0z_2)\ldots (0z_j),
\end{equation}
  where $i,j\geq 0$ and the words $x_k$, $y$ and $z_k$ are all Dyck words. Let $\ell_1<\ell_2<\ldots<\ell_i$ and $r_1<\ldots<r_j$ be the indices of the $1$s and $0$s respectively which occur in~\eqref{eq:decompo_w}. Construct an extended link pattern on $\{1,\ldots,N\}$ as follows: the $\ell_k$ and $r_k$ are left and right points respectively, while to each Dyck word in~\eqref{eq:decompo_w} associate a usual link pattern. This gives the desired inverse bijection to $\mathbf{w}$, as is readily checked.
\end{proof}

 To a TFPL $f$ is naturally associated an extended link pattern $\pi$ on $\{1,\ldots,N\}$ as follows: if $B_i,B_j\in\BottomEven$ are linked by a path in $f$, then $\{i,j\}\in\pi$, while if $B_i$ is linked to a vertex of $\LeftOdd(u)$ (\emph{resp.} $\RightOdd(v)$) then $i$ is a left point of $\pi$ (\emph{resp.} a right point). The following is now an 
 immediate consequence of Definition~\ref{defi:tfpl_boundary_conditions}~(\ref{it_tfpl3}) and the definition of 
 $\mathbf{w}(\pi)$. 

\begin{prop}
\label{prop:w_and_link_pattern}
 For any TFPL $f\in T_{u,v}^{w}$ with extended link pattern $\pi$, one has $w=\mathbf{w}(\pi)$ .
\end{prop}

%\begin{proof}
%Consider first an index $i$ such that $w_i=0$ and $w_{i+1}=1$. The vertices $B_i$ and $B_{i+1}$ in $f$ must then be linked by a path, because of Condition~\ref{it_tfpl3} in Definition~\ref{defi:tfpl_boundary_conditions} and the fact that paths cannot cross. By iterating this observation, we have that any Dyck word appearing in the factorization~\eqref{eq:decompo_w} is indeed associated to a link pattern on vertices $B_i\in\BottomEven$, as in Proposition~\ref{prop:bij_ext_link}. Now the remaining $1$s (\emph{resp.} $0$s) in ~\eqref{eq:decompo_w} are necessary linked to vertices in $\LeftOdd(u)$ (\emph{resp.} $\RightOdd(v)$), and are therefore left points (\emph{resp.} right points) of the extended link pattern attached to $w$. This is precisely the bijection of Proposition~\ref{prop:bij_ext_link} and thus achieves the proof.
%\end{proof}
 
So the words $u,v,w$ describe exactly where each path of $f$ starts and ends: the set of endpoints is determined once we know $u$ and $v$, and Proposition~\ref{prop:w_and_link_pattern} shows that $w$ encodes the extended link pattern, which suffices to determine the pairs of endpoints which are connected together.
 
%%%%%%%%%%%%%%%%%%%%%%%%%%%%%%%%
\subsection{Oriented TFPLs}
\label{sub:oriented_tfpls}
%%%%%%%%%%%%%%%%%%%%%%%%%%%%%%%%

TFPLs with boundary conditions $u,v,w$ appear naturally in the study of FPLs on the square grid, as shown in~\cite{CKLN,TFPL1}. The difficulty in enumerating them lies in part in the fact that their definition involves global conditions, since both conditions,~\eqref{it_global} in Definition~\ref{defi:tfpl} and~\eqref{it_tfpl3} in Definition~\ref{defi:tfpl_boundary_conditions}, involve figuring out how endpoints are connected two by two. The notion of oriented TFPLs that we study in this section only involves local conditions, and will therefore be easier to deal with. Their relation to TFPLs is studied in Section~\ref{sec:TFPLs}, where we will see the important fact that one can recover the enumeration of TFPLs from a certain weighted enumeration of oriented TFPLs.

\begin{defi}
\label{defi:otfpl}
  An \emph{oriented TFPL} of size $N$ is a TFPL on $G^N$ together with an orientation of each edge with the following conditions: each degree $2$ vertex has one incoming and one outgoing edge; the edges attached to $\LeftOdd$ are outgoing; the edges attached to $\RightOdd$ are incoming. 
\end{defi}

We introduce the same external edges as in non-oriented TFPLs, represented by dotted lines in Figure~\ref{fig:Example_oTFPL}; their orientation is chosen such that each vertex has one incoming and one outgoing edge. Note also that the global Condition~\eqref{it_global} in Definition~\ref{defi:tfpl} can be omitted when dealing with oriented TFPLs, since the orientations on the left or right boundaries automatically prevent paths from returning to these boundaries; therefore the constraints on an oriented TFPL configuration are indeed local.

\begin{figure}[!ht]
\centering
\includegraphics[height=5cm]{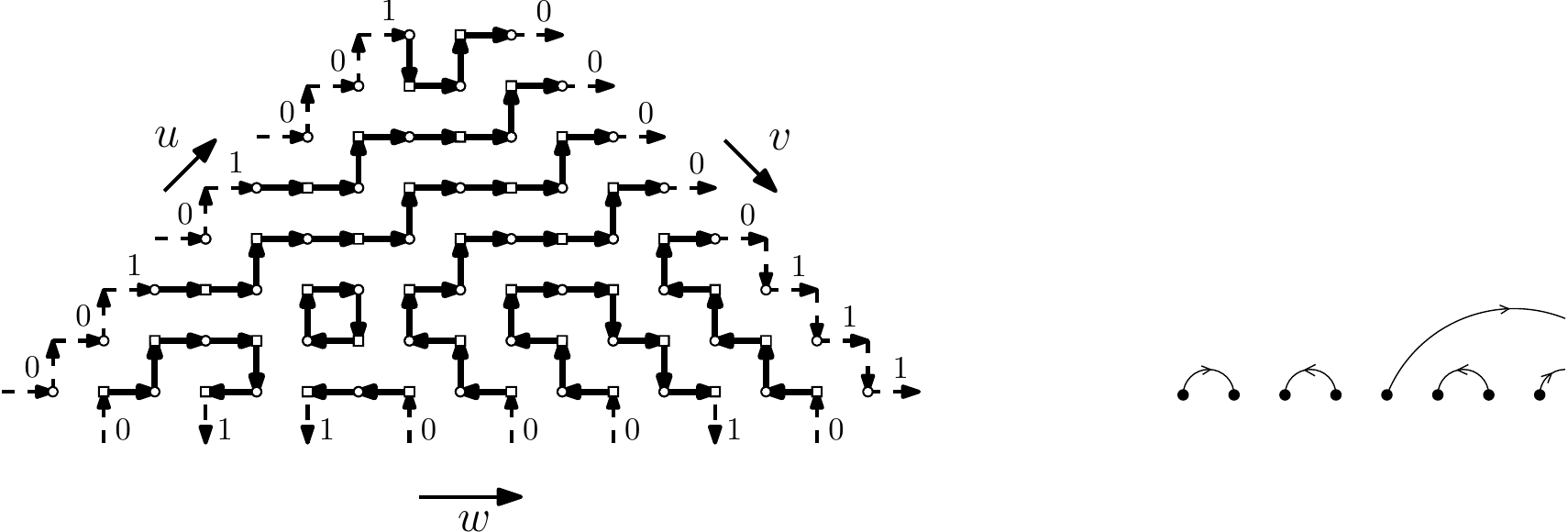}
\caption{\label{fig:Example_oTFPL} Oriented TFPL of size $8$ and its directed extended link pattern.}
\end{figure}

\begin{defi}
\label{defi:otfpl_boundary_conditions}
   We say that the oriented TFPL $f$ has boundary $(u,v;w)$ if the following hold:
\begin{itemize} 
  \item if the vertex $L_i\in\LeftOdd$ has out-degree $1$ then $u_i=1$, otherwise $u_i=0$;
  \item if the vertex $R_i\in\RightOdd$ has in-degree $1$ then $v_i=0$, otherwise $v_i=1$;
  \item if the vertex $B_i\in\BottomEven$ has in-degree $1$ then $w_i=1$, while if it has out-degree $1$ then $w_i=0$.
  \end{itemize}                                                                 
We denote the set of oriented configurations by $\overrightarrow{T}_{u,v}^{w}$ and their number by $\overrightarrow{t}_{u,v}^{w}$.                                  \end{defi}

Notice the important fact that while $u$ and $v$ have the same interpretation as in Definition~\ref{defi:tfpl_boundary_conditions} for the underlying TFPL, this is not the case for $w$ which concerns the local orientation of the edges and not the global connectivity of the paths.

 The following oriented version of Proposition~\ref{prop:vertical_symmetry} is immediate:

\begin{prop}
\label{prop:vertical_symmetry_oriented}
 Vertical reflection together with the reorientation of all edges 
exchanges $\overrightarrow{T}_{u,v}^{w}$ and $\overrightarrow{T}_{v^*,u^*}^{w^*}$; in particular, $\overrightarrow{t}_{u,v}^{w}=\overrightarrow{t}_{v^*,u^*}^{w^*}$.
\end{prop}

Also the concept of extended link patterns has the following natural analog in this context:

\begin{defi}
 A {\em directed extended link pattern} $\overrightarrow{\pi}$  on $\{1,2,\ldots,N\}$  is an extended link pattern on $\{1,2,\ldots,N\}$, such that in each pair from one of the link patterns we let an integer be the {\em source} and the other be the {\em sink}; by definition, left points are sinks, while right points are sources.  We let $RL(\overrightarrow{\pi})$ be the number of pairs where the larger integer is the source.
 We assign the  {\em source-sink word} $w=w_1\ldots w_{N}$  on $\{1,2,\ldots,N\}$ to an extended directed link pattern as follows: we set $w_i=0$ if and only if $i$ is a source in $\overrightarrow{\pi}$.

\end{defi}

 It is clear that each oriented TFPL is naturally associated with an extended directed link pattern; we represent $\overrightarrow{\pi}$ by orienting each linked pair from its source to its sink, while left and right points have their attached half arch oriented to the right, cf. Figure~\ref{fig:Example_oTFPL}, right. In this representation $RL(\overrightarrow{\pi})$ counts the number of arrows going from right to left, and is equal to $2$ in this example. 
% In Definition~\ref{defi:otfpl_boundary_conditions}, the word $w$ records the source-sink pattern of the directed extended link pattern $\overrightarrow{\pi}$ attached to $f$, and not the underlying link pattern as in Definition~\ref{defi:tfpl_boundary_conditions}\eqref{it_tfpl3} for non-oriented TFPLs.

There is a natural injection from ${T}_{u,v}^{w}$ to $\overrightarrow{T}_{u,v}^{w}$: given a TFPL $f$ with boundary $(u,v;w)$, orient all its closed paths clockwise, and each path between two vertices $B_i,B_j$ from $B_i$ to $B_j$ if $i<j$. The other paths have a forced orientation by Definition~\ref{defi:otfpl}. Note that the chosen orientation ensures that $w$ is indeed the bottom boundary word of the resulting oriented TFPL, therefore this is an injection from ${T}_{u,v}^{w}$ to $\overrightarrow{T}_{u,v}^{w}$, so that we have
 \begin{equation}
  \label{eq:embedding}
t_{u,v}^{w}\leq\overrightarrow{t}_{u,v}^{w}\quad\text{for any words $u,v,w$ of length $N$.}
 \end{equation}

In the other direction, with each oriented TFPL we can associate a non-oriented TFPL by ignoring the direction of the edges, but this operation does not preserve the bottom words in general. In Section~\ref{sec:TFPLs}, we will explain how from a certain weighted enumeration of oriented TFPLs one can deduce the numbers  $t_{u,v}^{w}$.
%This defines a function from  $\overrightarrow{T}_{u,v}^{w}$ to $\sqcup{T}_{u,v}^{w}$. Note that the image has bottom boundary word $w$ if and only if each path in the oriented TFPL connecting two vertices $B_i$ to $B_j$ is directed from left to right

%%%%%%%%%%%%%%%%%%%%%%%%%%%%%
\subsection{FPLs and TFPLs}
\label{sub:fpls_tfpls}
%%%%%%%%%%%%%%%%%%%%%%%%%%%%%

We recall here the definition of FPLs and their connection to TFPLs, and refer to~\cite{TFPL1} for a detailed explanation of interactions between FPLs and TFPLs. We fix a positive integer $n$, and let $Q_n$ be the square grid with $n^2$ vertices. We impose {\em periodic boundary conditions} on $Q_n$, which means that we select every other external edge on the grid, starting by convention with the topmost on the left side; we number these $2n$  external edges counterclockwise. A {\em Fully Packed Loop (FPL)} configuration $F$ of size $n$ is a subgraph of $Q_n$  such that each vertex of $Q_n$ is incident to two edges of $F$. An example of an FPL configuration is given in Figure~\ref{fig:fpl_example} (left).

\begin{figure}[!ht]
\begin{center}
\includegraphics[width=0.8\textwidth]{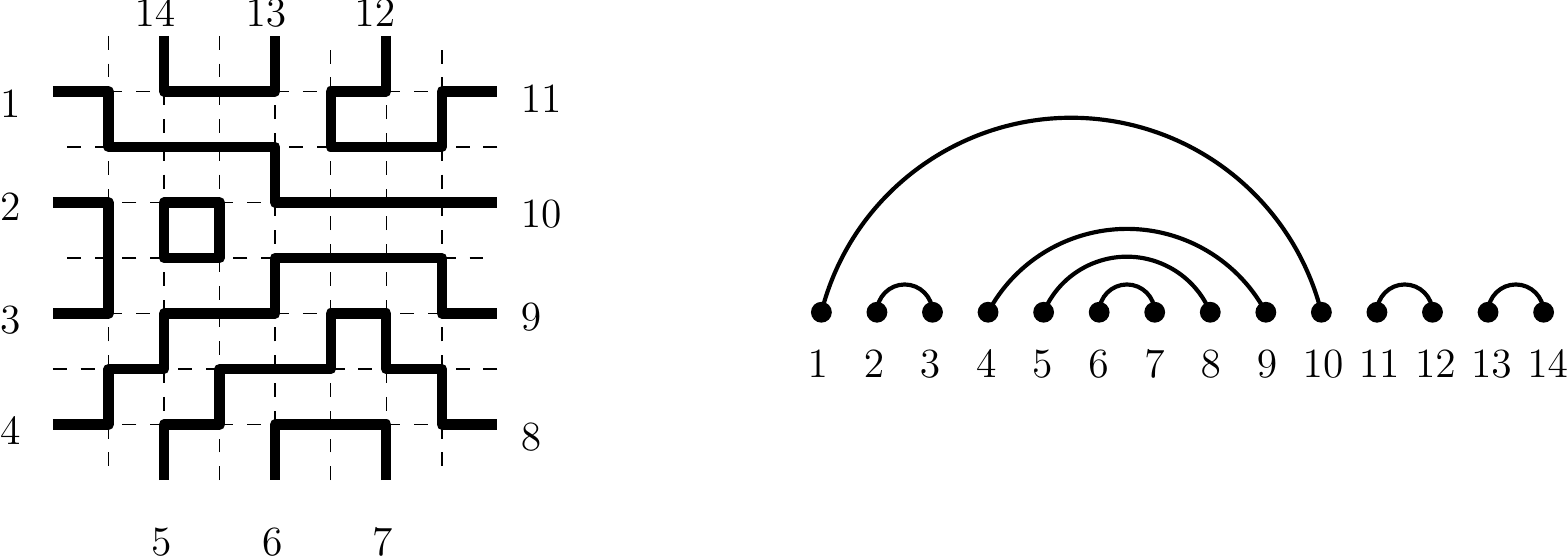}
\caption{An FPL configuration with its associated link pattern.
\label{fig:fpl_example}}
\end{center}
\end{figure}

 An FPL configuration $F$ on $Q_n$ naturally defines non-crossing paths between its external edges, so we can define the link pattern $\mathbf{\Pi}(F)$ as the set of pairs $\{i,j\}$ where $i,j$ label external edges which are the extremities of the same path in $F$: see Figure~\ref{fig:fpl_example}, right. If $\pi$ is a link pattern, we denote by ${A}_\pi$ the number of FPL configurations $F$ of size $n$ such that $\mathbf{\Pi}(F)=\pi$. Given an integer $m\geq 0$, define $\pi\cup m$ as the link pattern on $\{ 1,\ldots,2(n+m)\}$ given by the nested pairs $\{i,2n+2m+1-i\}$ for $i=1\ldots m$, and the pairs $\{i+m,j+m\}$ for each $\{i,j\}\in \pi$. Note that FPLs $F$ such that $\mathbf{\Pi}(F)=\pi\cup m$ are of size $n+m$.

 Given a Dyck word $\sigma\in\Dn$, let $\sigma'$ be the word obtained by removing the initial $0$ and final $1$ in $\sigma$, so that $\sigma=0\sigma'1$. It was shown in~\cite{CKLN,Thapper,TFPL1} that one has:
\begin{equation}
\label{eq:tfpl_to_fpl}
 A_{\pi\cup m}=\sum_{\sigma,\tau\in \Dn}\ssyt(\lambda(\sigma),n)\cdot t_{\sigma',\tau'}^{\mathbf{w}(\pi)'}\cdot \ssyt(\lambda(\tau^*),m-2n+1),
\end{equation}

Here $\ssyt(\lambda,m)$ is the number of semistandard tableaux of shape $\lambda$ and entries in $\{1,\ldots,m\}$; it is a polynomial in $m$ given by the hook-content formula.

%%%%%%%%%%%%%%%%%%%%%%%%%%%%%%%%%%%%%%%%%%%%%%%%%%%%%%%%%%%%%
%%%%%%%%%%%%%%%%%%%%%%%%%%%%%%%%%%%%%%%%%%%%%%%%%%%%%%%%%%%%%

%%%%%%%%%%%%%%%%%%%%%%%%%%%%%%%%%%%%%%%%%%%%%%%%%%%
\section{Recovering TFPLs from oriented TFPLs}
\label{sec:TFPLs}
%%%%%%%%%%%%%%%%%%%%%%%%%%%%%%%%%%%%%%%%%%%%%%%%%%%

In this section we study more precisely the connection between TFPLs and oriented TFPLs. The main result is that we can deduce the number of TFPLs from a certain weighted enumeration of oriented TFPLs, see Corollary~\ref{cor:otfpls_to_tfpls}.

\subsection{A relation between the orientation of closed paths, paths oriented from right to left and turns}
\label{sub:tfpls_paths}

We consider directed polygons in the plane. The {\it signed curvature}  of a turn  is the angle in $(-\pi,\pi]$ between the extension of the incoming edge and the outgoing edge, where we take the negative angle  if the turn is to the left, see Figure~\ref{fig:turningnumber}. The {\it turning number} of a directed polygon is the sum of the signed curvatures of its turns. The following is equivalent to the well-known fact that the sum of exterior angles of an undirected simple closed polygon (convex or concave) is $2 \pi$; see for instance~\cite{Meisters} for a proof of the latter.

\begin{figure}[ht]
\centering
\scalebox{0.25}{\includegraphics{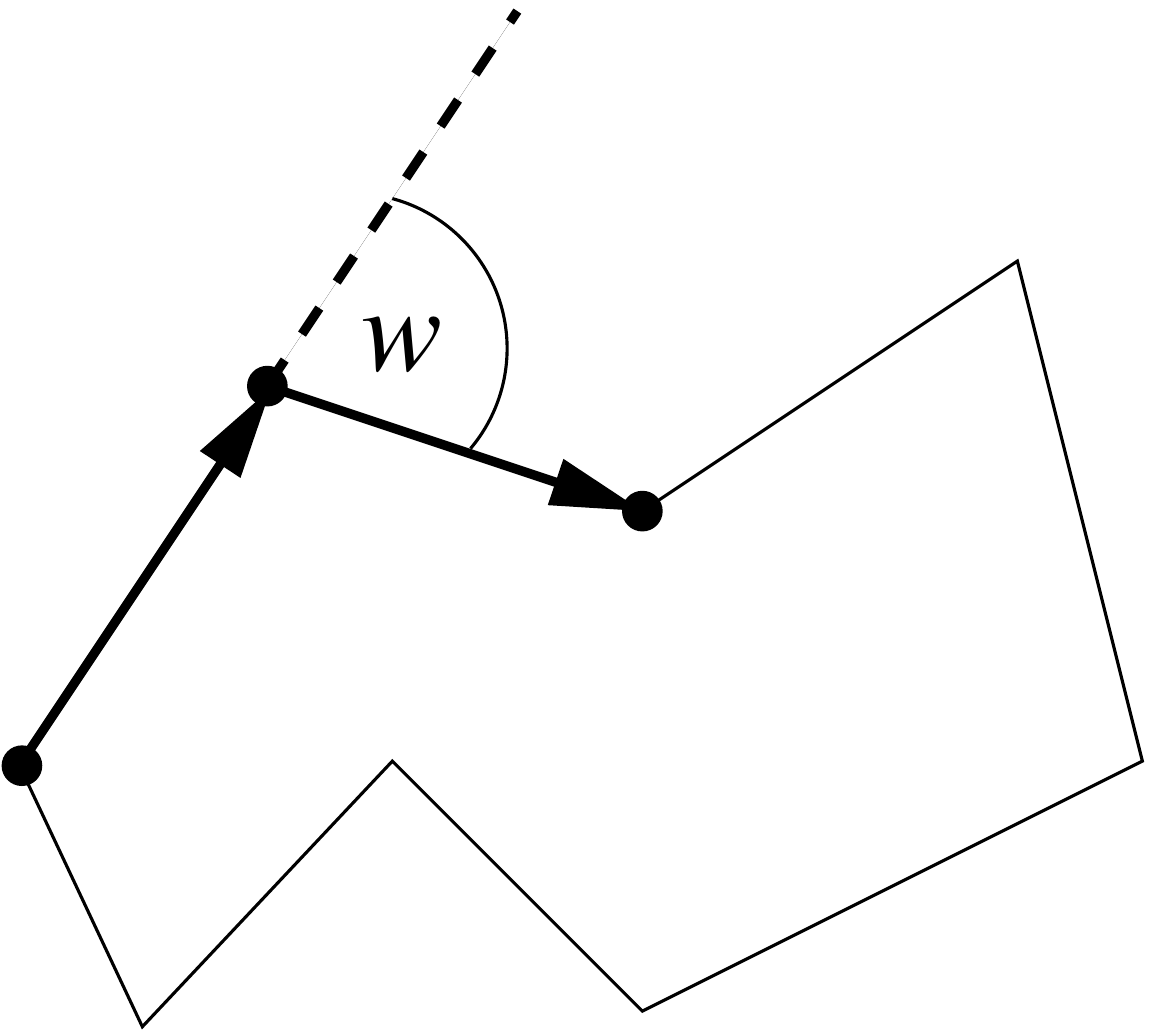}}
\hspace{4cm}
\scalebox{0.25}{\includegraphics{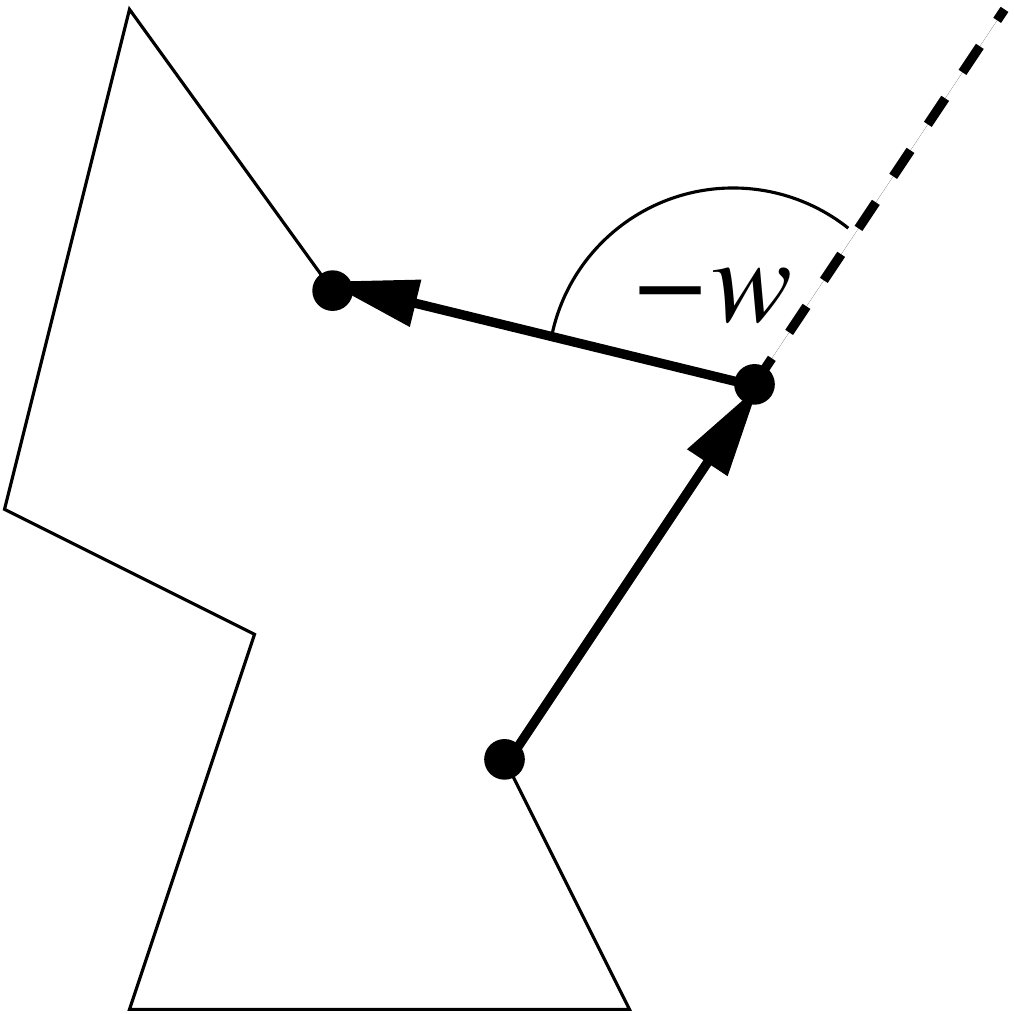}}
\caption{\label{fig:turningnumber} Signed curvature of turns.}
\end{figure}

\begin{lem}
\label{lem:turningnumber}
The turning number of a directed closed self-avoiding polygon is $2 \pi$ if it is oriented clockwise and it is $-2 \pi$ otherwise.
\end{lem}

We now restrict our considerations to directed lattice paths on the square grid with step set $\{ (\pm 1,0),  (0,\pm 1)\}$
with the additional assumption that our paths do not contain two consecutive steps going in opposite directions.
We say that a step is of type {\bf u} if it is a $(0,1)$-step; similar for {\bf r}, {\bf d}, {\bf l}. The eight possible turns  are displayed and named in Figure~\ref{fig:curve}. For a given directed path $p$, let $x_{ur}$ denote the number of turns of type $ur$, $x_{ru}$ denote the number of turns of type $ru$, etc., and set $x^{\circlearrowright} (p):=(x_{ur},x_{rd},x_{dl},x_{lu})$ and  $x^\circlearrowleft (p):=(x_{ru},x_{dr},x_{ld},x_{ul})$.

\begin{figure}[ht]
\centering
\includegraphics[height=3cm]{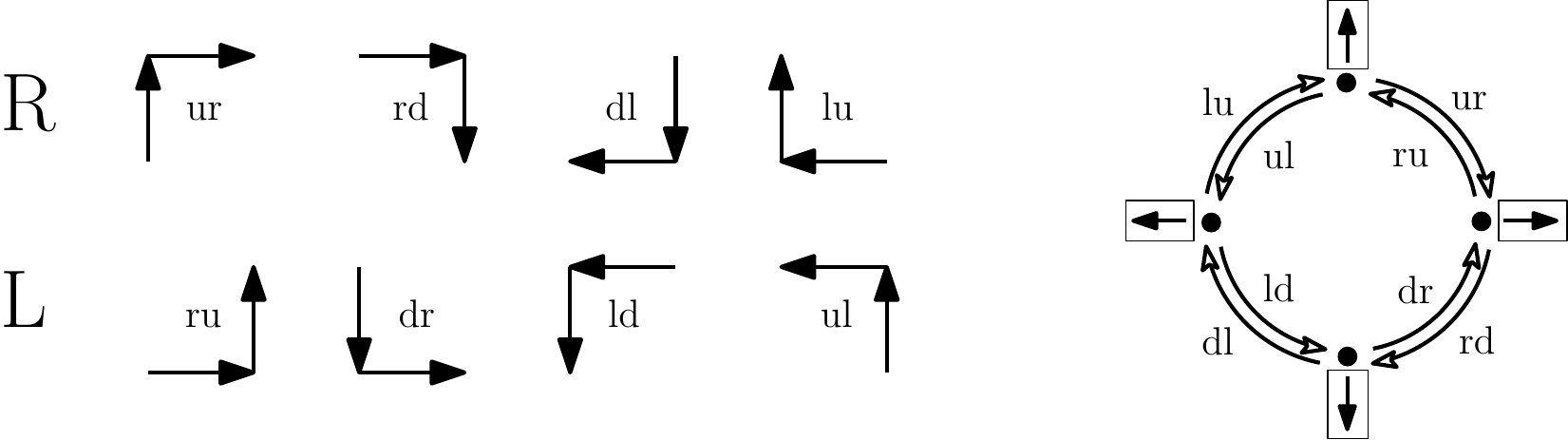}
\caption{\label{fig:curve} The eight types of turns and their possible successions.}
\end{figure}

The succession of steps is encoded by the simple graph of Figure~\ref{fig:curve} (right), which has {\bf u},{\bf r}, {\bf d}, {\bf l} as vertices and the possible turns as edges.

\begin{prop} 
\label{prop:turngrid}
Let $p$ be a directed path on the square grid.
\begin{enumerate}
\item[(1)] There exists an
integer $k$ with the property that 
\[x^\circlearrowright (p)-x^\circlearrowleft (p)-k (1,1,1,1)=:v(p)\]
belongs to $\{0,1\}^4$.
\item[(2)] 
The vector $v(p)=(v_{ur},v_{rd},v_{dl},v_{lu})$ is determined as follows: let the first and last steps of $p$ be {\bf a} and {\bf b} respectively, and consider the shortest clockwise path $C$ from {\bf a} to {\bf b} in the graph of Figure~\ref{fig:curve}, right. The coordinate $v_t$ equals $1$ if and only if $t$ labels an edge in $C$.
\item[(3)] Now suppose $p$ is closed and self-avoiding. Then
$(x^\circlearrowright (p)-x^\circlearrowleft (p)) \cdot (1,1,1,1)^t=4$ if the orientation is clockwise and $(x^\circlearrowright (p)-x^\circlearrowleft (p)) \cdot (1,1,1,1)^t=-4$ if the orientation is counterclockwise.
\end{enumerate}
\end{prop}

\begin{proof}
The first and the second part are a direct consequence of the fact that the possible sequences of successive turns in a path correspond to the sequences of edge labels of paths in the graph represented in Figure~\ref{fig:curve}, right. As for (3), it is a direct consequence of Lemma~\ref{lem:turningnumber}, because the signed curvature of the turns in the first row of Figure~\ref{fig:curve} is $\pi /2$ and it is $-\pi/2$ for the turns in the second row. 
\end{proof}

In the following consequence of the proposition, let $R$ (\emph{resp.} $L$) denote the set of turns displayed in the first (\emph{resp.} second) row of Figure~\ref{fig:curve}. Also, if $t$ is any type of turn and $P$ a collection of directed paths, we denote by $t(P)$ the number of occurrences of turns of type $t$ in $P$.

\begin{cor} 
\label{cor:loops}
Let $p$ be a directed closed self-avoiding path on the square grid and $t_\circlearrowright \in R, t_\circlearrowleft \in L$. Then $t_\circlearrowright(p)-t_\circlearrowleft(p)$ is equal to $1$ (\emph{resp.} $-1$) if $p$ is oriented clockwise (\emph{resp.} counterclockwise).
\end{cor}

\begin{proof}
Proposition~\ref{prop:turngrid}(2) implies $v(p)=0$. By Proposition~\ref{prop:turngrid}(3), the integer $k$ in Proposition~\ref{prop:turngrid}(1) is $1$ if the orientation is clockwise and $-1$ otherwise.
\end{proof}

This enables us to provide an interpretation for the difference of the number of turns of type $t_\circlearrowright$ and the number of turns of type $t_\circlearrowleft$ in an oriented TFPL if $t_\circlearrowright \in \{ dl, lu\}=:R'$ and $t_\circlearrowleft \in \{ ld, ul\}=:L'$. 

\begin{prop} 
\label{prop:counterclockwise}
Let $t_\circlearrowright \in R', t_\circlearrowleft \in L'$. For any oriented TFPL $f$, let $\overrightarrow{\pi}$ be the associated directed extended link pattern, and denote $RL(f):=RL(\overrightarrow{\pi})$. Also let $N^\circlearrowright(f),$ \emph{resp.} $N^\circlearrowleft(f)$, denote the number of closed paths in $f$ which are oriented clockwise, \emph{resp.} counterclockwise. Then
\[t_\circlearrowleft(f)-t_\circlearrowright(f) = RL(f)+\left(N^\circlearrowleft(f)-N^\circlearrowright(f)\right).\]
\end{prop}  

\begin{proof} By Corollary~\ref{cor:loops}, it is enough to show the following for $p$ a non closed path: 
$t_\circlearrowleft(p)=t_\circlearrowright(p)+1$ if $p$ goes from a bottom vertex $B_i$ to a vertex $B_j$ with $i>j$ (``from right to left''), and $t_\circlearrowleft(p)=t_\circlearrowright(p)$ otherwise.

By inspection, $p$ can start with a step of type {\bf r} or {\bf u} and can end with a step of type {\bf r} or {\bf d}: here the dotted edges in Figure~\ref{fig:Example_oTFPL} are considered part of the paths. By Proposition~\ref{prop:turngrid} $x^\circlearrowright (p)-x^\circlearrowleft (p)-k(1,1,1,1) \in 
\{(0,0,0,0),(1,0,0,0),(0,1,0,0),(1,1,0,0)\}$ in this case, in particular 
the coordinates corresponding to the turns $dl$, $ld$, $lu$ and 
$ul$ are vanishing. Now we complete such a path to a directed closed self-avoiding path by adding a path below the configuration with the least possible number of turns, see Figure~\ref{fig:complete} in a particular case. The closed path is oriented clockwise if and only if the original path was directed from left to right. By Proposition~\ref{prop:turngrid}(3), we have $k=0$ in this case and $k=-1$ otherwise. 
\end{proof}

\begin{figure}[!ht]
\scalebox{0.25}{\includegraphics{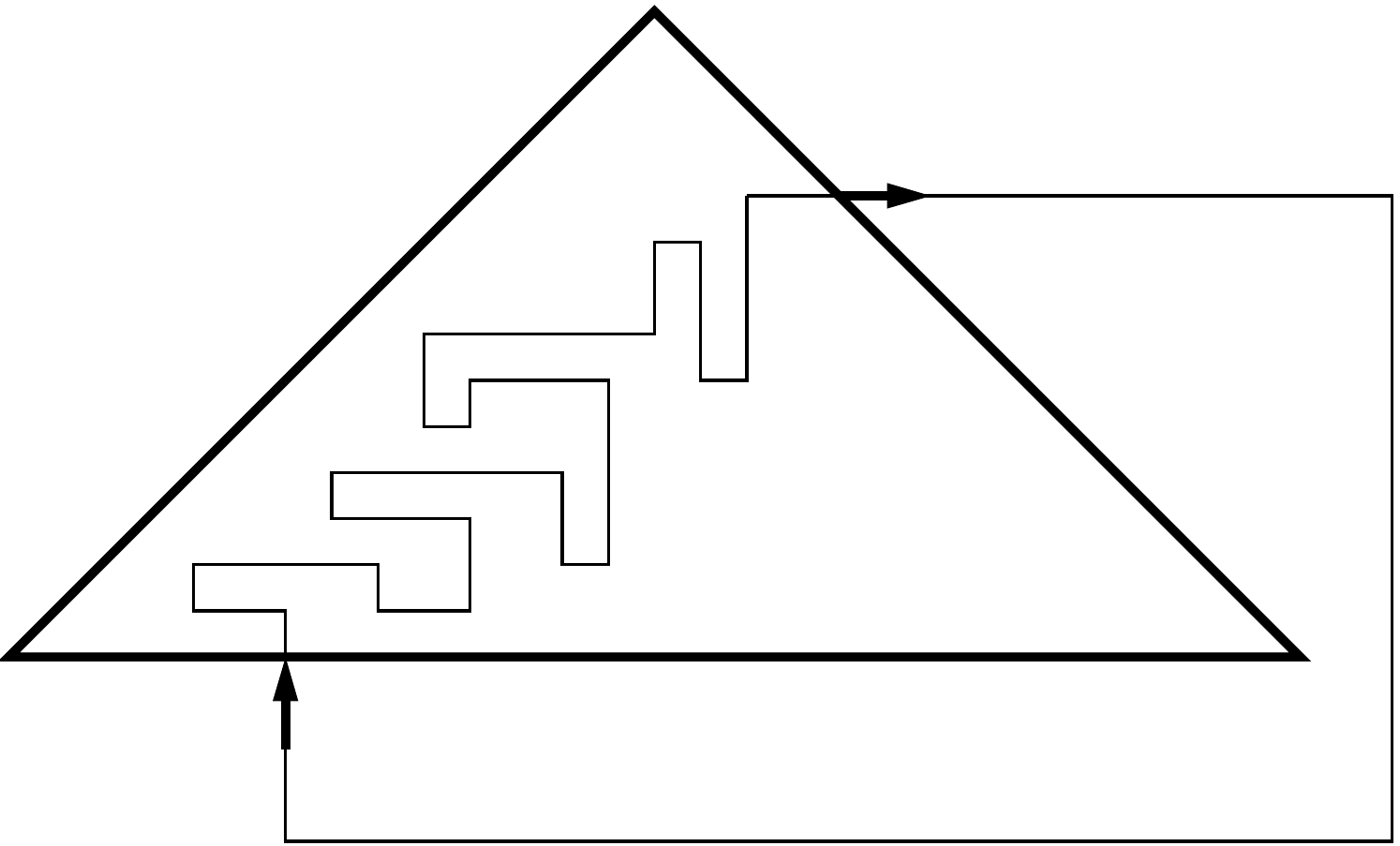}}
\caption{Closure of a path in a TFPL. \label{fig:complete} }
\end{figure}

This motivates the definition of the weighted enumeration of oriented TFPLs: fix $t_\circlearrowright \in R'$ and $t_\circlearrowleft \in L'$. 

\begin{equation}
\label{eq:weighted_tfpl}
 \overrightarrow{t}_{u,v}^{w}(q)=\sum_{f\in\overrightarrow{T}_{u,v}^{w}}q^{t_\circlearrowleft(f)-t_\circlearrowright(f)}=\sum_{f\in\overrightarrow{T}_{u,v}^{w}}q^{RL(f)}q^{N^\circlearrowleft(f)-N^\circlearrowright(f)}
\end{equation}

\subsection{Reorienting paths oriented from right to left}

The goal of this section is to  express the number of ordinary TPFLs in terms of the weighted enumeration of oriented TFPLs which was introduced 
in the previous section.  To this end, we let $\overline{T}_{u,v}^{w}$ denote the subset of oriented TFPLs in
$\overrightarrow{T}_{u,v}^{w}$ where the associated directed extended link pattern $\overrightarrow{\pi}$
verifies $RL(\overrightarrow{\pi})=0$, which means that all bottom paths
are oriented from left to right. Let also $\overline{t}_{u,v}^{w}(q)$ be the corresponding weighted enumeration, cf.~\eqref{eq:weighted_tfpl}. 
The following proposition relates $t_{u,v}^{w}$ to $\overline{t}_{u,v}^{w}(q)$.

\begin{prop} 
\label{prop:rho}
Let $\rho$ be a primitive sixth root of unity, so that $\rho$ verifies $\rho + 1/\rho=1$. Then
$t_{u,v}^{w} = \overline{t}_{u,v}^{w}(\rho)$.
\end{prop}

\begin{proof}
By Proposition~\ref{prop:counterclockwise}, we have
\[ t_\circlearrowleft(f)-t_\circlearrowright(f) = N^\circlearrowleft(f)-N^\circlearrowright(f)\]
for all elements in $\overline{T}_{u,v}^{w}$. Thus
$$
\overline{t}_{u,v}^{w}(q) = \sum_{m=0}^{\infty}
\sum_{\substack{f \in
{T}_{u,v}^{w}\\ \text{$f$ has $m$ closed paths}}}  \sum_{i=0}^{m}
\binom{m}{i} q^i (1/q)^{m-i}
= \sum_{m=0}^{\infty}
\sum_{\substack{f \in
{T}_{u,v}^{w}\\ \text{$f$ has $m$ closed paths}}} (q + 1/q)^m.
$$
The assertion follows as $\rho + 1/\rho=1$.
\end{proof}

\begin{defi}
 A word $w'$ of size $N$ is {\em feasible} for a
word $w$ of length $N$ if there exists a directed extended link
pattern $\overrightarrow{\pi}$  with underlying extended link pattern ${\bf w}^{-1} (w')$ such that $w$ is the source-sink word of $\overrightarrow{\pi}$. Such a $\overrightarrow{\pi}$ is unique, and we can then define 
$g(w,w')=RL(\overrightarrow{\pi})$ for all words $w,w'$ such that $w'$ is feasible for $w$.
\end{defi}
                                                                                                                                                                                                   
Every word $w$ is feasible for itself, by orienting all arches in a link pattern $\pi$ from right to left, and clearly one has $g(w,w)=0$. For another example, the word $w'=011100101100011$ is feasible for $w=101111001000011$. Indeed, $w'=\mathbf{w}(\pi)$ where $\pi$ is represented in Figure~\ref{fig:extended_link_pattern}. If one orients the pairs $\{1,2\},\{5,10\}$ and $\{6,7\}$ from right to left, and the remaining pairs from left to right, then the source-sink word of the directed pattern thus obtained is precisely $w$. In this case $g(w,w')=3$.

Consider now the transformation which takes a TFPL $f$ in $\overrightarrow{T}_{u,v}^{w}$ and reorients all its bottom paths from left to right. By definition the resulting configuration belongs to $\overline{T}_{u,v}^{w'}$ for a certain $w'$ 
(which depends on $f$) which is feasible for $w$. Note that the weight is decreased by $g(w,w')$ in the transformation, so we obtain the following:
\begin{equation}
\label{nonorienttoorient}
{\overrightarrow{t}}_{u,v}^{w}(q) =
\sum_{\text{$w'$ is feasible for $w$}}
q^{g(w,w')} \,
\overline{t}_{u,v}^{w'}(q).
\end{equation}

Our goal is to invert the last relation, so that together with the help of Proposition~\ref{prop:rho} we will be able to express the number of TFPLs in terms of the weighted enumeration of oriented TFPLs.
Note that if $w'$ is feasible for $w$ then $w$ and $w'$ have the same number of $0$s (and therefore of $1$s also).

\begin{defi}[Matrix $M(N_0,N_1)$]
\label{defi:matrix}
 Given $N_0,N_1$ such that $N_0+N_1=N$, the square matrix $M=M(N_0,N_1)$  has rows and columns indexed by words with $N_0$ $0$s and $N_1$ $1$s, and the entry $M_{w,w'}$ is given by $q^{g(w,w')}$ if $w'$ is feasible for $w$, and $0$ otherwise.
\end{defi}
 
 This is a square matrix of size $\binom{N}{N_0}$.

\begin{prop}
\label{prop:invertible}
  The matrix $M(N_0,N_1)$ is invertible.
\end{prop}

\begin{proof}
It is easy to see that $w'$ is feasible for $w$ if and only if there exist ordered pairs
$(i_1,j_1)$,$(i_2,j_2),\ldots,(i_k,j_k)$ verifying $w'_{i_s}=0,w'_{j_s}=1$ and $w'_{i_{s}+1}\cdots\ w'_{j_s-1}$ is a Dyck word for any $s$, such that $w$ is given by $w_{i_s}=1,w_{j_s}=0$ for any $s$ and $w_i=w'_i$ for all other indices. Then the description of feasibility just given shows that if $w'$ is feasible for $w$, then necessarily $w'\leq w$; also clearly $q^{g(w,w)}=q^0=1$. Otherwise said, given any linear ordering on words extending $\leq$ and using that order for rows and columns of $M$, we get that $M$ is lower triangular with $1$s on the diagonal, and is thus invertible.
%The weight $g(w,\pi)$ is the number of pairs where the elements were interchanged.
\end{proof}

\begin{figure}[!ht]
\centering
\includegraphics[width=0.4\textwidth]{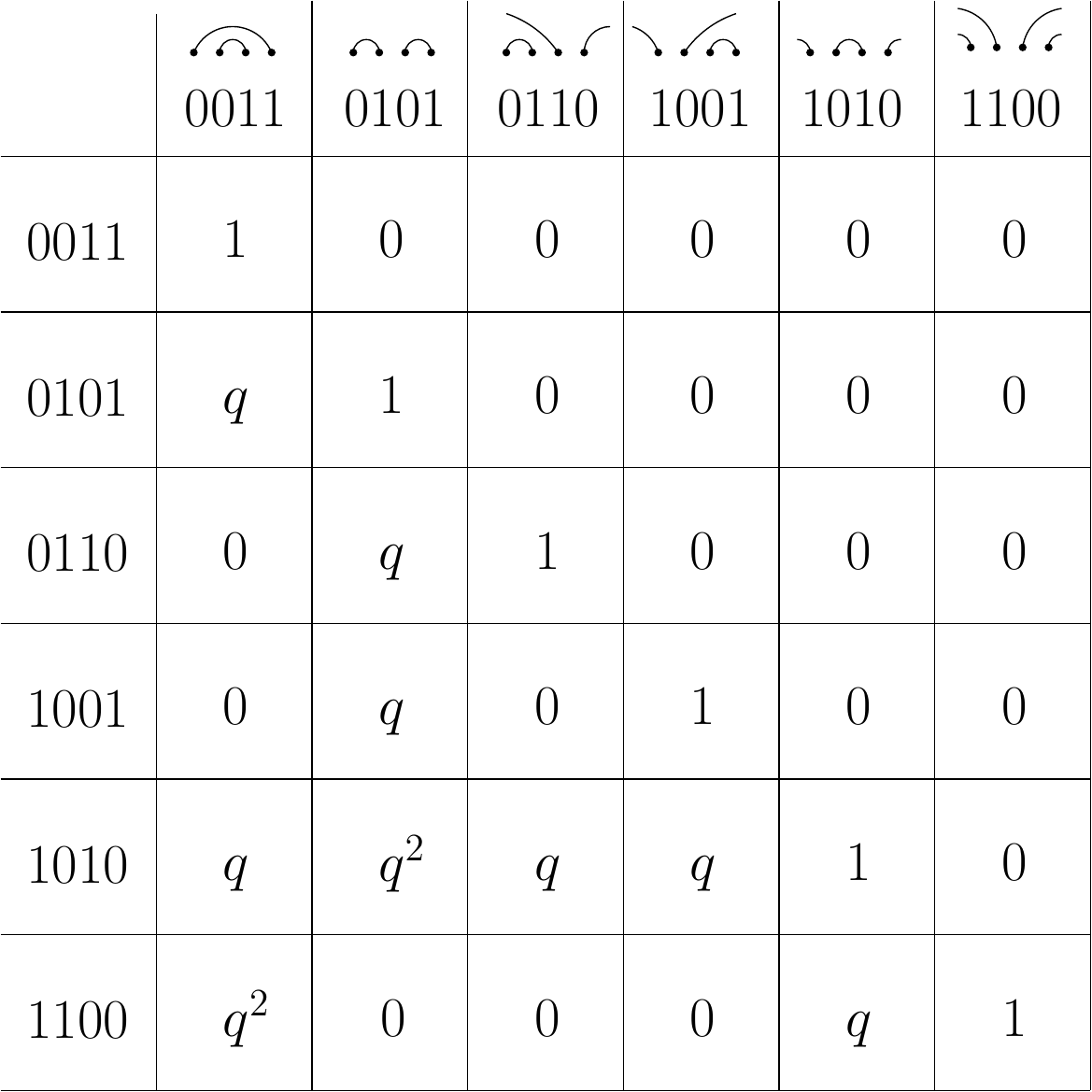}
\caption{\label{fig:Example_Matrix} The matrix $M$ for $N_0=N_1=2$.}
\end{figure}

The matrix $M$ is displayed for $N_0=N_1=2$ on Figure~\ref{fig:Example_Matrix}. From Proposition~\ref{prop:invertible}, we see that the relations~\eqref{nonorienttoorient} can be inverted:

\begin{cor}
\label{cor:otfpls_to_tfpls}
Let $u,v,w$ be three words of length $N$, and let $M=M(|w|_0,|w|_1)$. Then 
\[\overline{t}_{u,v}^{w}(q) =\sum_{w'} \left(M^{-1}\right)_{w,w'}{\overrightarrow{t}}_{u,v}^{w'}(q)\] 
and in particular 
\[
 t_{u,v}^{w} =  \sum_{w'}  \left(M^{-1}\right)_{w,w'}
{\overrightarrow{t}}_{u,v}^{w'} (\rho).
\]
\end{cor}

This expresses the number of TFPLs in terms of the number of oriented TFPLs, with the proviso that we have to keep track of the number of turns of type $R'$ and $L'$. We will use this inversion formula in Section~\ref{sec:excess_1}.

%%NEW%%
\begin{rem}
 The matrices $M(N_0,N_1)$ have made a recent appearance in the literature, in an apparently unrelated context. Indeed, for $q=1$, they are essentially given by certain upper triangular submatrices of the matrix in~\cite{Kenyon_Wilson}. In this paper, the entries $\left(M^{-1}\right)_{w,w'}$ of the inverse matrix are given a combinatorial interpretation: they enumerate (up to sign) certain ``Dyck tilings'' of the skew shape $\lambda(w)/\lambda(w')$.
  Furthermore, as is noted at the end of~\cite{Kim}, this combinatorial interpretation is easily extended to the case of general  $q$: this shows that the coefficients $\left(M^{-1}\right)_{w,w'}$ in Corollary ~\ref{cor:otfpls_to_tfpls} can be directly computed without resorting to matrix inversion.
\end{rem}

%%%%%%%%%%%%%%%%%%%%%%%%%%%%%%%%%%%%%%%%%%%%%%%%%%%%%%%%%%%%%
%%%%%%%%%%%%%%%%%%%%%%%%%%%%%%%%%%%%%%%%%%%%%%%%%%%%%%%%%%%%%

%%%%%%%%%%%%%%%%%%%%%%%%%%%%%%%%%%%%%%%%%%%%%%%%%%%
\section{Perfect matchings and nonintersecting lattice paths}
\label{sec:matchings_and_paths}
%%%%%%%%%%%%%%%%%%%%%%%%%%%%%%%%%%%%%%%%%%%%%%%%%%%

In the previous section we proved that one can express TFPLs in terms of oriented TFPLs. We focus now therefore on the latter, since they are easier to deal with as we argued in Section~\ref{sec:definitions}. In this section we will show that an oriented TFPL is essentially a pair of disjoint perfect matchings on some graphs, and that such matchings can be encoded by certain families of nonintersecting lattice paths.  
One goal is to provide easy proofs of necessary conditions on the boundary conditions of 
TFPLs (ordinary and oriented). 

%%%%%%%%%%%%%%%%%%%%%%%%%%%%%%%%
\subsection{Necessary conditions for the existence of TFPLs}
\label{sub:nec_cond}
%%%%%%%%%%%%%%%%%%%%%%%%%%%%%%%%

For a word $u$ we define $\d(u)$ as the number of \emph{inversions} in $u$, that is, the number of pairs $i<j$ such that $u_i=1$ and $u_j=0$.  Also note that $\d(u)=|\lambda(u)|$, the number of cells of the diagram  $\lambda(u)$.
We can then state the following theorem:

\begin{theo}
\label{theo:tfpl_necessary_conditions}
Let $u,v,w$ be three words of length $N$. Then $\overrightarrow{t}_{u,v}^{w}> 0$ implies the following three constraints:
\begin{enumerate}
 \item $|u|_0=|v|_0=|w|_0$;\label{it_nec1}
 \item $u\leq w$ and $v\leq w$;\label{it_nec2}
 \item $\d(u)+\d(v)\leq \d(w)$.\label{it_nec3}
\end{enumerate}
\end{theo}

Constraint~\eqref{it_nec1} can be equivalently stated as  $|u|_1=|v|_1=|w|_1$ since all words have the same length. From the inequality~\eqref{eq:embedding} we have then immediately the following corollary:

\begin{cor}
\label{cor:tfpl_necessary_conditions}
The conclusions of Theorem~\ref{theo:tfpl_necessary_conditions} hold also if ${t}_{u,v}^{w}>0$.
\end{cor}

The proofs of the different items of Theorem~\ref{theo:tfpl_necessary_conditions} will be given in Sections~\ref{sec:matchings_and_paths} and~\ref{sec:blue_red}. More precisely, part~\eqref{it_nec1} will be proven at the end of Section~\ref{sub:perfect_matchings}, part~\eqref{it_nec2} in Section~\ref{sub:invariants_matchings}, and part~\eqref{it_nec3} is Corollary~\ref{cor:inequality}.

These results were already partly known in the special case of ordinary TFPLs and where $w$ is such that $0w1$ is a Dyck word, which was the only case considered in previous papers due to its direct connection with FPLs on a square grid, as explained in Section~\ref{sub:fpls_tfpls}. Part~\eqref{it_nec2} was proven first in~\cite{CKLN} by a tedious argument, and later in~\cite{TFPL1} in a manner similar as the one given here in Section~\ref{sub:invariants_matchings}. Finally, part~\eqref{it_nec3} was proven in~\cite{Thapper} but only for non oriented TFPLs; also, the proof given there was algebraic, whereas we will show that $\d(w)- \d(u) - \d(v)$ enumerates occurrences of certain patterns in any configuration of $\overrightarrow{T}_{u,v}^{w}$ (cf. Formula~\eqref{eq:comb_interp_excess2}), which automatically proves the inequality in~\eqref{it_nec3}.

\subsection{Perfect matchings and oriented TFPLs}
\label{sub:perfect_matchings}

 For $u,w$ two words of length $N$, we define the graph $G^N_o(u,w)$ as the induced subgraph of $G^N$ obtained by removing the rightmost vertices $\RightOdd$, the vertices $B_i$ such that $w_i=0$, and the vertices $L_i$ such that $u_i=0$. We are interested in perfect matchings of $G^N_o(u,w)$: an example is given on the left of Figure~\ref{fig:matchings}. Given an edge in such a perfect matching $M$, orient it from its odd vertex to its even vertex: then the corresponding direction is up, down, left or right. Denote the respective corresponding sets of edges by  $\Oddu(M),\Oddd(M),\Oddl(M),\Oddr(M)$  and their cardinalities by $\oddu(M),\oddd(M),\oddl(M),\oddr(M)$.

We define similarly a graph $G^N_e(v,w)$ as follows: start from $G^N$ and remove the leftmost vertices $\LeftOdd$, the vertices $B_i,R_i$ for which $w_i=1$ and $v_i=1$ respectively. Given an edge in a perfect matching $M$ of $G^N_e(v,w)$, orient it from its even vertex to its odd vertex: then the corresponding direction is up, down, left or right. Denote the respective corresponding sets of edges by  $\Evenu(M),\Evend(M),\Evenl(M),\Evenr(M)$  and their cardinalities by $\evenu(M),\evend(M),\evenl(M),\evenr(M)$.

\begin{figure}[!ht]
 \begin{center}
 \includegraphics[width=\textwidth]{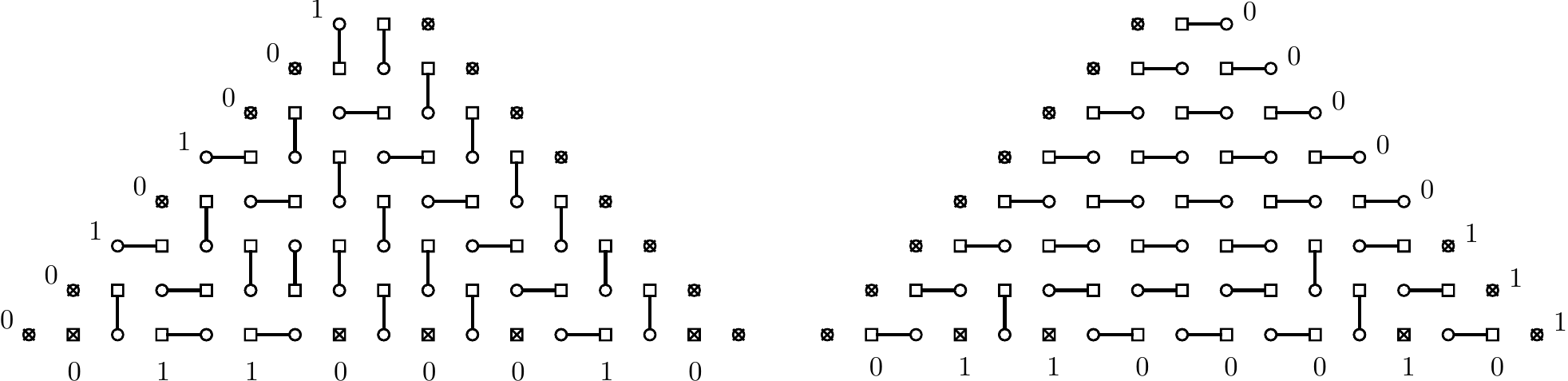}  
\caption{Perfect matchings on $G^8_o(00101001,01100010)$ and $G^8_e(00000111,01100010)$.\label{fig:matchings}}
 \end{center}
\end{figure}

The introduction of these graphs and matchings is motivated by the following result:

\begin{theo}
\label{theo:tfpl_to_matchings} 
 Let $N$ be a positive integer, $u,v,w$ be three words of length $N$. For any oriented TFPL configuration $f$ in $\overrightarrow{T}_{u,v}^{w}$, denote by $M_o(f)$ (respectively $M_e(f)$) the subset of its edges oriented from an odd vertex to an even vertex (resp. from an even vertex to an odd vertex).
 
  Then $f\mapsto\left(M_o(f),M_e(f)\right)$ is a bijection between:
\begin{enumerate}
 \item oriented TFPL configurations in $\overrightarrow{T}_{u,v}^{w}$, and
 \item ordered pairs $(M_o,M_e)$ where $M_o$ (\emph{resp.} $M_e$) is a perfect matching on $G^N_o(u,w)$ (\emph{resp.} on  $G^N_e(v,w)$), and such that $M_o$ and $M_e$ are \emph{disjoint} as subsets of edges of the graph $G^N$.
\end{enumerate}
\end{theo}

\begin{proof}
  In an oriented TFPL, all  indegrees and outdegrees are at most one; from this observation it follows that $M_o(f)$ and $M_e(f)$ form matchings on $G^{N}$. The vertices of $G^N$ which do not belong to any edge of $M_o(f)$ are those which are either odd with outdegree $0$ or even with indegree $0$. By definition of oriented TFPLs, these are exactly the vertices removed when passing from $G^{N}$ to $G^{N}_{o}(u,w)$. Therefore $M_o(f)$ is indeed a perfect matching on $G^{N}_{o}(u,w)$; by symmetry, one has that $M_e(f)$ is a perfect matching on $G^{N}_{e}(v,w)$. In an oriented TFPL, one cannot have two directed edges corresponding to the two possible orientations of the same underlying nonoriented edge; this implies immediately that $M_o(f)$ and $M_e(f)$ are disjoint matchings, and therefore the map $f\mapsto\left(M_o(f),M_e(f)\right)$ is well-defined between the sets described in $(1)$ and $(2)$.

 It is then immediate to verify that it is indeed a bijection between these two sets; its inverse consists in orienting edges of $M_o$ from odd vertices to even ones, edges of $M_e$ from even vertices to odd ones, and then considering the union of these two sets of edges on the full graph $G^{N}$.
\end{proof}

We can already prove the first part of Theorem~\ref{theo:tfpl_necessary_conditions}:

\begin{proof}[Proof of~\eqref{it_nec1} in Theorem~\ref{theo:tfpl_necessary_conditions}]
 Consider $f\in \overrightarrow{T}_{u,v}^w$; by Theorem~\ref{theo:tfpl_to_matchings} it gives rise to a perfect matching on $G^{N}_{o}(u,w)$. Now perfect matchings on a bipartite graph can only exist if there are the same number of even and odd vertices. A quick computation shows that this is the case for $G^{N}_{o}(u,w)$ if and only if $|u|_0=|w|_0$. 

Now $f$ also gives rise to a perfect matching on $G^{N}_{e}(v,w)$, which implies $|v|_1=|w|_1$ by symmetry. Since $v$ and $w$ have the same length, this is equivalent to $|v|_0=|w|_0$ and concludes the proof.
\end{proof}

From now on, we will assume that $|u|_0=|v|_0=|w|_0(=:N_0)$ and $|u|_1=|v|_1=|w|_1(=:N_1)$, since we have just proven that $\overrightarrow{T}_{u,v}^w$ is empty unless these conditions are fulfilled.

\subsection{Invariants of perfect matchings}
\label{sub:invariants_matchings}

The next theorem shows that in the perfect matchings under consideration, certain enumerations of edges only depend on the graph $G^N_o(u,w)$ and not the matching itself. The proof will rely on the following lemma.

\begin{lem} For any word $u$ with $|u|_1=N_1$ and $|u|_0=N_0$, one has
\label{lem:wordlem}
\begin{align}
  &\sum_{i=1}^N(N-i)u_i=\d(u)+N_1(N_1-1)/2; \label{id:wordlem1} \\
  &\sum_{i=1}^Niu_i=N_1(2N-N_1+1)/2-\d(u) \label{id:wordlem2}.
 \end{align}
\end{lem}

\begin{proof}
The first equality is deduced from the computation:
\[
      \d(u)=\sum_{i<j} u_i(1-u_j)=\sum_{i<j} u_i- \sum_{i<j} u_iu_j=\sum_i(N-i)u_i-N_1(N_1-1)/2.
\]
The second equality comes simply from the computation $\sum_i(N-i)u_i+\sum_iiu_i=N\sum_iu_i=NN_1.$
\end{proof}

\begin{theo}
\label{theo:identities_Go}
 We have the following identities for any perfect matching $M$ of $G^N_o(u,w)$:
\begin{align}
\oddu(M)+\oddd(M)+\oddr(M)+\oddl(M)&= N(N-1)/2+N_1; \label{o1} \\ 
\oddl(M)+\oddd(M) &=  \d(w)-\d(u); \label{o2}  \\ 
\oddu(M)+\oddl(M) &= N_0(N_0-1)/2  + \d(w). \label{o3} 
\end{align}
Similarly, we have the following identities for any perfect matching $M$ of $G^N_e(v,w)$:
\begin{align}
\evenu(M)+\evend(M)+\evenr(M)+\evenl(M)&= N(N-1)/2+N_0; \label{e1} \\ 
\evenl(M)+\evenu(M) &=  \d(w)-\d(v); \label{e2}  \\ 
\evend(M)+\evenl(M) &=  N_1(N_1-1)/2   + \d(w); \label{e3} 
\end{align}
\end{theo}

\begin{proof}
In Formula~\eqref{o1}, the left-hand side counts the total number of edges in a perfect matching of the graph $G^N_o(u,w)$; this is the number of odd (or even) vertices in $G^N_o(u,w)$, which is easily seen to be $N(N-1)/2+N_1$.

For Formula~\eqref{o2}, consider the NW-SE diagonals in $G^N_o(u,w)$; their number of vertices is given, from left to right, by
\[u_1,w_1;1+u_2,1+w_2;\ldots;i+u_{i+1},i+w_{i+1};\ldots;N-1+u_{N},N-1+w_{N}.\]
 The edges of any perfect matching $M$ of $G^N_o(u,w)$ connect two consecutive diagonals, and we denote the numbers of these edges by $x_1,y_1,x_2,y_2,\ldots,x_{N}$ from left to right. Since $M$ is perfect, one has the obvious equations $x_i+y_i=w_i+(i-1), i=1\ldots N$, and $y_i+x_{i+1}=u_{i+1}+i,i=1\ldots N-1$. Together with the initial condition $x_1=u_1$, the unique solution to these equations is
\[ y_i=\sum_{j\leq i} (w_j-u_j)\quad\text{and}\quad x_{i+1}=u_{i+1}+i-y_i\quad\text{for}\quad i=1,\ldots,N-1.\]

Now, by the first identity in Lemma~\ref{lem:wordlem}, one has \[\oddl(M)+\oddd(M)=\sum_{i}y_i=\sum_{i}(\sum_{j\leq i} (w_j-u_j))=\d(w)-\d(u),\]
 which proves~\eqref{o2}.
The proof of Formula~\eqref{o3} is similar and uses the NE-SW diagonals. We leave it to the reader, and will give another proof in Section~\ref{sub:path_configurations}.
% We use a similar argument for Formula~\eqref{o3}, but now consider the NE-SW diagonals, which from right to left possess 
% \[w_{N},1;w_{N-1}+1,2;\ldots;w_{N-i+1}+i-1,i;\ldots;w_1+N-1,N_1\]
% points. Letting the number of edges between two diagonals be $x'_1,y'_1,x'_2,\ldots,x'_{N}$, we obtain this time the equations
% $x'_i+y'_i=i, i=1\ldots N$ and $y'_i+x'_{i+1}=w_{N-i}+i$ for $i=1,\ldots,N$. Together with the initial condition $x'_1=w_N$, the unique solution to these equations is
% 
% \[ x'_{i}=w_{N-i+1}+\cdots+w_{N}=\sum_{j=N-i+1}^{N} w_{j}\quad\text{for}\quad i=1,\ldots,N\]
% so that, by the second identity in Lemma~\ref{lem:wordlem}, 
% \begin{align*}
% \oddd+\oddr=\sum_{i=1}^{N}x'_i&=N_1(2N-N_1+1)/2-\d(w).
% \end{align*}

To prove the remaining identities~\eqref{e1},\eqref{e2},\eqref{e3}, note that $G^N_e(v,w)$ is isomorphic to $G^N_o(v^*,w^*)$ via a simple vertical 
reflection and the reorientation of edges. This induces a bijective correspondence $M\mapsto M'$ from perfect matchings on $G^N_e(v,w)$ to those on  $G^N_o(v^*,w^*)$. In this correspondence one checks immediately $\oddd(M')=\evenu(M),\oddu(M')=\evend(M),\oddl(M')=\evenl(M),\oddr(M')=\evenr(M)$.
\end{proof}

\begin{proof}[Proof of~\eqref{it_nec2} in Theorem~\ref{theo:tfpl_necessary_conditions}]
 The proof of Formula~\eqref{o2} has as a byproduct that $u\leq w$:  indeed for any matching $M$ on $G^N_o(u,w)$, we have for all $i$ that $y_i \ge 0$ since it enumerates edges, so that $\sum_{j\leq i} (w_j-u_j)\geq 0$ for all $i$, which means precisely that $u\leq w$. Now we know that from any TFPL in with boundary conditions $(u,v;w)$ one constructs a matching on $G^N_o(u,w)$, so that $\overrightarrow{t}_{u,v}^{w}> 0$ implies $u\leq w$.
 By the vertical symmetry of TFPLs (Proposition~\ref{prop:vertical_symmetry_oriented}), we get that $v^*\leq w^*$, which is equivalent to $v\leq w$ and achieves the proof.
\end{proof}

\begin{rem}
There is a well-known necessary and sufficient condition for a finite  bipartite graph to admit a perfect matching, namely Hall's marriage condition~\cite{Hall};  this says that for every subset $S$ of odd vertices, the set $T$ of even vertices adjacent to at least one element of $S$ must verify $|T|\geq |S|$. Applied to the first $i$ odd NW-SE diagonals in $G^N_o(u,w)$, this means that $u_1+\ldots+u_i\leq w_1+\ldots+w_i$, for all $i$. Thus we obtain $u\leq w$, so we get another proof of Theorem~\ref{theo:tfpl_necessary_conditions}~\eqref{it_nec2}. The extra information that we get from Formula~\eqref{o2} is a combinatorial interpretation of $\d(w)-\d(u)$ in each matching $M$; notice that in the Ferrers diagram encoding, $\d(w)-\d(u)$ is the number of cells belonging to $\lambda(w) / \lambda(u)$. This interpretation will be used in the proof of item~\eqref{it_nec3} of Theorem~\ref{theo:tfpl_necessary_conditions} in Section~\ref{sec:blue_red}.
\end{rem}

\subsection{From matchings to nonintersecting paths}
\label{sub:path_configurations}

We assume that $u,w$ are given, and we consider the graph $G^N_o(u,w)$; we will describe the classic bijection between perfect matchings on such graphs and certain configurations of paths.
%~\footnote{reference - mihai ciucu couldn't think of any. he told me that he learned about it on the domino-list in the 1990s from dana randall} 

Firstly, add a new set of vertices, the {\em blue vertices}, in the middle of each horizontal edge of $G^N$ which has an odd vertex to its left. Now given a perfect matching $M$, we will construct certain blue lattice paths on these blue vertices: Let $K$ be an edge of $M$, then we perform the following transformations (the reader is advised to look at figure~\ref{fig:oddmatchingstopaths}):
\begin{itemize}
\item if $K\in \Oddd(M)$, join the blue vertices which are to the right of its top vertex and to the left of its bottom vertex;
\item if $K\in \Oddu(M)$, join the blue vertices which are to the right of its bottom vertex and to the left of its top vertex;
\item if $K\in \Oddl(M)$, join the blue vertices which are to its right and to its left;
\item if $K\in \Oddr(M)$, do nothing.
\end{itemize}

\begin{figure}[!ht]
 \begin{center}
 \includegraphics[width=0.5\textwidth]{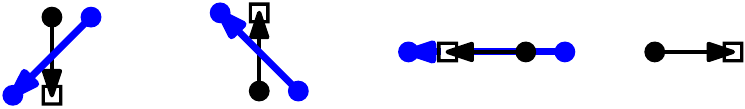}  
\caption{From perfect matchings on $G^N_o(u,w)$ to nonintersecting paths.\label{fig:oddmatchingstopaths}}
 \end{center}
\end{figure}

 Define $I_w=\{1\leq i_1<\ldots<i_{N_0}\leq N\}$ as the set of indices $i$ such that $w_i=0$, and define similarly $I_u=\{1 \le j_1<\ldots<j_{N_0} \le N\}$. Let $D_{k}=(2i_k-2,0)$  and $E_{l}=(j_l-1,j_l-1)$ for $1\leq k,l\leq N_0$. Let $\mathcal{P}(D_{k},E_{l})$ be the set of paths from $D_{k}$ to $E_{l}$ using steps $(-1,1),(-1,-1),(-2,0)$ which never go below the $x$-axis.
 
\begin{prop}
\label{prop:oddmatchingstopaths}
The correspondence defined above is a bijection between:
\begin{enumerate} 
 \item Perfect matchings of $G^N_o(u,w)$, and
 \item Nonintersecting paths $(P_1,P_2,\ldots,P_{N_0})$ with $P_k\in \mathcal{P}(D_{k},E_{k})$.
\end{enumerate}
\end{prop}

An example is provided in Figure~\ref{fig:matchingstopaths}, left.

\begin{figure}[!ht]
 \begin{center}
 \includegraphics[width=0.9\textwidth]{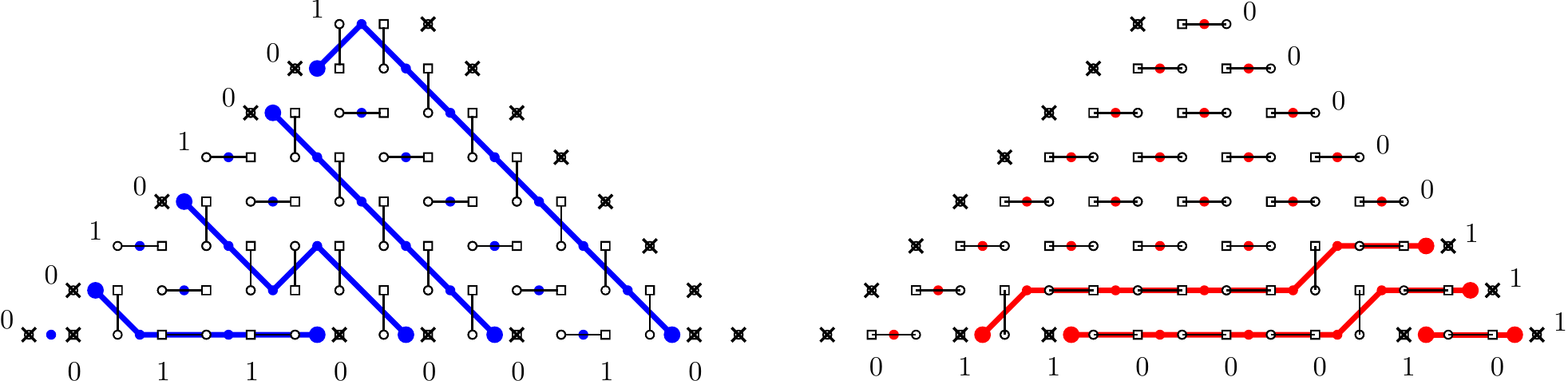}  
\caption{From matchings to nonintersecting paths.\label{fig:matchingstopaths}}
 \end{center}
\end{figure}

Now consider the graph $G^N_e(v,w)$ for word $v,w$. We insert {\em red vertices}, in the middle of each horizontal edge of $G^{N}$ which has an odd vertex to its right. Given a perfect matching $M'$ of $G^N_e(v,w)$ and $K'$ an edge of $M'$, we perform the following:
\begin{itemize}
\item if $K'\in \Evend(M')$, join the red vertices which are to the left of its bottom vertex and to the right of its top vertex;
\item if $K'\in \Evenu(M')$, join the red vertices which are to the left of its top vertex and to the right of its bottom vertex;
\item if $K'\in \Evenl(M')$, join the red vertices which are to its left and to its right;
\item if $K'\in \Evenr(M')$, do nothing.
\end{itemize}

\begin{figure}[!ht]
 \begin{center}
 \includegraphics[width=0.5\textwidth]{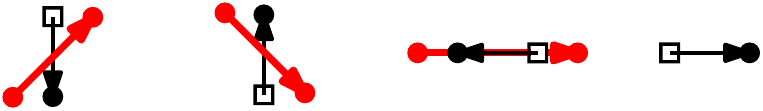}  
\caption{From even matchings to path configurations.\label{fig:evenmatchingstopaths}}
 \end{center}
\end{figure}

 Define $I'_w=\{1 \le i'_1<\ldots<i'_{N_1} \le N\}$ as the set of indices $i$ such that $w_i=1$, and define similarly $I'_v=\{1 \le j'_1<\ldots<j'_{N_1} \le N\}$. Let $D'_{k}=(2i'_k-1,0)$  and $E'_{l}=(N-1+j'_l,N-j'_l)$ for $1\leq k,l\leq N_1$. Let $\mathcal{P}'(D'_{k},E'_{l})$ be the set of paths from $D'_{k}$ to $E'_{l}$ using steps $(1,1),(1,-1),(2,0)$ which never go below the $x$-axis.
\begin{prop}
\label{prop:evenmatchingstopaths}
The correspondence defined above is a bijection between:
\begin{enumerate} 
 \item Perfect matchings of $G^N_e(v,w)$, and
 \item Nonintersecting paths $(P_1,P_2,\ldots,P_{N_1})$ of paths with $P_k\in \mathcal{P}'(D'_{k},E'_{k})$.
\end{enumerate}
\end{prop}

An example is provided in Figure~\ref{fig:matchingstopaths}, right.\medskip
 
\noindent{\bf Application:} Let us show that these nonintersecting paths permit us to give an easy proof of Formula~\eqref{o3}, as announced in the proof of Theorem~\ref{theo:identities_Go}. 
Any TFPL in $\overrightarrow{T}_{u,v}^w$ gives rise to a matching $M$ on $G^N_o(u,w)$, which is encoded by a configuration $(P_1,P_2,\ldots,P_{N_0})$ of nonintersecting paths with $P_k\in \mathcal{P}(D_{k},E_{k})$ by Proposition~\ref{prop:oddmatchingstopaths}. Let us compute the dot product \[A:=\sum_k\overrightarrow{D_{k}E_{k}}\cdot(-1/2,1/2)\] in two ways: On the one hand, $\overrightarrow{D_{k}E_{k}}\cdot (-1/2,1/2)=i_k-1$, so that, 
by~\eqref{id:wordlem1}, $A=\sum_k{(i_k-1)}=N_0(N_0-1)/2+\d(w)$. On the other hand, by decomposing $\overrightarrow{D_{k}E_{k}}$ as the sum of the steps of $P_k$, one sees that $\overrightarrow{D_{k}E_{k}}\cdot (-1/2,1/2)$ is the total number of $(-1,1)$ and $(-2,0)$ steps in $P_k$. Therefore $A$ is the total number of such steps in the configuration, which is equal to $\oddu(M)+\oddl(M)$ by the transformations of Figure~\ref{fig:oddmatchingstopaths}; this achieves the proof of Formula~\eqref{o3}.

A dot product with $(1/2,1/2)$ gives a new proof of~\eqref{o2} in a similar fashion. As we will see in Section~\ref{sec:blue_red}, the advantage of working with nonintersecting paths will also be apparent when we superimpose blue and red paths, creating what we call {\em path tangles}.
 
\subsection{The number of perfect matchings in $G^N_o(u,w)$ and $G^N_e(v,w)$}
The knowledgeable reader may have noticed that our nonintersecting paths correspond to an instance of the Gessel-Viennot lemma~\cite{GV,GV1}, which permits us to count the number of perfect matchings in $G^N_o(u,w)$, \emph{resp.} in $G^N_e(v,w)$. In order to apply their lemma we need to count 
prefixes of Schr\"oder paths.

\begin{lem}
\label{lem:schroeder}
The number of paths from $(0,0)$ to $(2n+m,m)$ with steps of type 
$(1,1), (1,-1), (2,0)$ which never go below the $x$-axis is equal to
\[
\sum_{p=0}^{n} \left( \binom{2n-2p+m}{n-p} - \binom{2n-2p+m}{n-p-1} \right)  \binom{2n+m-p}{p} 
\]
\end{lem}
 
\begin{proof}
The number of paths from $(0,0)$ to $(2n+m,m)$ with steps of type 
$(1,1), (1,-1), (2,0)$ and $p$ horizontal steps is equal to the number of 
paths from $(0,0)$ to $(2n-2p+m,m)$ with steps of type $(1,1), (1,-1)$ multiplied by the number of $p$-subsets of $\{0,1,\ldots,2n+m-1\}$
containing no consecutive integers, hence equal to 
\[
\binom{2n-2p+m}{n-p} \binom{2n+m-p}{p}.
\]
The total number $t(n,m)$ of paths from $(0,0)$ to $(2n+m,m)$ with steps of type $(1,1), (1,-1), (2,0)$ is obtained by summing this over all $p$ between $0$ and $n$. 
The number of paths that go below the $x$-axis is by the reflection principle ($y=-1$) equal to the number of paths from $(0,-2)$ to 
$(2n+m,m)$ with steps of type $(1,1), (1,-1), (2,0)$. Thus, the number of paths that never go below the $x$-axis is equal to 
$t(n,m) - t(n-1,m+2)$. 
\end{proof}

\begin{prop}
\label{prop:count_matchings}
The number of perfect matchings of $G^N_o(u,w)$ is given by 
\[
\det_{1 \le k, l \le N_0} 
\left( \sum_{p=0}^{i_k-j_l} \left( \binom{2 i_k - j_l - 2 p -1}{i_k-j_l-p} -
\binom{2 i_k - j_l - 2p -1}{i_k - j_l -p -1} \right) \binom{2 i_k - j_l -p -1}{p}  \right),
\]
while the number of perfect matchings of $G^N_e(v,w)$ is given by 
\[
 \det_{1 \le k, l \le N_1} 
\left( \sum_{p=0}^{j_l-i_k} \left( \binom{j_l-2 i_k - 2 p + N +1}{j_l-i_k-p} - \binom{j_l - 2 i_k - 2 p + N +1 }{j_l-i_k-p-1} \right) \binom{j_l - 2 i_k + N +1 -p}{p} \right).
\]
\end{prop}

\begin{proof}
Let $M(u,w)$ be the $N_0\times N_0$ matrix with entries $\left|\mathcal{P}(D_{k},E_{l})\right|$, and $M'(v,w)$ be the $N_1\times N_1$ matrix with entries $\left|\mathcal{P'}(D'_{k},E'_{l})\right|$. Then from~\cite{GV} we know that the numbers of perfect matchings in $G^N_o(u,w)$ and $G^N_e(v,w)$ is given by $\det\left(M(u,w)\right)$ and $\det\left(M'(v,w)\right)$, respectively. Now the paths in $\mathcal{P}(D_{k},E_{l})$ are suffixes of Schr\"{o}der paths, 
while the paths in $\mathcal{P'}(D'_{k},E'_{l})$ are prefixes of Schr\"oder paths. Both can be counted using Lemma~\ref{lem:schroeder}, which proves the explicit form of the coefficients of the matrix.
\end{proof}

Since we showed that oriented TFPLs correspond to disjoint matchings, taking the product of the two formulas of Proposition~\ref{prop:count_matchings} only gives an upper bound for the numbers $\overrightarrow{t}_{u,v}^w$; this is in general a very poor bound, the constraint that the matchings are disjoint being in general hard to fulfill.

\begin{figure}[!ht]
 \begin{center}
 \includegraphics[height=4cm]{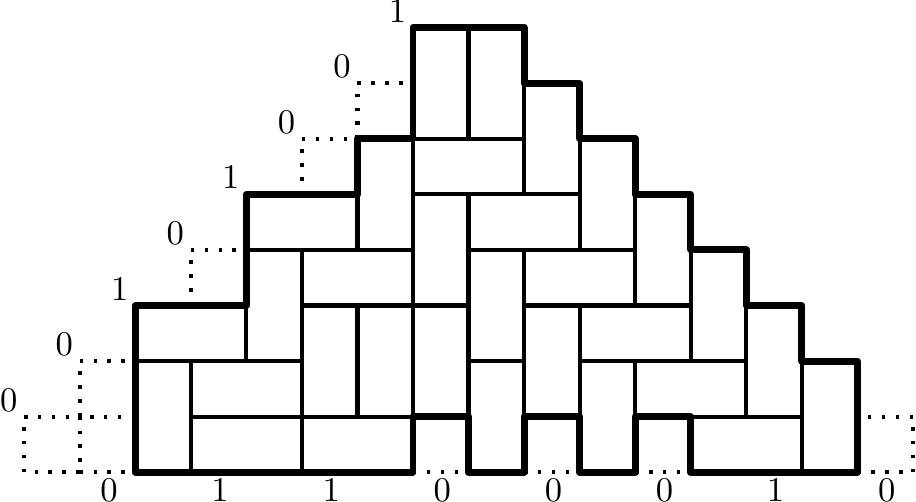}  
\caption{Domino tiling of a dented half Aztec diamond that corresponds to the matching in Figure~\ref{fig:matchings}, left.
\label{fig:aztec}}
 \end{center}
\end{figure}

Interestingly, these formulas count certain domino tilings of a half Aztec diamond~\cite{Aztec1}. To obtain the Aztec diamond of a given odd matching, we replace each vertex of $G^N_o(u,w)$ by a unit square and introduce a 
domino covering of two squares if the corresponding vertices are connected by a matching edge, see Figure~\ref{fig:aztec}. The region to be tiled is the upper half of an Aztec diamond, from which all unit squares corresponding to the occurrences of $0$ in $u$ and $w$ were removed. It may be interesting to find out if some particular choices of $u$ and $w$ give interesting enumerations, although we have not pursued this line of research.

%%%%%%%%%%%%%%%%%%%%%%%%%%%%%%%%%%%%%%%%%%%%%%%%%%%%%%%%%%%%%
%%%%%%%%%%%%%%%%%%%%%%%%%%%%%%%%%%%%%%%%%%%%%%%%%%%%%%%%%%%%%

%%%%%%%%%%%%%%%%%%%%%%%%%%%%%%%%%%%%%%%%%%%%%%%%%%%
\section{Path tangles}
\label{sec:blue_red}
%%%%%%%%%%%%%%%%%%%%%%%%%%%%%%%%%%%%%%%%%%%%%%%%%%%

\subsection{Path Tangles}

Let $u,v,w$ be three words which each possess $N_0$ $0$s and $N_1$ $1$s. Consider two perfect matchings $M,M'$ on $G^N_o(u,w)$ and $G^N_e(v,w)$ respectively. By Theorem~\ref{theo:tfpl_to_matchings}, 
 they give rise to an oriented TFPL if they are \emph{disjoint} subsets of $G^N$.  We want to translate this in terms of the nonintersecting paths of Propositions~\ref{prop:oddmatchingstopaths} and~\ref{prop:evenmatchingstopaths}. 

Assume $K$ is an edge appearing in both matchings $M$ and $M'$: we study what happens when the rules illustrated in Figures~\ref{fig:oddmatchingstopaths} and ~\ref{fig:evenmatchingstopaths} are applied:

\begin{itemize}
 \item \emph{$K$ is vertical}: this occurs when $K$ belongs to $\Oddd(M)\cap \Evenu(M')$ (\emph{resp.} $\Oddu(M)\cap \Evend(M')$); in this case $K$ gives rise to two down steps (\emph{resp.} two up steps), one blue and one red, crossing in their midpoints; 
\item \emph{$K$ is horizontal}: this occurs when $K$ belongs to $\Oddl(M)\cap \Evenr(M')$  (\emph{resp.} $\Oddr(M)\cap \Evenl(M')$); in this case $K$ gives rise to a blue (\emph{resp.} red) horizontal step whose midpoint is not part of any red (\emph{resp.} blue) step;
\end{itemize}

\begin{figure}[!ht]
 \begin{center}
 \includegraphics{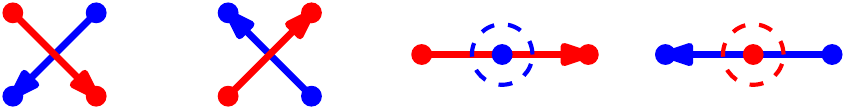}  
\caption{Forbidden local configurations in the blue-red paths. The little dashed circle indicates that no path uses that vertex.\label{fig:forbidden_configurations}}
 \end{center}
\end{figure}

So if the two matchings are disjoint, they will avoid the local pattens of Figure~\ref{fig:forbidden_configurations}; conversely, the absence of such patterns give rise to disjoint matchings. We have thus obtained the following theorem:

\begin{theo}
\label{theo:path_model_characterization}
 For any words $u,v,w$ of length $N$, denote by $\mathcal{P}(u,w)$ (\emph{resp.} $\mathcal{P}'(v,w)$) the set of families of nonintersecting paths considered in 
Propositions~\ref{prop:oddmatchingstopaths} and~\ref{prop:evenmatchingstopaths}, respectively.

The set of oriented TFPLs $\overrightarrow{T}_{u,v}^w$ is in bijection with the set of pairs $(B,R)\in \mathcal{P}(u,v)\times\mathcal{P}'(v,w)$ that verify the two following conditions:
\begin{enumerate}
 \item \label{it:path_model_characterization_1} No diagonal step of $R$ can cross a diagonal step of $B$.
 \item \label{it:path_model_characterization_2} Each middle point of a blue (\emph{resp.} red) horizontal step is used by a red (\emph{resp.} blue) step.
\end{enumerate}
We denote by $\Pathconfig(u,v;w)$ the set of such configurations, which we will call \emph{(blue-red) path tangles}.
\end{theo}

 The path tangle associated with the oriented TFPL of Figure~\ref{fig:Example_oTFPL} is represented in Figure~\ref{fig:Tangle}. Notice how the starting points of blue and red paths on the bottom are intertwined according to the word $w$.

\begin{figure}[!ht]
 \begin{center}
 \includegraphics{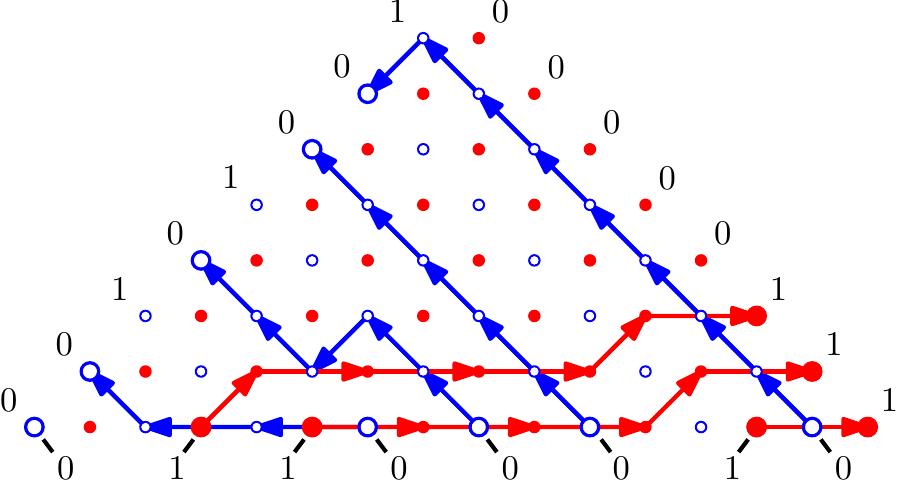}  
\caption{Blue-red path tangle corresponding to the oriented TFPL of Figure~\ref{fig:Example_oTFPL} .\label{fig:Tangle}}
 \end{center}
\end{figure}

\begin{figure}[!ht]
 \begin{center}
 \includegraphics[width=0.8\textwidth]{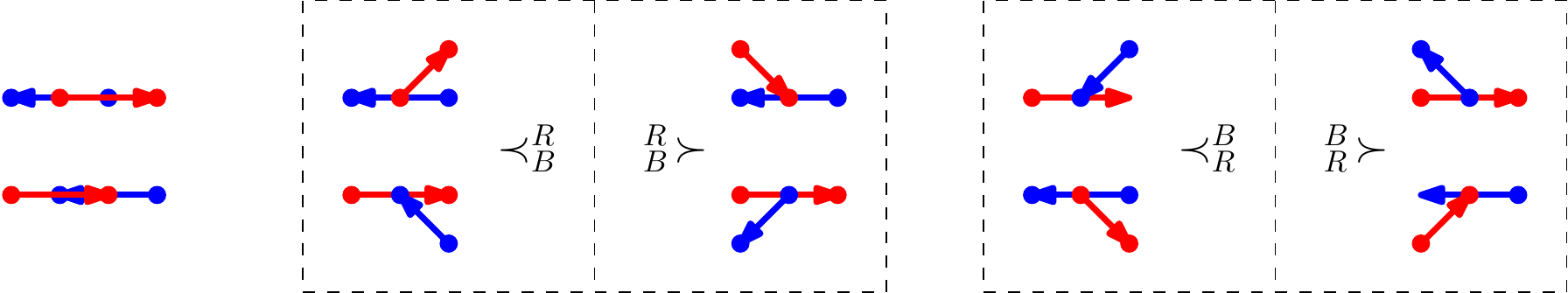}  
\caption{All possible intersections of two steps in a path tangle.\label{fig:intersecting_steps}}
 \end{center}
\end{figure}

{\noindent \bf Weight:} Corollary~\ref{cor:otfpls_to_tfpls} shows that in order to count non-oriented TFPLs, we need to be able to count oriented TFPLs weighted according to Equation~\eqref{eq:weighted_tfpl}. Since we are now going to work with the blue-red path tangles, we need to translate $t_\circlearrowright$ and $t_\circlearrowleft$ turns into this model, when $t_\circlearrowright\in R'$ and $t_\circlearrowleft\in L'$.

\begin{prop}
\label{prop:brweight}
If $f$ is an oriented TFPL with associated path tangle $C$, then the quantity 
$t_\circlearrowleft(f) -t_\circlearrowright(f)$ is given by  
\[ 
\alpha (\ldo + \lde) + (1- \alpha)(\ulo + \ule) -
\beta (\dlo+\dle) - (1-\beta)(\luo +  \lue),
\] 
in the tangle $C$, where $\dlo$ denotes the number of local configurations in $C$ of type $\dlo$, etc., and $\alpha, \beta \in \mathbb{R}$.
\end{prop}

\begin{proof}
By definition $t_\circlearrowright(f)$ counts the occurrences of the turns $dl$ or $lu$ in $f$, while $t_\circlearrowleft(f)$ counts turns of type $ld$ or $ul$. We translate these into the language of path tangles 
and obtain: 
\begin{align*}
& ld =  \evenleftdown + \oddleftup  =  \ldo + \lde \\    
& ul =  \evenupleft + \oddupleft =  \ulo +  \ule \\
& dl =  \evendownleft + \odddownleft =  \dlo + \dle \\
& lu =  \evenleftup + \oddleftup =  \luo +  \lue 
\end{align*}
\end{proof}

\subsection{Combinatorial interpretation of $\d(w)-\d(v)-\d(u)$}

We come to the main result of this section:

\begin{theo}
\label{theo:excessformula}
For any TFPL configuration in $\overrightarrow{T}_{u,v}^{w}$, one has the two formulas:
\begin{align}
\label{eq:comb_interp_prelim2} 
\evenleft + \oddleft - \d(w)  &= \evenleftleft + \oddleftleft  +\oddleftup +
 \evenupleft +
   \odddownleft  + \evenleftdown \text{, and} \\
 \label{eq:comb_interp_excess2} 
  \d(w) - \d(u) - \d(v) &=    \odddown + \evenup +  \evenleftleft + \oddleftleft  +\oddleftup +
 \evenupleft + \odddownleft  + \evenleftdown 
\end{align} 
Equivalently, for any path tangle in $\Pathconfig(u,v;w)$, one has
\begin{align}
\label{eq:comb_interp_prelim} \blueh + \redh - \d(w) &=  \raisebox{-0.4cm}{\includegraphics{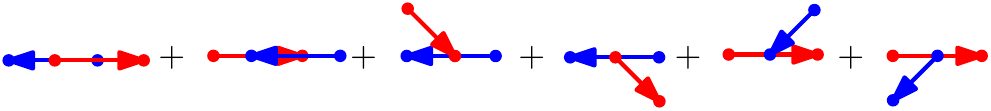}} \text{, and} \\
\label{eq:comb_interp_excess}  \d(w) - \d(u) - \d(v) &=  \raisebox{-0.4cm}{\includegraphics{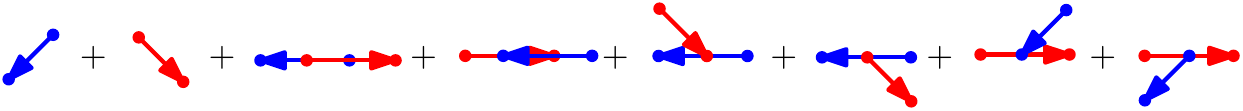}}.
 \end{align} 
\end{theo}

The equivalence of the two pairs of formulas clearly follows from the bijection between 
oriented TFPLs in $\overrightarrow{T}_{u,v}^{w}$ and path tangles in 
$\Pathconfig(u,v;w)$ (Theorem~\ref{theo:path_model_characterization}). 
See also Figure~\ref{fig:oddmatchingstopaths} and Figure~\ref{fig:evenmatchingstopaths}.
As an immediate consequence we obtain the following corollary, which achieves the proof of Theorem~\ref{theo:tfpl_necessary_conditions} by proving its part~\eqref{it_nec3}:

\begin{cor}
\label{cor:inequality}
 If $u,v,w$ are such that $\overrightarrow{t}_{u,v}^{w}\neq 0$, then $\d(w) \ge \d(u) + \d(v)$.
\end{cor}

\begin{proof}
Indeed $\overrightarrow{t}_{u,v}^{w}\neq 0$ means that $\overrightarrow{T}_{u,v}^{w}$ has at least one element $f$, so the second identity in Theorem~\ref{theo:excessformula} shows that $\d(w) - \d(u) - \d(v)$ counts certain local 
configurations in $f$ and therefore is nonnegative.
\end{proof}

The following notion is essential in the proof of the theorem.

\begin{defi}[Intersecting pairs of paths]
 In a path tangle, we say that a pair $(b,r)$ consisting of a blue path $b$ and of a red path $r$ is \emph{intersecting} if $b$ and $r$ intersect at least once.
\end{defi}

\begin{lem}
\label{lem:intersecting_paths}
 For any blue-red path tangle in $\Pathconfig(u,v;w)$, the set of its intersecting pairs of paths is in bijection with the set of inversions of $w$. 
\end{lem}

\begin{proof}[Proof of Lemma~\ref{lem:intersecting_paths}]
Observe that $i<j$ is an inversion in $w$ if and only if the vertex $B_i$ is a starting point of a red path and $B_j$ is a starting point of a blue path. Since $B_i$ is to the left of $B_j$, and red paths go right while blue paths go left, these two paths must intersect. On the other hand, suppose we have an intersecting pair, and let $B_i$ be 
the starting point the red path and $B_j$ be the starting point of the blue path. For reasons that where given before, $B_i$ must be left of $B_j$. This implies that $i < j$ is an inversion in $w$.
\end{proof}

The following lemma will be fundamental for the proof of Theorem~\ref{theo:excessformula}.

\begin{lem} 
\label{lem:twow}
In any oriented TFPL in $\overrightarrow{T}_{u,v}^{w}$, respectively path 
tangle in $\Pathconfig(u,v;w)$, we have 
\begin{align*}
\d(w) &= \oddupleft + \evendownleft - \evenleftdown - \oddleftup = \ule + \dlo - \ldo - \lue 
\text{, and}   \\
\d(w) &= \evenleftup + \oddleftdown - \odddownleft - \evenupleft = \luo + \lde - \dle - \ulo. 
\end{align*}
\end{lem}

\begin{proof}
Fix a tangle $C$ in $\Pathconfig(u,v;w)$. 
We consider a particular intersecting pair $(b,r)$ in $C$; observe Figure~\ref{fig:proof_delta_formula} where the general structure of such a pair is sketched.
Overlapping horizontal steps are organized into segments $S_i$ as represented in Figure~\ref{fig:proof_delta_formula}, while the regions between any two consecutive segments, in which the paths are disjoint from each other, are denoted by $R_1,\ldots,R_{m-1}$. If $R_i$ has the red path above the blue path, then the right extremity of the segment $S_{i}$ has type $\prec^R_B$ while the left extremity of the segment $S_{i+1}$ has type ${}^R_B\!\succ$. Now, since the right extremity of the segment $S_m$ is of type $\prec^R_B$ we conclude that, for each intersecting pair, 
\[
\prec^R_B-{}^R_B\!\succ = 1.
\]
We sum over all intersecting pairs and use Lemma~\ref{lem:intersecting_paths} to obtain the 
first formula; see also Figure~\ref{fig:intersecting_steps}.

Similarly, if $R_i$ has the blue path above the red path, then these extremities are of type $\prec^B_R$ and ${}^B_R\!\succ$ respectively. Since the left extremity of the segment 
$S_1$ is of type ${}^B_R\!\succ$, 
this implies for each individual intersecting pair
$$
{}^B_R\!\succ-\prec^B_R = 1,
$$
and we obtain the second formula.
\end{proof}

\begin{figure}[!ht]
 \begin{center}
 \includegraphics{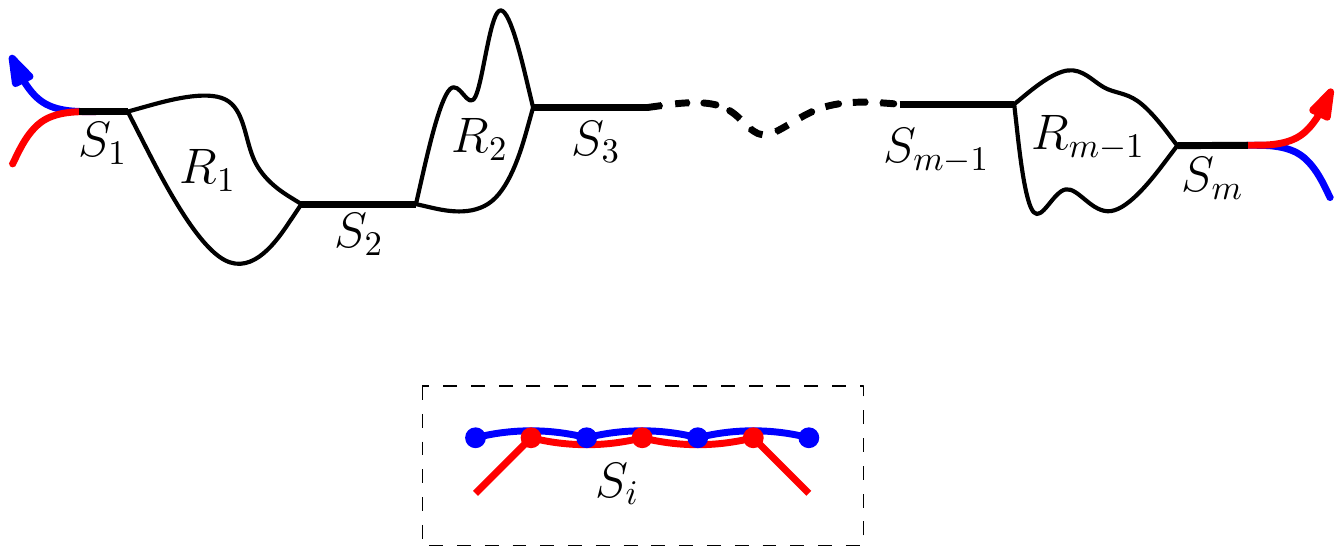}  
\caption{Structure of an intersecting pair\label{fig:proof_delta_formula}}
 \end{center}
\end{figure}

\begin{proof}[Proof of Theorem~\ref{theo:excessformula}]

 First we notice that identity~\eqref{eq:comb_interp_excess2} is an immediate consequence  of~\eqref{eq:comb_interp_prelim2}: Equations~\eqref{o2} and~\eqref{e2} imply indeed
\begin{align*}
  \d(w) - \d(u) - \d(v) &= (\d(w) - \d(u)) + (\d(w) - \d(v)) - \d(w)\\
                        &= \odddown + \evenup + \evenleft + \oddleft - \d(w).
\end{align*}

We proceed to prove~\eqref{eq:comb_interp_excess2}. In a fixed oriented TFPL, each horizontal 
step to the left is either preceded by a down step, by a left step or by an up step; 
similarly it is either followed by an up step, by a left step or a down step. This implies 
the following.
\begin{align*}
\evenleft + \oddleft &= \frac{1}{2} \left( \evenleft   +\evenleft \right) + 
\frac{1}{2} \left( \oddleft + \oddleft \right) \\
&=
\frac{1}{2} \left( \odddownleft + \oddleftleft + \oddupleft 
                +  \evenleftup + \evenleftleft + \evenleftdown \right)\\
                     &\quad 
                      + \frac{1}{2} \left( \evendownleft + \evenleftleft + \evenupleft 
                          +  \oddleftup + \oddleftleft + \oddleftdown \right)
\end{align*}
By Lemma~\ref{lem:twow}, 
$$
\d(w) = \frac{1}{2} \left( \d(w) + \d(w) \right) = \frac{1}{2} \left( \oddupleft + \evendownleft - \evenleftdown - \oddleftup 
+ \evenleftup + \oddleftdown - \odddownleft - \evenupleft \right).
$$
Now we take the difference of the two identities and obtain the desired result.
\end{proof}

%NEW
\begin{rem}
It is also possible to prove the theorem in the spirit of the proof of Lemma~\ref{lem:twow}, without resorting to TFPLS: By analyzing closely the structure of a given intersecting pair $(b,r)$ as in Figure~\ref{fig:proof_delta_formula}, one can check that {\em for a fixed particular pair}
 \[ \blueh + \redh -1 =  \raisebox{-0.4cm}{\includegraphics{Formula_Excess_path_model_bis}}
\]
where all blue/red intersections involved concern the paths $b$ and $r$; for the horizontal steps from $b$ or $r$ on the l.h.s., the intersecting path is $r$ or $b$ (remember that all horizontal steps are necessarily crossed by Theorem~\ref{theo:path_model_characterization}).  Summing over all $\d(w)$ intersecting pairs one obtains Identity~\ref{eq:comb_interp_prelim}.
\end{rem}

%%%%%%%%%%%%%%%%%%%%%%%%%%%%%%%%%%%%%%%%%%%%%%%%%%%%%%%%%%%%%
%%%%%%%%%%%%%%%%%%%%%%%%%%%%%%%%%%%%%%%%%%%%%%%%%%%%%%%%%%%%%

%%%%%%%%%%%%%%%%%%%%%%%%%%%%%%%%%%%%%%%%%%%%%%%%%%%
\section{Configurations of excess $0$}
\label{sec:excess_0}
%%%%%%%%%%%%%%%%%%%%%%%%%%%%%%%%%%%%%%%%%%%%%%%%%%%

We start by defining the excess of oriented TFPLs, which one can see as a measure of complexity of the object.

\begin{defi}[Excess]
 Given three words $u,v,w$ of length $N$, we define the \emph{excess}  as $\operatorname{exc}(u,v;w)=\d(w) - \d(u) - \d(v)$. If $\operatorname{exc}(u,v;w)=k$ then oriented TFPLs or path tangles with boundary $(u,v;w)$ are said to have excess $k$.
\end{defi}

There is no oriented TFPL or path tangle with boundary $(u,v;w)$ unless the excess is nonnegative, by Corollary~\ref{cor:inequality}. In this section we enumerate configurations of excess $0$, recovering in particular the results of~\cite{TFPL2}.

\subsection{Characterization}

\begin{prop}
\label{prop:excess0} Given a path tangle $C\in\Pathconfig(u,v;w)$, one has $\operatorname{exc}(u,v;w)=0$ if and only none of the following configurations occurs in $C$:
\begin{center}
  \includegraphics{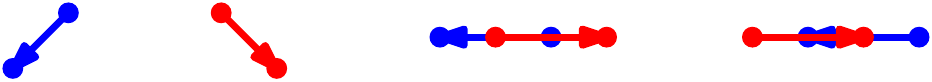}
\end{center}
Equivalently, for a TFPL $f$ in $\overrightarrow{T}_{u,v}^{w}$, one has  $\operatorname{exc}(u,v;w)=0$ if and only if
$\oddd(f)=\evenu(f)=0$ and there are no two consecutive left arrows.
\end{prop}

\begin{proof} This is an immediate consequence of Formula~\eqref{eq:comb_interp_excess} (equivalently, Formula~\eqref{eq:comb_interp_excess2}) in Theorem~\ref{theo:excessformula}.
\end{proof}

For oriented TFPLs we have the following consequences.

\begin{prop}
\label{prop:struct_tfpls_exc0} Oriented TFPLs of excess $0$ have the following properties:
\begin{enumerate}
\item\label{leftright} They do not contain paths joining two bottom 
vertices that are oriented from right to left.
\item\label{weight1} Their weight is $1$.
\item\label{closedpath} They do not contain closed paths.
\end{enumerate}
In particular, $\overrightarrow{t}_{u,v}^{w}(q)=t_{u,v}^w$ if $\exc(u,v;w)=0$.
\end{prop}

\begin{proof}
{\it Part (1):} Suppose there existed an oriented TFPL configuration with boundary $(u,v;w)$ of excess $0$ 
containing a bottom path that is oriented from right to left. Consider  the configuration we obtain 
by reorienting this path and let $(u,v;w')$ be its boundary. The word $w'$ is obtained from $w$ by interchanging a 
$0$ and a $1$, where the $1$ is located left of the $0$. This implies  
$\d(w') < \d(w) = \d(u) + \d(v)$, which is a contradiction to 
Theorem~\ref{theo:excessformula}.

{\it Part (2):} By Proposition~\ref{prop:excess0}, $\dle=\lue=\ldo=\ulo=0$, which implies, by Proposition~\ref{prop:brweight}, that 
the weight is $q$ raised to
$$
 \frac{1}{2} \lde + \frac{1}{2} \ule - \frac{1}{2} \dlo - \frac{1}{2} \luo.
$$
This expression vanishes as 
$\lde = \dlo$  and  $\ule = \luo$. 
This is because occurrences of $\lde$ and $\dlo$ always appear in pairs sharing the horizontal blue edge
since $\evenll=0$, $\oddll=0$,  $\ulo=0$ and $\lue=0$; a similar argument leads to the second equation.

{\it Part (3):} By Proposition~\ref{prop:counterclockwise} and since there is no bottom path oriented from right to left, the 
weight is $q$ raised to the difference of the number of closed paths oriented counterclockwise  and the number of 
closed paths oriented clockwise, which must be zero. If there were a closed path then, by reorienting it, we would obtain another oriented TFPL of excess zero. However, this changes the above mentioned difference, which is impossible by Part (2).
\end{proof}

This proposition generalizes Lemma~13 of~\cite{TFPL2}.

\subsection{Puzzles and Littlewood--Richardson coefficients}
\label{sub:puzzles_LR}

The vertices involved in a given path tangle of size $N$ have integer coordinates $(x,y)$ verifying 
$x \geq y \ge 0$ and $x+y\leq 2N-1$; let $V_N$ be this set of points. Recall that blue (\emph{resp.} red) paths use vertices whose sum of coordinates is even (\emph{resp.} odd).

We superimpose a triangular grid $\TG_N$ on the vertices $V_N$ as follows: Southwest-Northeast edges ($/$-edges) have blue vertices as middle points, while Southeast-Northwest edges ($\backslash$-edges) have red vertices as middle points. See Figure~\ref{fig:triangular_grid}, left, for the case $N=3$. We will in fact rescale this triangular grid so that it becomes made of equilateral triangles $\bigtriangleup,\bigtriangledown$, as shown on the right of the same picture.

\begin{figure}[!ht]
\begin{center}
 \includegraphics[width=0.6\textwidth]{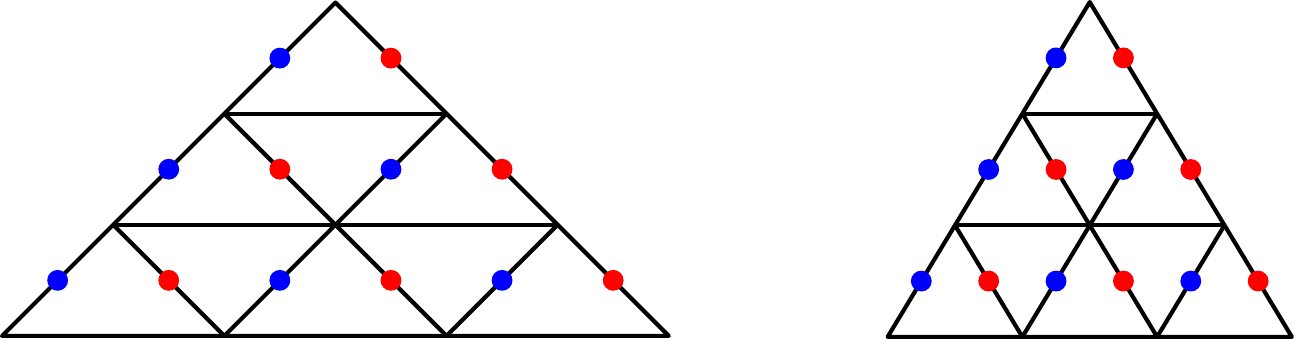}  
\caption{From $V_3$ in the square grid to the triangular grid $\TG_N$.\label{fig:triangular_grid}}
 \end{center}
\end{figure}

Consider a path tangle for a given $N$. For each blue path (\emph{resp.} red path), if its starting point is $(2i,0)$  (\emph{resp.} $(2i+1,0)$), then change it to $(2i+\frac{1}{2},-\frac{1}{2})$ and add half an up step from it to the original starting point. (See Figure~\ref{fig:triangular_grid_bis} for an example.) As a result all starting points now occupy the positions $(2i+\frac{1}{2},-\frac{1}{2})$ for $i$ going from $0$ to $N-1$. We call this the \emph{extended} path tangle. When the triangular grid is superimposed, the added points are precisely the middle points of the bottom edges of $\TG_N$. The interior of the equilateral triangles are filled with small pieces which come from the up or horizontal steps of paths, while the down steps follow the edges of these equilateral triangles; see an example of these on Figure ~\ref{fig:triangular_grid_bis}. Note that for any kind of step, only half of it can appear in a given triangle.

\begin{figure}[!ht]
\begin{center}
 \includegraphics[width=0.6\textwidth]{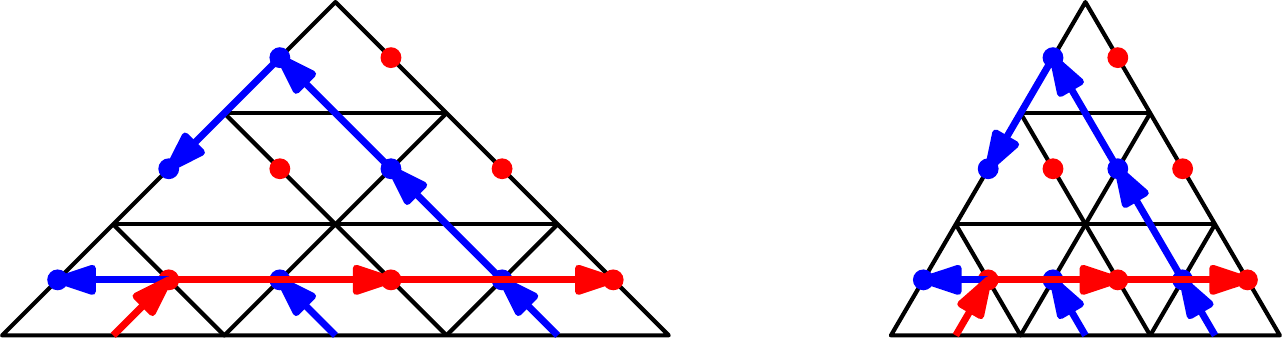}  
\caption{Blue-red path tangles on $\TG_N$.\label{fig:triangular_grid_bis}}
 \end{center}
\end{figure}

Suppose now that our path tangle has excess $0$. It has no down step by Proposition~\ref{prop:excess0}, and thus blue and red steps appear in the interior of triangles. The following proposition tells us which local configurations can appear.

\begin{lem}
A path tangle has excess $0$ if and only if the induced configurations on equilateral triangles belong to Figure~\ref{fig:puzzle_pieces_excess_0}. 
\end{lem}

\begin{proof}
 Suppose we have a path tangle of excess $0$. Let us first consider an upwards equilateral triangle $\bigtriangleup$. Assume first there are no horizontal steps crossing $\bigtriangleup$. Then one can have up steps of either color coming from the bottom edge of $\bigtriangleup$, or no steps at all. Note that one cannot have both steps at the same time since such a crossing is forbidden by Condition~\ref{it:path_model_characterization_1} of Theorem~\ref{theo:path_model_characterization}. The only possibilities are $U_1,U_2$ and $U_5$ on Figure~\ref{fig:puzzle_pieces_excess_0}. Now consider the case where there is a horizontal step crossing $\bigtriangleup$. By Proposition~\ref{prop:excess0} there can be only one such horizontal step, say red (the blue case is symmetric). In this case its midpoint occurs on the $/$ edge of $\bigtriangleup$; by Condition~\ref{it:path_model_characterization_2} of Theorem~\ref{theo:path_model_characterization}, the blue paths must use this midpoint. Only up steps are allowed in the case of excess $0$, and therefore the induced configuration of 
$\bigtriangleup$ is $U_4$. The case of a blue horizontal step is symmetric and gives configuration $U_3$.

The case of the downwards equilateral triangle $\bigtriangledown$ is similar and left to the reader.

 Conversely, path tangles built up by piecing together triangles of Figure~\ref{fig:puzzle_pieces_excess_0} have necessarily excess $0$ because they do not contain any of the forbidden patterns of Theorem~\ref{theo:path_model_characterization}.
\end{proof}

\begin{figure}[!ht]
 \begin{center}
 \includegraphics{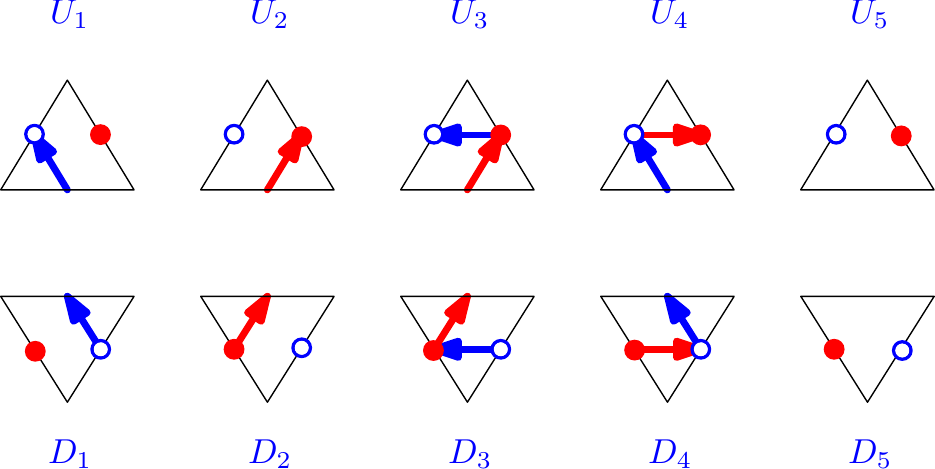}  
\caption{Local configurations for excess $0$.\label{fig:puzzle_pieces_excess_0}}
 \end{center}
\end{figure}

\begin{rem}
 The path tangles that we obtain in this case are also considered in~\cite{ZJ-LR}, where it is also mentioned that this model is equivalent to Knutson--Tao puzzles. We shall show the equivalence here in more detail, especially since we wish to extend our results beyond the case of excess $0$.
\end{rem}

If two triangles $\bigtriangleup,\bigtriangledown$ are adjacent through an edge and each have one of the local configurations from Figure~\ref{fig:puzzle_pieces_excess_0}, then these configurations must be ``compatible'' to ensure that they come from a path tangle. More precisely, we must ensure that the half steps from the triangles must be paired up to form full steps. These compatibility conditions can be encoded by labeling the edges of the triangles $\bigtriangleup,\bigtriangledown$, and allowing two triangles to be adjacent if and only if the edge they share has the same labels in both of them.
 
 We use the three labels $0,1,2$, and we now detail their interpretation according to what type of edge $/$, 
 $\backslash$ or $-$ they are attached to (see Figure~\ref{fig:puzzle_pieces_excess_0_labeled} for the triangles of Figure~\ref{fig:puzzle_pieces_excess_0} with the labels attached).
\begin{enumerate}
     \item Edges $/$ (correspond to blue vertices) 
\begin{itemize}
\item[$\bullet$] $0$ means that the blue vertex is not the midpoint of a horizontal red edge, and a blue path goes through the vertex; if it is an edge on the left boundary then 
a blue path ends there.
 \item[$\bullet$] $1$ means that no blue path uses this vertex (which implies that it is not 
the midpoint of a horizontal red edge).
\item[$\bullet$] $2$ means that the blue vertex is the midpoint of a horizontal red step (which implies that a blue path goes through the vertex).
\end{itemize}
    \item Edges $\backslash$ (correspond to red vertices)
\begin{itemize}
\item[$\bullet$] $0$ means that no red path uses this vertex (which implies that it is not 
the midpoint of a horizontal blue step).
\item[$\bullet$] $1$ means that the red vertex is not the midpoint of a horizontal blue edge, and 
a red path goes though the vertex; if it is an edge on the right boundary then a red path ends there.
\item[$\bullet$] $2$ means that the red vertex is the midpoint of a horizontal blue step (which implies that a red path goes through the vertex).
\end{itemize}
    \item Edges $-$ (correspond to possible midpoints of up steps)
\begin{itemize}
\item[$\bullet$] $0$ means that it is crossed by a blue up step; if it is an edge on the bottom boundary then a blue path starts there.
\item[$\bullet$] $1$ means that it is crossed by a red up step; if it is an edge on the bottom boundary then a red path starts there.
\item[$\bullet$] $2$ means that no up step is crossing.
\end{itemize}
\end{enumerate}

\begin{figure}[!ht]
 \begin{center}
 \includegraphics{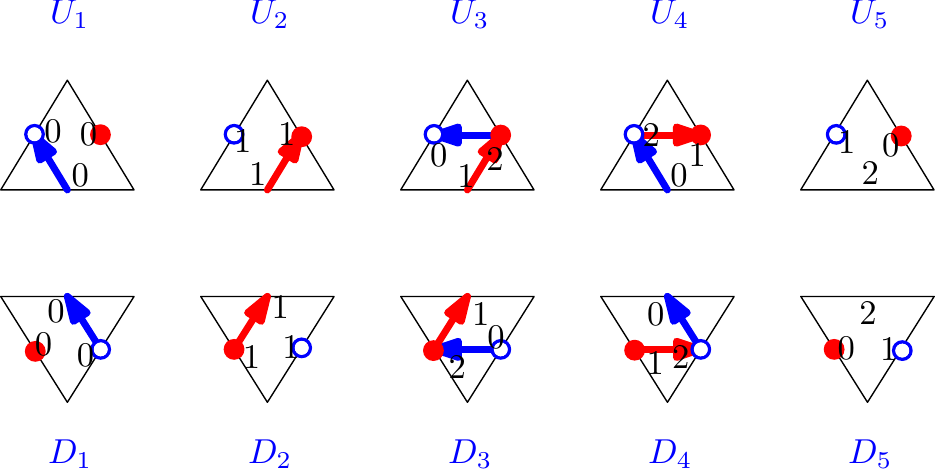}  
\caption{The edge labelings of the configurations of Figure ~\ref{fig:puzzle_pieces_excess_0}. \label{fig:puzzle_pieces_excess_0_labeled}}
 \end{center}
\end{figure}

Note that the labelings obtained on both types of triangles are all distinct, and thus one can remove the local path configurations inside each of these triangles and retain only the labels without losing any information.

\begin{defi}[Knutson--Tao puzzles]
 A \emph{Knutson--Tao puzzle} of size $N$ is a labeling of each triangle of type 
$\bigtriangleup$ or $\bigtriangledown$ of $\TG_N$ by one of the possibilities of Figure~\ref{fig:puzzle_pieces_excess_0_labeled}, so that, whenever two triangles are adjacent, their common edge has the same label in both triangles. 
 The puzzle has boundary $(u,v;w)$ if the labels on the left, right and bottom sides of $\TG_N$ are given by $u,v$ and $w$ respectively, when read from left to right. 
\end{defi}

Such puzzles were introduced in~\cite{KT,KTW} as a combinatorial model for the Littlewood-Richardson coefficients. We briefly recall the definition of the Littlewood-Richardson coefficient: Let ${\bf x}=\{x_1,\ldots,x_n\}$ be a set of variables and $\Lambda({\bf x})$ be the algebra of symmetric functions in ${\bf x}$.  Schur functions 
$s_{\lambda}({\bf x})$ associated with Ferrers diagram $\lambda$ with $n$ rows form a basis of  $\Lambda({\bf x})$ and can be defined as follows
$$
s_{\lambda}({\bf x})= \frac{\det\limits_{1 \le i , j \le n} \left( x_i^{\lambda_j + n - j} \right)}{\prod\limits_{1 \le i < j \le n} \left( x_i - x_j \right) },
$$
where $\lambda_i$ is the number of cells in the $i$-th row of the diagram. The {\it Littlewood-Richardson coefficient} $c^{\lambda}_{\mu,\nu}$ is indexed by three Ferrers diagrams $\mu, \nu, \lambda$ and defined through the expansion
$$
s_{\mu}({\bf x}) s_{\nu}({\bf x}) = \sum_{\lambda} c^{\lambda}_{\mu,\nu} s_{\lambda}({\bf x}).
$$
Equivalently, in the representation theory of the group $\mathbf{GL}_n(\mathbb{C})$, the 
coefficient $c_{\mu,\nu}^{\lambda}$ appears as the multiplicity of the irreducible 
representation $V_{\lambda}$ in the tensor product of the irreducible representations 
$V_{\mu}$ and $V_{\nu}$.  If $u,v,w$ are three words, we set $c_{u,v}^w:=c_{\lambda(u),\lambda(v)}^{\lambda(w)}$ where the correspondence $u\mapsto \lambda(\cdot)$ was introduced in   Section~\ref{sec:definitions}, see Figure~\ref{fig:word_to_partition}.
Note that the coefficient $c_{\mu,\nu}^{\lambda}$ is non-zero only if $|\lambda|=[\mu|+|\nu|$, so that $c_{u,v}^w$ is non-zero only if $\d(w)=\d(u)+\d(v)$.

\begin{theo}
\label{theo:LR}
 Let $u,v,w$ be words of the same length and with the same number of $0$s such that $\d(w)=\d(u)+\d(v)$. Then the number of oriented TFPLs with boundary $(u,v;w)$ is the Littlewood--Richardson coefficient $c_{u,v}^w$.
\end{theo}

\begin{proof}
 By Theorem~\ref{theo:path_model_characterization}, oriented TFPLs are equivalent to path tangles with the same boundary. In the present case of excess $0$, to show that the procedure above is a bijection with Knutson--Tao puzzles, it remains to show that the boundary conditions are preserved. However, the interpretations of the labels on the three types of edges given above shows that, in the bijection between path tangles and puzzles, the centers of the $-$-edges on the bottom boundary that carry the label $0$ are precisely
 starting points of blue paths, while the $/$-edges on the left boundary carrying the label $0$ are precisely the endpoints of the blue paths. The case of red paths is symmetric.
 
Now as proved in~\cite{ZJ-LR,KT,KTW}, such puzzles are counted by $c_{u,v}^w$.
\end{proof}
 
This result generalizes~\cite{TFPL2}, which concerned the case of ordinary TFPLs. Now we will go one step further and consider the case of excess $1$.

%%%%%%%%%%%%%%%%%%%%%%%%%%%%%%%%%%%%%%%%%%%%%%%%%%%%%%%%%%%%%
%%%%%%%%%%%%%%%%%%%%%%%%%%%%%%%%%%%%%%%%%%%%%%%%%%%%%%%%%%%%%

%%%%%%%%%%%%%%%%%%%%%%%%%%%%%%%%%%%%%%%%%%%%%%%%%%%
\section{Configurations of excess $1$}
\label{sec:excess_1}
%%%%%%%%%%%%%%%%%%%%%%%%%%%%%%%%%%%%%%%%%%%%%%%%%%%

We now want to enumerate configurations of excess $1$, i.e. such that $\d(w)-\d(u)-\d(v)=1$. The idea is to first transform such configurations into certain new puzzles, and then reduce the enumeration of these puzzles to the case of Knutson--Tao puzzles.

\subsection{Characterization and weights}

The following is analogous to Proposition~\ref{prop:excess0}.

\begin{prop}
\label{prop:excess1}
A path tangle has excess $1$ if and only if there is one local configuration (the {\em excess}) among the first four in the following list that appears precisely once, whereas the other seven configurations in the list do not appear at all.
$$
\db \quad \dr \quad \evenll \quad \oddll \quad|\quad \lue \quad \ulo \quad \dle \quad \ldo.
$$
In terms of oriented TFPLs the list is as follows.
$$
\odddown \quad  \evenup \quad  \evenleftleft \quad \oddleftleft  \quad|\quad \oddleftup \quad
 \evenupleft \quad
   \odddownleft  \quad \evenleftdown
$$
\end{prop}

\begin{proof}
This is a direct consequence of Theorem~\ref{theo:excessformula}.
Note that if, for instance, a path tangle contained an occurrence of 
$\lue$ then it would also contain an occurrence of $\dr$. Hence 
a path tangle that contains the local configuration $\lue$ has excess $2$ at least.
 \end{proof}

In the following, $\db$, $\dr$, etc. denotes again the 
number of local configurations of type $\db$, $\dr$, etc., respectively, 
in a given path tangle.

\begin{defi}[Type $BD$, $RD$, $DHD$, $DHU$]
A path tangle of excess $1$ is said to be of type $BD$ ({\it resp.} $RD$, $DHD$, $DHU$)  
if $\db=1$ ({\it resp.} $\dr=1$, $\evenll=1$, $\oddll=1$). The type of an oriented TFPL of excess $1$ is 
the type of the corresponding path tangle.
\end{defi}

We compute the weight of path tangles of excess $1$ and given type.

\begin{prop} 
\label{prop:weight_excess1}
The weight of a 
path tangle (resp. an oriented TFPL configuration),  of excess $1$ is $1$ if it is of type $BD$ or $RD$, it is $q$ if it is 
of type $DHD$ and it is $1/q$ if it is of type $DHU$. 
\end{prop}

\begin{proof} By Proposition~\ref{prop:brweight} and since
$\dle=\lue=\ldo=\ulo=0$ by Proposition~\ref{prop:excess1}, the weight of a path tangle 
of excess $1$ is
$$ \frac{1}{2} \lde + \frac{1}{2} \ule - \frac{1}{2} \dlo - \frac{1}{2} \luo.$$

For path tangles of type $BD$ and $RD$ we have $\dlo=\lde$ and 
$\luo=\ule$ by an argument given in Proposition~\ref{prop:struct_tfpls_exc0}, and hence the weight is $1$ in this case.

In path tangles of type $DHD$, the blue and the red horizontal step in the unique occurrence
of $\evenll$  is preceded by a blue and a red up step respectively. This 
implies $\lde = \dlo + 1$ and $\ule= \luo +1$, which proves the claim for this type. 

In path tangles of type $DHU$, the blue and the red horizontal step in the unique occurrence of type $\oddll$ is followed by 
a blue and red up step respectively. Therefore
$\dlo= \lde + 1$ and $\luo = \ule +1$.
\end{proof}

\subsection{Puzzles of excess~$1$}

Here we introduce puzzles that correspond to path tangles and oriented TFPLs of excess $1$. Each type of excess will be reflected by a new puzzle piece which 
is located  ``at the excess'' in the path tangle. For the rest of the path tangle we use the pieces of the ordinary Knutson-Tao puzzles. 

\begin{defi}[Puzzles of excess~$1$] 
A \emph{$BD$--puzzle} of size $N$ is a labeling of  ${\mathcal T}_N$ such that 
\begin{enumerate}
\item there is precisely one  pair of
adjacent $/$-edges (the {\it excess}) labeled as indicated in the first column of Figure~\ref{fig:puzzle_pieces_excess_1}, 
\item the labeling of each triangle 
can be found in  Figure~\ref{fig:puzzle_pieces_excess_0_labeled}, and, 
\item whenever two triangles are adjacent, their common edge has the same label in both triangles with the exception of the pair of adjacent edges that was selected in (1).
\end{enumerate}
A \emph{$RD$--puzzle} of size $N$  is defined analogously with (1) being replaced by 
``there is precisely one  pair of
adjacent $\backslash$-edges (the {\it excess}) labeled as indicated in the second column of Figure~\ref{fig:puzzle_pieces_excess_1}''.
A \emph{$DHD$--puzzle} contains a unique triangle $\bigtriangledown$ (the {\it excess}) whose edges are labeled with $2$ 
(Figure~\ref{fig:puzzle_pieces_excess_1}, Column~3), while 
a  \emph{$DHU$--puzzle} contains a unique triangle $\bigtriangleup$ (the {\it excess}) whose edges 
 are labeled with $2$
(Figure~\ref{fig:puzzle_pieces_excess_1}, Column~4) and which has no edge on the boundary of ${\mathcal T}_N$.
\end{defi}

\begin{figure}[!ht]
 \begin{center}
 \includegraphics[width=0.7\textwidth]{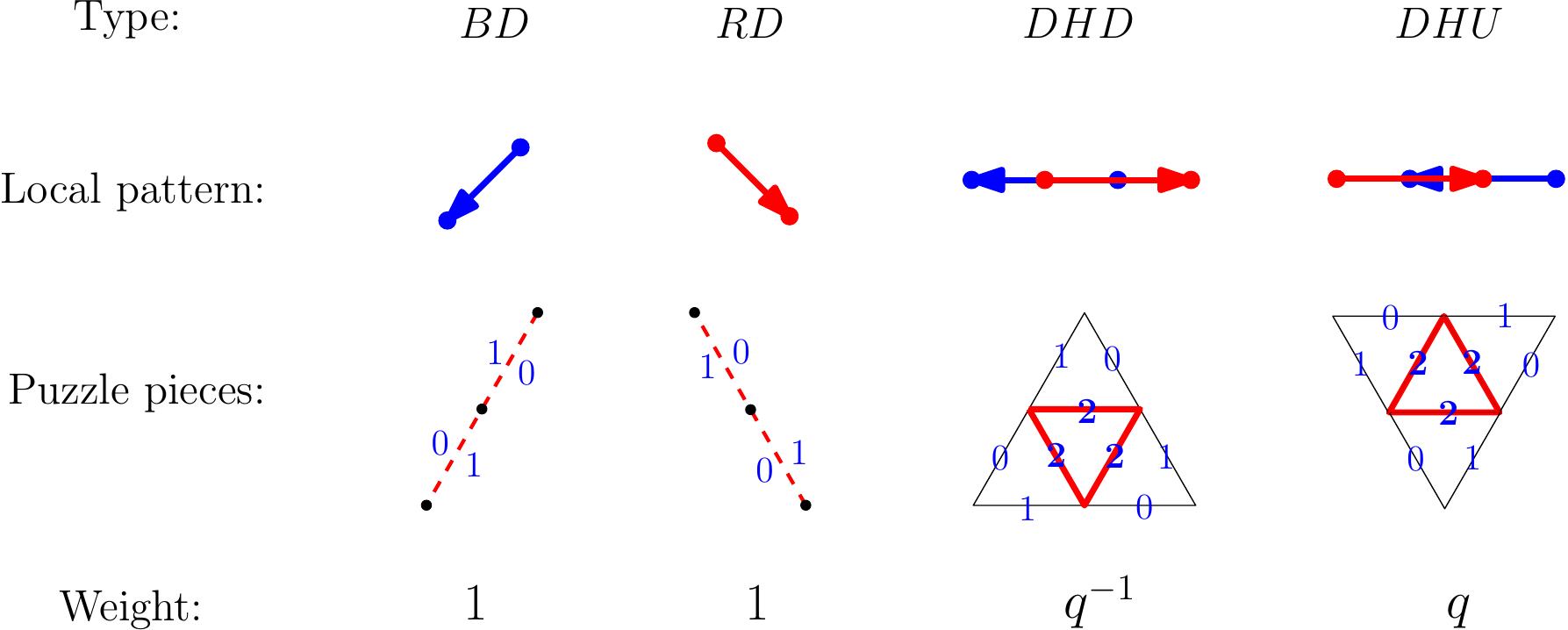}  
\caption{Supplementary local configurations in the excess-$1$-case. \label{fig:puzzle_pieces_excess_1}}
 \end{center}
\end{figure}

\begin{prop} Let $u,v,w$ be words with $\exc(u,v;w)=1$ and $X \in \{BD, RD, DHD, DHU\}$.
The path tangles with boundary $(u,v;w)$ of type $X$ are in bijection with
$X$--puzzles with boundary $(u,v;w)$.
\end{prop}

\begin{proof} We consider the cases $X= BD, DHD$; the case $X=RD$ follows by symmetry from the case 
$X=BD$ and the case $X=DHU$ is similar to 
the case $X=DHD$.

{\it Case $X=BD$.} Consider a path tangle of excess $1$ with $\db=1$ and superimpose the triangular grid ${\mathcal T}_{N}$. If the blue down step is not located on the left boundary then it follows the edges
of four triangles of the grid, two of which are of type $\bigtriangleup$.  We address them according to their relative position (top, bottom) as 
$\bigtriangleup_t,  \bigtriangleup_b, \bigtriangledown_t,  \bigtriangledown_b$. As  $\ldo=0$ and $\dle=0$, none of these triangles are traversed by a 
horizontal red step. It is easily checked that this implies that, after removing the blue down step from the path tangle, the  four triangles appear in the list given in Figure~\ref{fig:puzzle_pieces_excess_0}. To be more precise, 
\begin{itemize}
\item $\bigtriangleup_t \in  \{U_1,U_3\}$, 
\item $\bigtriangleup_b \in \{U_2, U_5\}$, 
\item $\bigtriangledown_t \in \{D_2, D_5\}$, and
\item $\bigtriangledown_b \in \{D_1, D_3\}$.
\end{itemize}
All combinations are possible.
Now it is evident that the exceptional puzzle piece has to be placed along the two adjacent $/$-edges that are traversed by the blue down step
as 
\begin{itemize}
\item $U_1, U_3$ are the triangles  of type $\bigtriangleup$, whose $/$-edges carry the label $0$, 
\item $U_2, U_5$ are the triangles  of type $\bigtriangleup$, whose $/$-edges carry the label $1$,
\item $D_2, D_5$ are the triangles of type $\bigtriangledown$, whose $/$-edges carry the label $1$, and,  
\item $D_1, D_3$ are the triangles of type $\bigtriangledown$, whose $/$-edges  carry the label $0$.
\end{itemize}
If the blue down step is on the boundary the argument is similar.
The fact that the boundary words coincide follows from the interpretations of the labels on the edges given in Section~\ref{sec:excess_0}, with the exception
of the two edges of the excess if it is located on the left boundary. However in this case, the center of the bottom  $/$-edge of the excess is the endpoint of a blue path while the top  $/$-edge is not, which is consistent as the bottom edge contributes a $0$ to the boundary condition of the puzzle and the top edge contributes a $1$.

\begin{center}
\includegraphics[height=2.5cm]{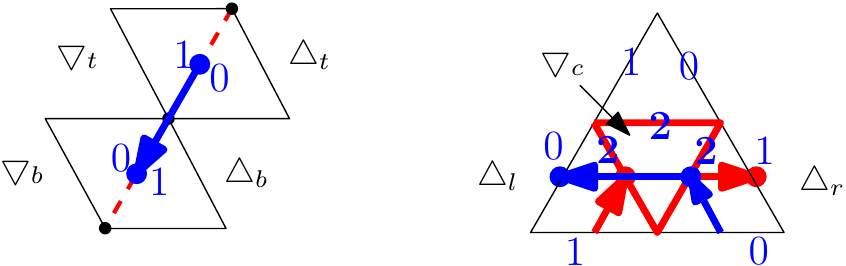}
\end{center}

{\it Case $X=DHD$.} Consider a blue-red path tangle of excess $1$ with one occurrence of $\evenll$ and superimpose the triangular grid ${\mathcal T}_{N}$. This local configuration has a non-empty intersection with two triangles of type 
$\bigtriangleup$ and 
one triangle of type $\bigtriangledown$.  We address them according to their relative position (left, right, center) as 
$\bigtriangleup_l,  \bigtriangleup_r, \bigtriangledown_c$. The red edge of the local configuration must be preceded by a red up step; this implies that $\bigtriangleup_l = U_3$. Also the blue edge must be preceded by a blue up step, which implies $\bigtriangleup_r = U_4$. 
The triangle $\bigtriangleup_t$ which is adjacent to $\bigtriangledown_c$ via the top edge of the latter is of type $U_5$. It is clear that also all the 
other triangles except for $\bigtriangledown_c$ are of types given in Figure~\ref{fig:puzzle_pieces_excess_0}.
If we use the labeling given in Figure~\ref{fig:puzzle_pieces_excess_0_labeled} then all common edges with $\bigtriangledown_c$ carry the label 
$2$. 
\end{proof}

\subsection{Moving an excess of type $BD$ or $RD$}
\label{sub:move_BD_RD}

Suppose we are given a $BD$--puzzle. In Figure~\ref{fig:puzzle_moves_excess_1}, upper half, it is indicated how it is possible 
to move the excess towards the right boundary if we assume certain labelings of some triangles close to the excess. (Note that it is crucial that the labeling
of the boundary edges of the local configuration does not change.)
Similarly, for $RD$--puzzle where the excess is not already on the right boundary of ${\mathcal T}_{N}$,
moves are given in  Figure~\ref{fig:puzzle_moves_excess_1}, lower half. 
In the course of moving the excess to the boundary it is sometimes necessary to change the type 
of the excess (see moves $BR$ and $RB$). In these cases we also 
have intermediate puzzles of type $DHD$ and type $DHU$ respectively; this will be of importance in Section~\ref{sub:DHD_DHU}.
 
The following lemma shows that there is always precisely one move that can be applied.

\begin{figure}[!ht]
 \begin{center}
 \includegraphics[width=0.7\textwidth]{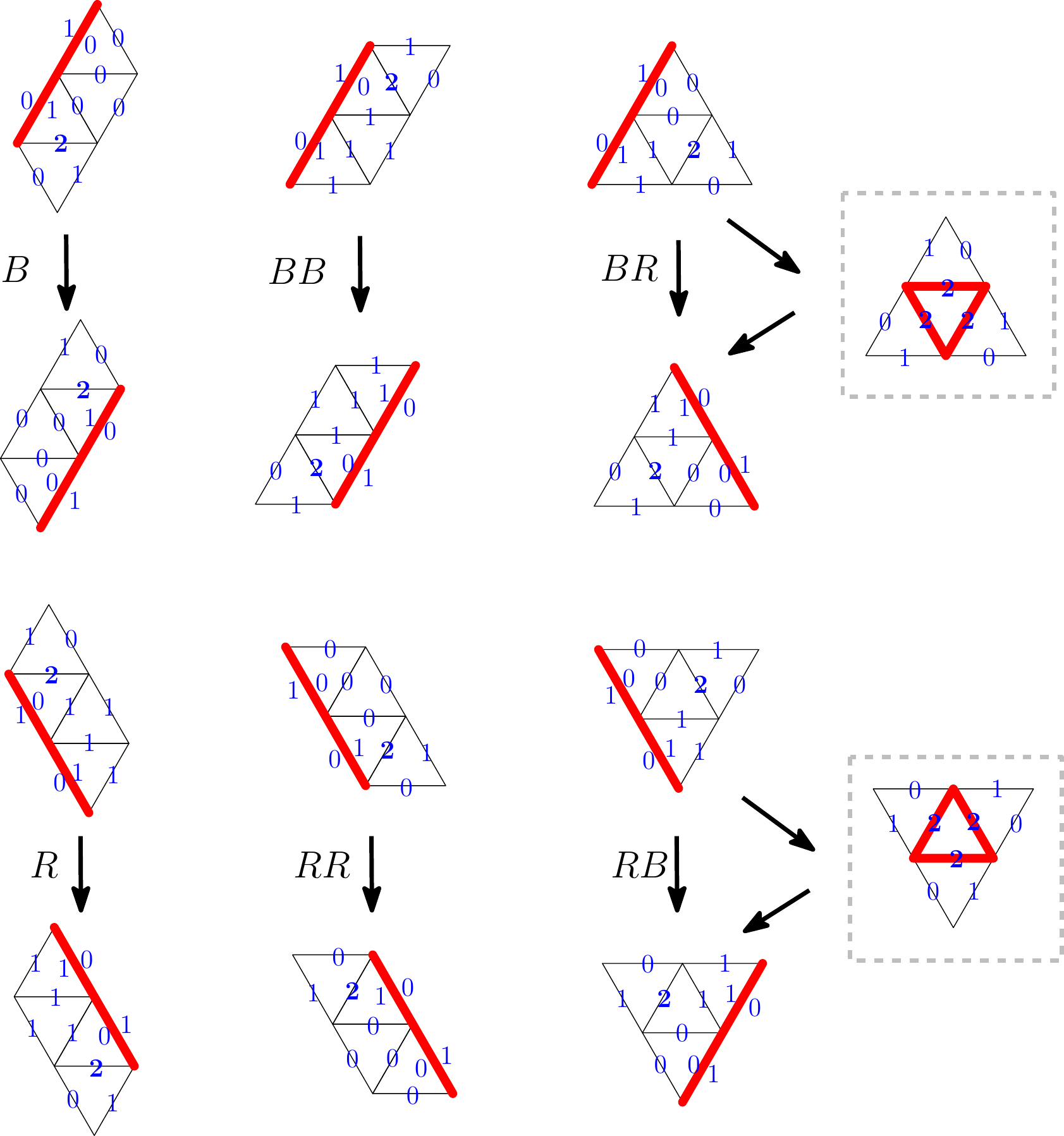}  
\caption{Rules for moving the excess. The rightmost rules involving $DHU$ and $DHD$ will only be used in Section~\ref{sub:DHD_DHU}.\label{fig:puzzle_moves_excess_1}}
 \end{center}
\end{figure}

\begin{lem}
\label{lem:move_right} 
For a given $BD$--puzzle, there is precisely one move in $\{B, BB, BR\}$ that can be applied.
For a given $RD$--puzzle where the excess is not on the right boundary of ${\mathcal T}_{N}$, there is precisely one move in 
$\{R, RR, RB\}$ that can be applied.
\end{lem}

\begin{proof} We only consider the case of $BD$--puzzles.
Let $\bigtriangleup_t, \bigtriangleup_b$ denote the two triangles of type $\bigtriangleup$ that are adjacent to the excess (top, bottom), and 
$\bigtriangledown_c$ denote the triangle that is adjacent to $\bigtriangleup_t$ and $\bigtriangleup_b$. We have 
$\bigtriangleup_t \in \{U_1, U_3\}$ and $\bigtriangleup_b \in \{U_2, U_5\}$. Note that it is not possible to have 
$\bigtriangleup_t = U_3$ and $\bigtriangleup_b = U_5$, because then there is no option for $\bigtriangledown_c$.

If $\bigtriangleup_t=U_1$ and  $\bigtriangleup_b=U_2$ then $\bigtriangledown_c=D_4$ and $\bigtriangleup_r=U_4$ where 
$\bigtriangleup_r$ is the third triangle that is adjacent to $\bigtriangledown_c$ via an edge. Here it is possible to apply 
Move~$BR$.
If $\bigtriangleup_t=U_1$ and $\bigtriangleup_b=U_5$ then $\bigtriangledown_c=D_1$ and $\bigtriangledown_b=D_5$ 
where  $\bigtriangledown_b$ is the triangle adjacent to $\bigtriangleup_b$ via the horizontal edge. Now it is possible to 
apply move~$B$.
Finally, if $\bigtriangleup_t=U_3$ and  $\bigtriangleup_b=U_2$ then $\bigtriangledown_c=D_2$ and $\bigtriangledown_r=D_3$, where $\bigtriangledown_r$ is the triangle adjacent to $\bigtriangleup_t$ via the 
$\backslash$-edge. It is possible to apply move~$BB$.
\end{proof}

Clearly, the application of a move of type $B$ does not reduce the distance of the excess to the right boundary; however it reduces the distance to the bottom boundary and once we have reached this boundary, moves of type $BB$ and $BR$ can be used to eventually reach the right boundary. By symmetry, we obtain an analogous result for moving an excess in puzzles of type $BD$ or $RD$ to the left boundary.

\begin{cor} 
\label{cor:move_left}
For each $BD$-puzzle where the excess is not on the left boundary of ${\mathcal T}_{N}$ there is precisely one move in 
$\{B^{-1}, BB^{-1}, RB^{-1}\}$ that can be applied. Likewise, for each $RD$--puzzle, there is precisely one move in 
$\{R^{-1}, RR^{-1}, BR^{-1}\}$ that can be applied.
\end{cor}

{\bf $\mr(P)$ and $\ml(P)$:} For a  puzzle $P$ of type $BD$ or $RD$ where the excess is not on the right boundary, we let 
$\mr(P)$ denote the puzzle that we obtain after applying the move given by Lemma~\ref{lem:move_right}. 
Likewise, we define $\ml(P)$ using Corollary~\ref{cor:move_left}. It will also be necessary to have the moves translated into 
the path tangles. This is accomplished in 
Figure~\ref{fig:pathmodel_moves_excess_1}. In this figure, isolated vertices stand for vertices that are not involved in the path tangle. 

\begin{figure}[!ht]
 \begin{center}
 \includegraphics[width=0.5\textwidth]{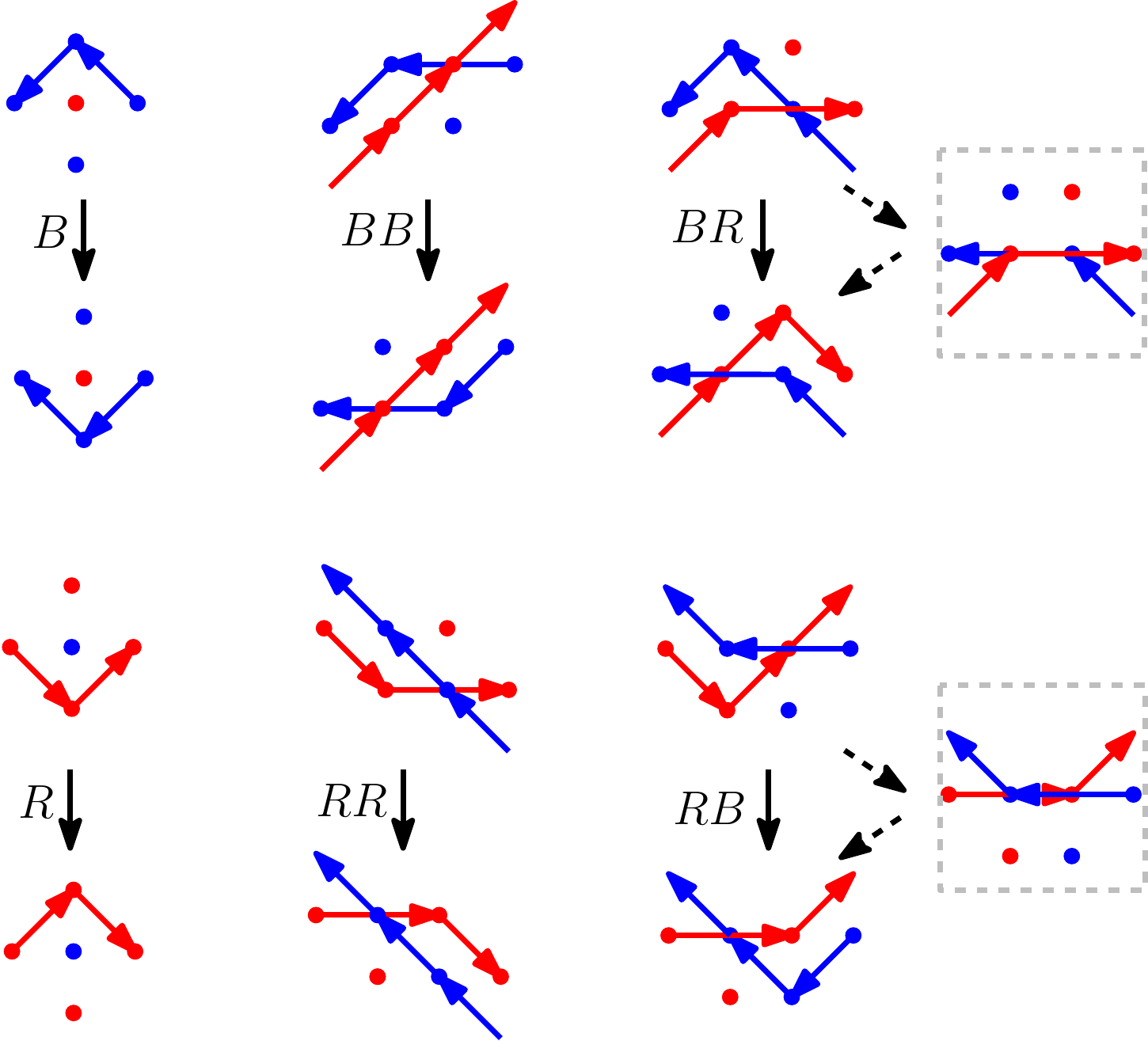}  
\caption{Moves in a blue-red path tangle with excess $1$.\label{fig:pathmodel_moves_excess_1}}
 \end{center}
\end{figure}

%NEW
\begin{rem}
Some of the moves described here already appeared under slightly disguised forms in the articles~\cite{KT} and~\cite{Purbhoo}. In~\cite{KT}, the authors introduce a \emph{gash} in the puzzles they deal with, and move the gash according to some set of rules. A similar situation occurs in~\cite{Purbhoo} where the ``migration'' moves are those of Figure~\ref{fig:puzzle_moves_excess_1} (the ones involving $DHD$- or $DHU$ pieces excepted), and are used to migrate pieces from one side to the other. In both cases the intermediary puzzles do not matter, the moves being just bijective devices to transform a puzzle into another. Our situation is different, because our moves take place between the very puzzles that we want to enumerate: we need to be able to count the number of moves that occur in total.
\end{rem}

\subsection{The path of an excess of type $BD$ or $RD$}
\label{sub:path_BD_RD}

As indicated before, repeated application of $\mr$ to a puzzle of type $BD$ or $RD$ always leads to a $RD$--puzzle with the excess on the right boundary, while 
repeated application of $\ml$ leads to a $BD$--puzzle with the excess on the left boundary. No puzzle may appear more than once 
in this procedure. 
We introduce some related notions.

\begin{defi}[$\p(P)$, $\lp(P)$, $\rp(P)$, $\lh(P)$, $\rh(P)$]
Suppose $P$ is a puzzle of type $BD$ or $RD$.
\begin{itemize}
\item The {\it path of $P$}, denoted by $\p(P)$, is  
the set of all puzzles that can be reached by repeatedly application of $\mr$ or $\ml$ to $P$. 
\item The unique $BD$--puzzle in $\p(P)$ with  the excess on the left boundary 
of ${\mathcal T}_N$ is said to be the {\it left puzzle of $P$}, denoted by $\lp(P)$; the {\it right puzzle $\rp(P)$ of $P$} is defined analogously.
\item The {\it left height of $P$}, denoted by $\lh(P)$, is the height of the center of the excess in $\lp(P)$ where the height is $h$ if the center of the excess lies on the $(h+1)$-st horizontal line of ${\mathcal T}_N$, counted from the bottom; the {\it right height $\rh(P)$} is defined analogously.
\end{itemize}
\end{defi}

%  \begin{figure}[!ht]
%   \begin{center}
%   \includegraphics[width=0.7\textwidth]{Migration}  
%  \caption{Migration.\label{fig:migration}}
%   \end{center}
%  \end{figure}

We also need the following counting functions.

\begin{defi}[$\# B(P)$, $\# BB(P)$, $\# BR(P)$, $\# R(P)$, $\# RR(P)$, $\# RB(P)$]
Suppose  $P$ is a puzzle of type $BD$ or $RD$ and  $X \in \{B, BB, BR, R, RR, RB\}$. We let $\# X(P)$ denote the number of moves of type $X$ that are necessary to transform $\lp(P)$ into 
$\rp(P)$.
\end{defi}

Observe that obviously 
$$
|\p(P)| = \# B(P) + \# BB (P) + \# BR (P) + \# R(P) + \# RR(P) + \# RB(P) + 1.$$
The following two propositions will allow us to compute the number of puzzles in the path of a given puzzle.

\begin{prop} 
\label{prop:equiv_puzzles_1}
For each puzzle $P$ of size $N$ and type $BD$ or $RD$ we have the following four identities.
\begin{enumerate}
\item\label{prop:equiv_puzzles_11} $\# BR(P) = \# RB(P) + 1$
\item\label{prop:equiv_puzzles_12} $\# R(P)- \# B(P)=\rh(P)-\lh(P)$
\item\label{prop:equiv_puzzles_13} $\# BB(P)+\# BR(P)+\# R(P)+\# RR(P)+\# RB(P)=N-\lh(P)$
\item\label{prop:equiv_puzzles_14}  $\# B(P)+\# BB(P)+\# BR(P)+\# RR(P)+ \# RB(P)=N-\rh(P)$
\end{enumerate}
\end{prop}

\begin{proof}
The first identity is obvious as $\lp(P)$ is of type $BD$ and $\rp(P)$ is of type $RD$, and $BR, RB$ are the only moves that change 
the type of a puzzle.

The second identity follows from the fact that move $B$ decreases the height of the excess by $1$, 
move $R$ increases this height by $1$, while all other moves  have no
 effect on the height.
 
As for the third identity observe that $BB, BR, R, RR, RB$ are precisely the moves that shift the center of the excess from one $\backslash$-diagonal of the grid ${\mathcal T}_N$ to the next $\backslash$-diagonal on the right, while the center of the excess stays on the same diagonal if we apply move $B$.  The identity follows, since an excess on the left boundary of ${\mathcal T}_N$
at height $\lh(P)$ lies on the $\lh(P)$-th $\backslash$-diagonal of ${\mathcal T}_N$ if counted from the left whereas the right boundary of the grid is the $N$-th $\backslash$-diagonal. 

The fourth identity follows from the third by symmetry.
\end{proof}

\begin{prop} 
\label{prop:equiv_puzzles_2}
For each puzzle $P$ of size $N$ and type $BD$ or $RD$ we have the following two identities.
\begin{enumerate}
\item\label{prop:equiv_puzzles_21} $\# BB(P) + \# RB(P) = \text{$\#$ of $1$s among the first $(N-\rh(P))$ letters of $v$}$
\item\label{prop:equiv_puzzles_22} $\# RR(P) + \# RB(P) = \text{$\#$ of $0$s among the last $(N-\lh(P))$ letters of $u$}$
\end{enumerate} 
\end{prop}

\begin{proof} We consider the first identity; the second follows by symmetry. Here it is convenient to argue in terms of  blue-red path tangles; we advise the reader to look at Figure~\ref{fig:jump_over}. In order to move the excess in $\lp(P)$ (where it is a blue down step) from the left boundary of 
${\mathcal T}_N$ to the right boundary (which results in  $\rp(P)$ where the excess is a red down step), the excess has to ``jump over'' 
a number of red paths. These are precisely the red paths that end above the excess in $\rp(P)$.  Since 
endpoints of red paths are encoded by $1$s, the number of these paths is given by the right-hand side in the first identity.

For the left-hand side, observe that in the process of applying $\mr$ repeatedly to $\lp(P)$, there are essentially two possibilities how an excess can overcome a red path. The first option is a move of type $BB$. The second option is that an excess of type $BD$ is transformed via $BR$ into an excess of type  $RD$, i.e. the excess moves from a blue path to a red path. Then it may stay on the red path for while (moves of type $R$ and $RR$ can be applied) until it jumps back to 
a blue path via a move of type $RB$. The first types of jumps are counted by $\# BB(P)$ and the second types by $\# RB(P)$. Also note that it is impossible for an excess to jump back (i.e. from the region below a red path to the region above a red path) by applying moves in $\{B, BB, BR, R, RR, RB\}$. 
\end{proof}

\begin{figure}[!ht]
 \begin{center}
 \includegraphics[width=0.4\textwidth]{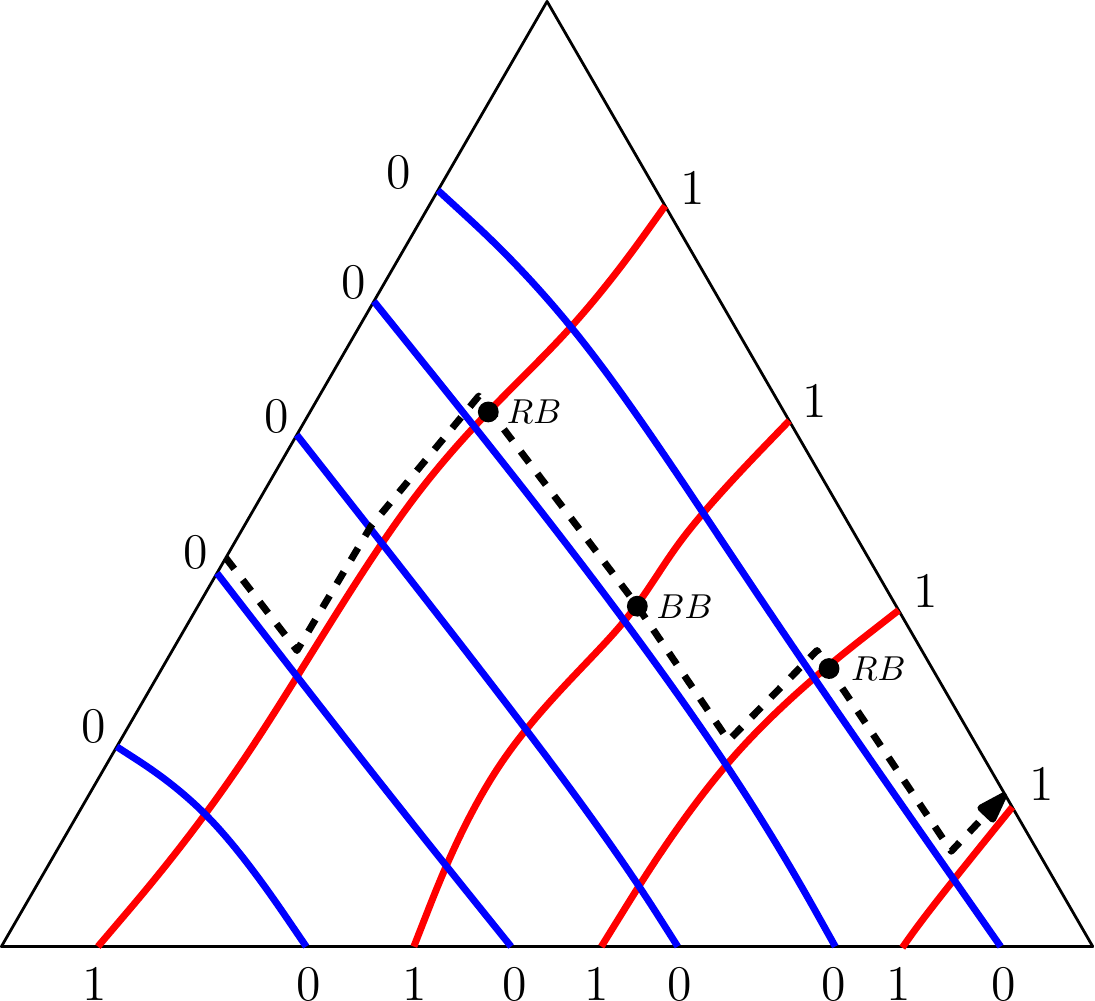}  
\caption{ Path created by moves in the tangle model: Moves $RB$ and $BB$ correspond to ``jumping over'' a red path.
%Situation in Proposition~\ref{prop:equiv_puzzles_2}.
\label{fig:jump_over}}
 \end{center}
\end{figure}

\subsection{Enumeration of oriented TFPLs of type $BD$ and type $RD$}
\label{sub:enum_BD_RD}

\begin{defi} We say that a word $u^+$ {\it covers} $u$, 
denoted by $u \to u^+$, if there exist $u_L,u_R$ such that $u=u_L01u_R$ and $u^+=u_L10u_R$. Morerover, for $i=0,1$, we define $R_i(u,u^+)=|u_R|_i$ and $L_i(u,u^+)=|u_L|_i$, and 
set $R(u,u^+)=R_0(u,u^+)+R_1(u,u^+)+1$ and $L(u,u^+)=L_0(u,u^+)+L_1(u,u^+)+1$.
\end{defi} 

Observe that, in the language of Ferrers diagrams, $u^+$ covers $u$ if and only if $\lambda(u)$ is contained in $\lambda(u^+)$ and the skew shape 
$\lambda(u^+) / \lambda(u)$ consists of a unique cell. The following lemma is essential.

\begin{lem} 
\label{lem:bij_excess1}
Let $u,v,w$ be words of length $N$ with excess $1$. Under each of
the following four sets of restrictions there is a bijection between puzzles $P$ of type $BD$ or $RD$ with boundary $(u,v;w)$ and 
pairs $(Q,i)$ of (ordinary) Knutson--Tao puzzles $Q$ and integers $i$. 
\begin{enumerate} 
\item One of the moves in $\{BB, BR, R\}$ can be applied to $P$; the boundary of $Q$ is $(u^+,v;w)$ where $u^+$ covers $u$ and 
$i \in \{0,1,\ldots,R_1(u,u^+)\}$. The number of these objects is 
\begin{equation}
\label{id:BB_BR_R}
\sum_{u^+:u \to u^+} \left( R_1(u,u^+) + 1 \right) c_{u^{+},v}^{w} .
\end{equation}
\item One of the moves in $\{RB, RR\}$ can be applied to $P$; the boundary of $Q$ is $(u^+,v;w)$ where $u^+$ covers $u$ and $i \in \{1,2,\ldots,R_0(u,u^+)\}$. The number of these objects is  
\begin{equation}
\label{id:RB_RR}
\sum_{u^+:u \to u^+} R_0(u,u^+) c_{u^{+},v}^{w} .
\end{equation}
\item One of the moves in $\{B,RB,RR\}$ can be applied to $P$; the boundary of $Q$ is $(u,v^+;w)$ where $v^+$ covers $v$ and 
$i \in \{1,\ldots,L_0(v,v^+)\}$.  The number of these objects is 
\begin{equation}
\label{id:B_RB_RR}
\sum_{v^+:v \to v^+}  L_0(v,v^+)  c_{u,v^{+}}^{w}.
\end{equation}
\item One of the moves in $\{BB,BR\}$ can  be applied to $P$; the boundary of  $Q$ is $(u,v^{+};w)$ where $v^+$ covers $v$ and 
$i \in \{0,1,\ldots,L_1(v,v^+)\}$. The number of these objects is
\begin{equation}
\label{id:BB_BR}
\sum_{v^+:v \to v^+} \left( L_1(v,v^+) + 1 \right) c_{u,v^{+}}^{w}.
\end{equation}
\end{enumerate}
\end{lem}

\begin{proof} 
Suppose $P$ is any puzzle with boundary $u, v, w$ of type $BD$ or $RD$.
Then the puzzles $\lp(P)$ and $\rp(P)$ can obviously be transformed into ordinary Knutson--Tao puzzles by removing the outer labeling of the exceptional puzzle pieces.
In the following, we identify $\lp(P)$ and $\rp(P)$ with these Knutson--Tao puzzles. The boundary of $\lp(P)$ is given by $(u^+(P),v;w)$, where $u^+(P)$ is obtained from $u$ by switching 
the letters $0$ and $1$ that are in position $\lh(P)$ and $\lh(P)+1$ of $u$. Similarly, $(u,v^+(P);w)$ is the 
boundary of $\rp(P)$, where $v^+(P)$ is obtained from $v$ by switching the letters $0$ and $1$ that are in 
position $N - \rh(P)$ and $N- \rh(P)+1$.

Now we concentrate on the first bijection; the procedure is similar for the others.
Subtract Proposition~\ref{prop:equiv_puzzles_2}\eqref{prop:equiv_puzzles_22} from Proposition~\ref{prop:equiv_puzzles_1}\eqref{prop:equiv_puzzles_13} to obtain
\begin{equation}
\label{id:bij_excess1}
\# BB(P) + \# BR(P) + \# R(P) = \text{$\#$ of $1$s among the last $(N-\lh(P))$ letters of $u$}.
\end{equation}
Using the notation introduced above, the right-hand side can also be written as $R_1(u,u^+(P))+1$. 
Now the bijection is a follows: suppose $P$ is a puzzle as described and set $Q=\lp(P)$. Equation~\eqref{id:bij_excess1} shows that $\lp$ maps 
precisely  $R_1(u,u^+(P))+1$ puzzles to this particular $Q$. The integer $i$ is the pointer to the position of $P$ in $\p(P)$
among all options that are mapped to $Q$.  

The fundamental identity for the second bijection is Proposition~\ref{prop:equiv_puzzles_2}\eqref{prop:equiv_puzzles_22}; note that the right-hand side 
can be replaced by $R_0(u,u^+(P))$.
 
For the third bijection, we subtract Proposition~\ref{prop:equiv_puzzles_2}\eqref{prop:equiv_puzzles_21} from Proposition~\ref{prop:equiv_puzzles_1}\eqref{prop:equiv_puzzles_14} and use Proposition~\ref{prop:equiv_puzzles_1}\eqref{prop:equiv_puzzles_11} to replace $\# BR(P)$. We obtain
$$
\# B(P) + \# RB(P) + \# RR(P) = (\text{$\#$ of $0$s among the first $(N-\rh(P))$ letters of $v$})-1.
$$
The right-hand side is equal to $L_0(v,v^+(P))$. 

For the fourth bijection note that the right-hand side of Proposition~\ref{prop:equiv_puzzles_2}\eqref{prop:equiv_puzzles_21}  is equal to 
$L_1(v,v^+(P)) $, and use Proposition~\ref{prop:equiv_puzzles_1}\eqref{prop:equiv_puzzles_11} to replace $\# BR(P)$.
\end{proof}

In the following lemma we provide an identity for Littlewood-Richardson coefficients that will be helpful in simplifying our 
formulas.

\begin{lem} 
\label{lem:simplification}
Let $u,v,w$ be words of with excess $\operatorname{exc}(u,v;w)=1$. Then
$$
\sum_{u^+:u \to u^+} c_{u^+,v}^{w} = 
\sum_{v^+:v \to v^+} c_{u,v^+}^{w}.
$$
\end{lem}

\begin{proof} 
This follows, for instance, by introducing an excess of type $BD$ on the left boundary of the puzzle and moving it to the right boundary of the puzzle. This has also a simple algebraic proof, obtained by writing the trivial identity $(s_{\lambda(u)}s_\square) s_{\lambda(v)}=s_{\lambda(u)}(s_\square s_{\lambda(v)})$ in terms of Littlewood--Richardson coefficients.
\end{proof}

We are finally able to enumerate puzzles of type $BD$ and type $RD$. Some more refined enumerations are also possible and will be helpful to 
deal with puzzles of type $DHD$ and $DHU$.

\begin{theo}
\label{theo:BD_RD}
Let $u,v,w$ be words of length $N$ with excess $1$.
\begin{enumerate}
\item The number of $BD$--puzzles with boundary $(u,v;w)$  to which move $B$ can be applied is equal to 
$$
\sum_{v^+:v \to v^+}  L_0(v,v^+)  c_{u,v^{+}}^{w} - \sum_{u^+:u \to u^+} R_0(u,u^+) c_{u^{+},v}^{w}.
$$
\item The number of $BD$--puzzles with boundary $(u,v;w)$ is equal to 
$$
\sum_{v^+:v \to v^+}  L(v,v^+)  c_{u,v^{+}}^{w}  -
\sum_{u^+:u \to u^+} R_0(u,u^+) c_{u^{+},v}^{w} .
$$
\item The number of $RD$--puzzles with boundary $(u,v;w)$ to which move $R$ can be applied is equal to 
$$
\sum_{u^+:u \to u^+}  R_1(u,u^+)  c_{u^{+},v}^{w}  - 
\sum_{v^+:v \to v^+}  L_1(v,v^+)  c_{u,v^{+}}^{w}.
$$
\item The number of $RD$--puzzles with boundary $(u,v;w)$ is equal to 
$$
\sum_{u^+:u \to u^+} R(u,u^+)  c_{u^{+},v}^{w}   -
\sum_{v^+:v \to v^+}  L_1(v,v^+) c_{u,v^{+}}^{w}.
$$
\item The number of puzzles of type $BD$ or $RD$ with boundary $(u,v;w)$ is equal to 
$$
\sum_{u^+:u \to u^+} \left( R_1(u,u^+)+1 \right)  c_{u^{+},v}^{w}   +
\sum_{v^+:v \to v^+}  \left( L_0(v,v^+) +1 \right) c_{u,v^{+}}^{w}.
$$
\end{enumerate}
\end{theo}
\begin{proof}
In order to obtain the number of puzzles to which move $B$ can be applied, one has to subtract 
\eqref{id:RB_RR} from~\eqref{id:B_RB_RR}. This has to be added to~\eqref{id:BB_BR} to obtain 
the total number of $BD$--puzzles.

For the number of puzzles to which move $R$ can be applied, we have to subtract~\eqref{id:BB_BR} from~\eqref{id:BB_BR_R}, 
and apply Lemma~\ref{lem:simplification}. To obtain the number of $RD$-puzzles, 
we add this to the sum of~\eqref{id:RB_RR} and  
$\sum\limits_{v^+:v \to v^+}  c_{u,v^{+}}^{w}=\sum\limits_{u^+:u \to u^+} c_{u^+,v}^{w}$. 
The last expression accounts for the $RD$--puzzles with the excess on the right boundary. 

As for the total number of puzzles of type $BD$ or $RD$ add the formulas in (2)
and (4), and use Lemma~\ref{lem:simplification}.
\end{proof}

\subsection{Enumeration of oriented TFPLs of type $DHD$ and type $DHU$}
\label{sub:DHD_DHU} %NEW
By Theorem~\ref{theo:BD_RD}, we now know the total number of puzzles with the excess piece of type $BD$ or $RD$ and the boundary $(u,v;w)$. To obtain a formula for $\overrightarrow{t}_{u,v}^w$ we still need to count puzzles of type $DHD$ and $DHU$. Thanks to the rules of Figure~\ref{fig:puzzle_moves_excess_1}, we know that $DHD$-puzzles (\emph{resp.} $DHU$-puzzles) are in bijection with $BD$-puzzles where move $BR$ can be applied (\emph{resp.} $RD$-puzzles where move $RB$ can be applied). The enumeration of such puzzles was not done in the previous sections, and requires in fact a new idea which exploits the symmetry of Knutson-Tao puzzles (we also give a second proof at the end of this section).

We need the following auxiliary objects, which are simply rotated $BD$-puzzles.

\begin{defi}[$gd$--puzzles] 
A \emph{$gd$--puzzle} of size $N$ is a labeling of  ${\mathcal T}_N$ such that 
\begin{enumerate}
\item there is precisely one  pair of
adjacent horizontal edges (the {\it excess}) labeled as indicated in Figure~\ref{fig:Puzzle_Moves_DH}, e.g. move $g$,
\item the labeling of each triangle 
can be found in  Figure~\ref{fig:puzzle_pieces_excess_0_labeled}, and, 
\item whenever two triangles are adjacent, their common edge has the same label in both triangles with the exception of the pair of adjacent horizontal edges that was selected in (1).
\end{enumerate}
\end{defi}

\begin{figure}[!ht]
 \begin{center}
 \includegraphics[width=0.6\textwidth]{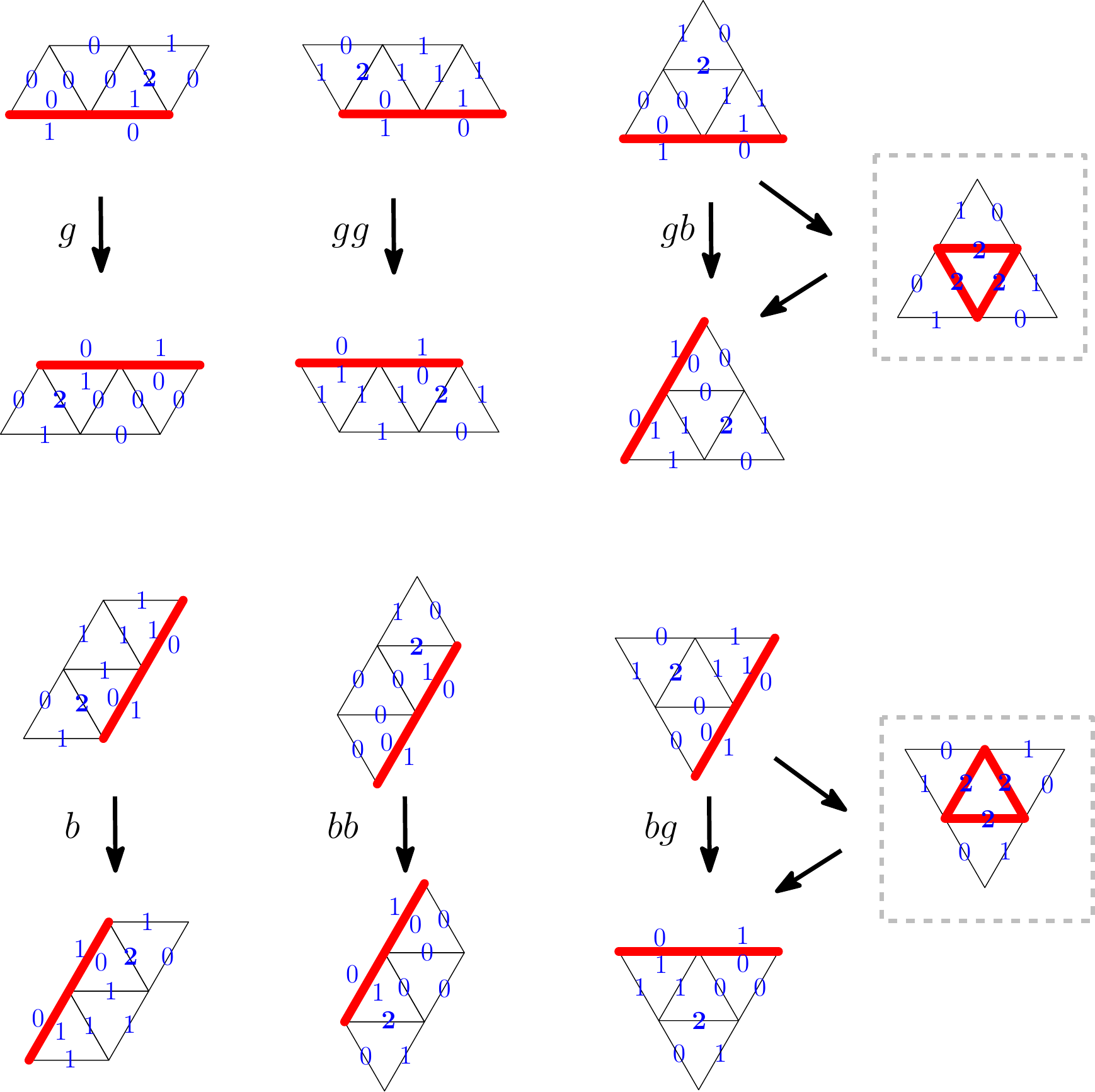}  
\caption{\label{fig:Puzzle_Moves_DH}}
 \end{center}
\end{figure}

We consider puzzles of type $BD$ or type $gd$ and the moves displayed in Figure~\ref{fig:Puzzle_Moves_DH}. In fact these puzzles are (by rotation of $120^\circ$) equivalent to puzzles of type 
$BD$ or type $RD$, see Figure~\ref{fig:Puzzle_Rotation}. The dictionary is as follows. For the types we have 
$$
(BD,RD) \leftrightarrow (gd, BD),
$$
for the moves we have 
$$
(B,BB,BR,R,RR,RB) \leftrightarrow (g,gg,gb,b,bb,bg)
$$
and for the boundary we have 
$$
(u,v;w) \leftrightarrow (\overleftarrow{z},x;\overleftarrow{y})
$$
where $(u,v;w)$ is the boundary of the puzzle of type $BD$ or $RD$, and 
$(x,y;z)$ is the boundary of the puzzle of type $BD$ or $gd$, 
and $\overleftarrow{z}$ is obtained from $z$ by reading it from right to 
left.
In the corollary below, we use this  correspondence to translate the third identity of Theorem~\ref{theo:BD_RD} into this setting.
In its proof, we use the following extension of Lemma~\ref{lem:simplification}.

\begin{lem} 
\label{lem:simplification_extended}
Let $u,v,w$ be words of excess $0$. Then
$$
\sum_{u^+:u \to u^+} c_{u^+,v}^{w} = 
\sum_{v^+:v \to v^+} c_{u,v^+}^{w}=
\sum_{w^-:w^- \to w} c_{u,v}^{w^-}
$$
\end{lem}

\begin{proof} 
The equivalence of the first and the third expression follows from introducing an excess of type $BD$ on the left boundary and shifting it to the bottom boundary using the moves in Figure~\ref{fig:Puzzle_Moves_DH}.
\end{proof}

\begin{figure}[!ht]
 \begin{center}
 \includegraphics[height=5cm]{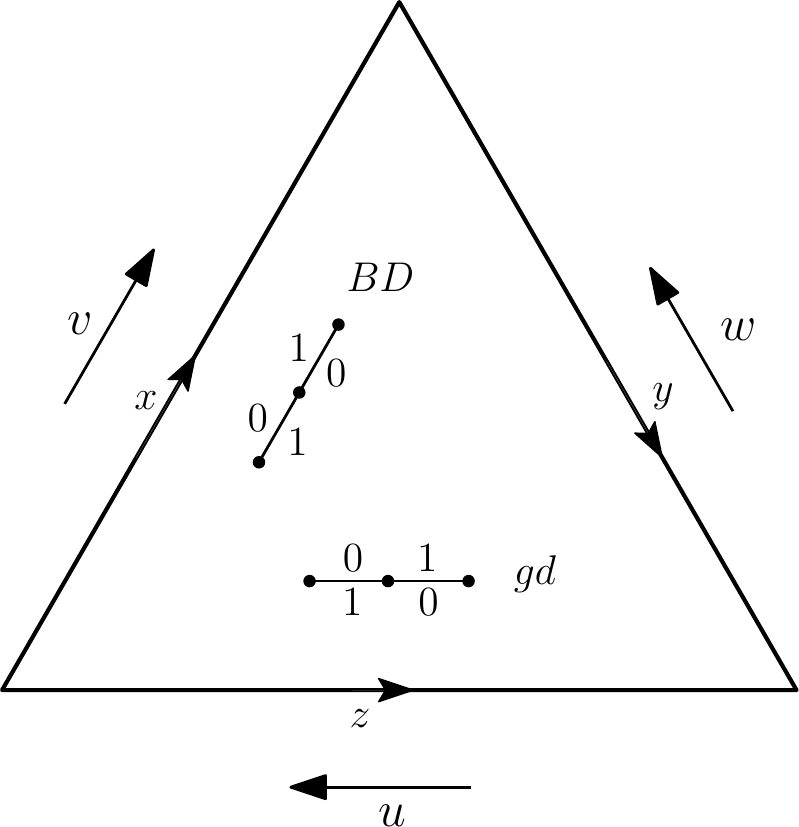}  
\caption{Correspondence between puzzles of type $BD$ or $gd$, and puzzles of type $BD$ or $rd$. \label{fig:Puzzle_Rotation}}
 \end{center}
\end{figure}

\begin{cor} The number of $BD$--puzzles with boundary $(x,y;z)$  to which move $b$ can be applied is equal to 
$$
\sum_{z^-:z^- \to z}  (L_1(z^-,z)+1)  c_{x,y}^{z^-} - \sum_{x^+:x \to x^+} (L_1(x,x^+)+1) c_{x^+,y}^{z}.
$$
\end{cor}

\begin{proof} By the correspondence, these puzzles are equivalent to $RD$-puzzles with boundary 
$(\overleftarrow{z},x;\overleftarrow{y})$ to which move $R$ can be applied.
The third identity in Theorem~\ref{theo:BD_RD} implies that the number in 
question is 
$$
\sum_{z^+:\overleftarrow{z} \to z^+}  R_1(\overleftarrow{z},z^+)  c_{z^+,x}^{\overleftarrow{y}} - \sum_{x^+:x \to x^+} L_1(x,x^+) c_{\overleftarrow{z},x^+}^{\overleftarrow{y}}.
$$
We use the transformation $z^-:= \overleftarrow{(z^+)}$. As $a \to b$ is equivalent to $\overleftarrow{b} \to \overleftarrow{a}$ and $R_1(\overleftarrow{z},z^+)=L_1(\overleftarrow{(z^+)},z)$, 
this is equal to 
$$
\sum_{z^-:z^- \to z}  L_1(z^-,z) c_{\overleftarrow{(z^-)},x}^{\overleftarrow{y}}   - \sum_{x^+:x \to x^+} L_1(x,x^+) c_{\overleftarrow{z},x^+}^{\overleftarrow{y}}.
$$
We apply the known identities $c_{p,q}^{r}=c_{q,p}^{r}$ and $c_{p,\overleftarrow{r}}^{\overleftarrow{q}}=c_{p,q}^{r}$  for Littlewood-Richardson coefficients 
to obtain the expression in the statement of the corollary. \end{proof}

\begin{theo} 
\label{theo:excess_1}
Let $u,v,w$ be words of length $N$ with excess $1$. 
\begin{enumerate}
\item The number of $DHD$--puzzles with boundary $(u,v;w)$ is 
$$
\sum_{u^+:u \to u^+} L_1(u,u^+) c_{u^+,v}^{w} +
\sum_{v^+:v \to v^+} \left( L_1(v,v^+) + 1 \right) c_{u,v^{+}}^{w}  
- \sum_{w^-:w^- \to w} L_1(w^-,w)  c_{u,v}^{w^-}.
$$
\item The number of $DHU$--puzzles with boundary $(u,v;w)$ is 
$$
 \sum_{u^+:u \to u^+} L_1(u,u^+) c_{u^+,v}^{w}  + 
\sum_{v^+:v \to v^+}  L_1(v,v^+)  c_{u,v^{+}}^{w}   - \sum_{w^-:w^- \to w} L_1(w^-,w)  c_{u,v}^{w^-}.
$$
\item The number of oriented TFPL configurations with boundary $(u,v;w)$ is 
$$
 \sum_{u^+:u \to u^+} (|u|_1 + L_1(u,u^+) ) c_{u^+,v}^{w} +
\sum_{v^+:v \to v^+}  (L(v,v^+)+L_1(v,v^+)+1)  c_{u,v^{+}}^{w}  
 - 2 \sum_{w^-:w^- \to w}  L_1(w^-,w)  c_{u,v}^{w^-}.
$$
\item The weighted enumerated of oriented TFPL configurations of excess $1$ and with boundary $(u,v;w)$ is 
\begin{align*}
& \sum_{u^+:u \to u^+}  \left( R_1(u,u^+)  + (q+q^{-1}) L_1(u,u^+)    + 1  \right)
c_{u^{+},v}^{w}   
\\ &
+ 
\sum_{v^+:v \to v^+}  \left( L_0(v,v^+)  +  (q+q^{-1}) L_1(v,v^+)   + 1 + q  \right) c_{u,v^{+}}^{w}  \\
& -  \sum_{w^-:w^- \to w} (q+q^{-1}) L_1(w^-,w)   c_{u,v}^{w^-}.
\end{align*}
\end{enumerate}
\end{theo}

\begin{proof} 
Since move $b$ is equal to move $BB^{-1}$, the expression in the corollary is the number of $BD$--puzzles to which 
move $BB$ can be applied. If we subtract it from ~\eqref{id:BB_BR}, we obtain the number of 
$BD$--puzzles to which move $BR$ can be applied. However, this is also the number of $DHD$--puzzles, 
see Figure~\ref{fig:puzzle_moves_excess_1}.

For the second assertion, observe that it is a direct consequence of Proposition~\ref{prop:equiv_puzzles_2}\eqref{prop:equiv_puzzles_21} that the number of puzzles of type $BD$ or $RD$ to which one of the moves in $\{BB,RB\}$ can  be applied is equal to
$$
\sum_{v^+:v \to v^+}  L_1(v,v^+) c_{u,v^{+}}^{w}.
$$
This implies together with the corollary that the number of $RD$--puzzles with boundary $u,v,w$ to which move $RB$ can be applied is equal to the expression displayed in the lemma.
But this is also the number of $DHU$--puzzles, see Figure~\ref{fig:puzzle_moves_excess_1}.

The third formula follows from adding the first two to the fifth formula in Theorem~\ref{theo:BD_RD}. The last formula follows similarly by taking Proposition~\ref{prop:weight_excess1} into account.
\end{proof}

%NEW
We end this section by sketching another proof of Theorem~\ref{theo:excess_1}. For this, one notices that puzzles where any number of $DHU$-pieces are allowed already appeared in the literature: they were introduced to compute certain coefficients of the $K$-theory of the Grassmannian (see~\cite{Knutson-talk} for instance). On the other hand, Grothendieck polynomials $G_\lambda$ (see~\cite{Lasc,LascSchu}), indexed by partitions, were introduced to give a purely algebraic way to compute the coefficients of this $K$-theory.
 Putting things together in our case where there is exactly one $DHU$ piece, it follows that the number of $DHU$-puzzles with boundary $(u,v;w)$ is the opposite of the coefficient of $G_{\lambda(w)}$ in the product of $G_{\lambda(u)}$ and $G_{\lambda(v)}$; here we assume $\d(w)=\d(u)+\d(v)+1$. This coefficient can be computed for instance with the help of formulas from~\cite{Lenart_Combinatorial_K_theory} which express Grothendieck polynomials in terms of Schur functions and conversely (we write $G_u,s_u$ for $G_{\lambda(u)},s_{\lambda(u)}$ respectively):

\begin{align*}
 G_{u}&=s_{u}-\sum_{u^+:u \to u^+} R_0(u,u^+) s_{u^+}+\text{sum of $s_{u'}$ with $\d(u')\geq \d(u)+2$}\\
 s_{w}&=G_{w}+\sum_{w^+:w \to w^+} R_0(w,w^+) G_{w^+}+\text{sum of $G_{w'}$ with $\d(w')\geq \d(w)+2$}
 \end{align*}

 From this one gets easily Formula (2) in Theorem ~\ref{theo:excess_1} above, and the remaining formulas follow easily.

\subsection{From oriented TFPLs of excess $1$ to ordinary TFPLs of excess $1$}
\label{sub:final_result}

\begin{theo} 
\label{theo:excess_1o}
Let $u,v,w$ be words of excess $1$. The number of TFPLs with boundary $(u,v;w)$ is 
$$ t_{u,v}^{w}=\sum_{v^+:v \to v^+}  \left( |v|_1 + L(v,v^+)  + 1 \right) c_{u,v^{+}}^{w}  
-  \sum_{w^-:w^- \to w}  L_1(w^-,w)   c_{u,v}^{w^-}.
$$ 
\end{theo}

\begin{proof}
In the case of excess $1$,~\eqref{nonorienttoorient} simplifies to 
$$
\overrightarrow{t}_{u,v}^{w}(q) = \overline{t}_{u,v}^{w}(q) + \sum_{w^{-}:w^{-} \to w} q \overline{t}_{u,v}^{w^{-}} (q).
$$
This is because $\overline{t}_{u,v}^{w'} (q)=0$ by Corollary~\ref{cor:inequality} if $w'$ is feasible for $w$ but neither 
$w'=w$ nor $w' \to w$.
We have $\overline{t}_{u,v}^{w^{-}} (q)=c_{u,v}^w$ by Proposition~\ref{prop:struct_tfpls_exc0},
since $\exc(u,v;w^-)=0$. Therefore, 
$$
\overline{t}_{u,v}^{w}(q) = \overrightarrow{t}_{u,v}^{w}(q) - \sum_{w^{-}:w^{-} \to w} q \, c_{u,v}^{w^{-}}.
$$
Now Theorem~\ref{theo:excess_1}, Proposition~\ref{prop:rho}  and 
Lemma~\ref{lem:simplification_extended}
imply the result.
\end{proof}

\bibliographystyle{alpha}
\bibliography{Biblio_TFPL3}

\end{document}